\documentclass[11pt]{article}

\usepackage{geometry,fullpage}
\usepackage{titling}
\usepackage{bm}
\usepackage{amsthm,amssymb,amsmath,amsfonts,rotate}
\usepackage{graphicx,algorithmic,algorithm}
\usepackage[]{datetime}
\usepackage{aliascnt}

\everymath{\color{black}}

\usepackage[pdftex,
backref=page,
plainpages = true,
pdfpagelabels,
hyperfootnotes=true,
pdfpagemode=FullScreen,
bookmarks=true,
bookmarksopen = true,
bookmarksnumbered = true,
breaklinks = true,
hyperfigures,
pagebackref,
urlcolor = magenta,
urlcolor = MidnightBlack,
anchorcolor = green,
hyperindex = true,
colorlinks = true,
linkcolor = black!30!blue,
citecolor = black!30!green
]{hyperref}

\RequirePackage[T1]{fontenc}
\RequirePackage[utf8]{inputenc}
\RequirePackage{amssymb}
\usepackage{alphabeta}
\usepackage{textcomp}
\DeclareUnicodeCharacter{2192}{\ifmmode\to\else\textrightarrow\fi}
\DeclareUnicodeCharacter{2203}{\ensuremath\exists}
\DeclareUnicodeCharacter{183}{\cdot}
\DeclareUnicodeCharacter{2200}{\forall}
\DeclareUnicodeCharacter{2264}{\leq}
\DeclareUnicodeCharacter{2265}{\geq}
\DeclareUnicodeCharacter{8614}{\mathbin{\mapsto}}
\DeclareUnicodeCharacter{8656}{\Leftarrow}
\DeclareUnicodeCharacter{8657}{\Uparrow}
\DeclareUnicodeCharacter{8658}{\Rightarrow}
\DeclareUnicodeCharacter{8659}{\Downarrow}
\DeclareUnicodeCharacter{8669}{\rightsquigarrow}
\newcommand{\eqdef}{\stackrel{{\scriptsize\rm def}}{=}}
\DeclareUnicodeCharacter{8797}{\eqdef}
\DeclareUnicodeCharacter{8870}{\vdash}
\DeclareUnicodeCharacter{8873}{\Vdash}
\DeclareUnicodeCharacter{8871}{\models}
\DeclareUnicodeCharacter{9121}{\lceil}
\DeclareUnicodeCharacter{9123}{\lfloor}
\DeclareUnicodeCharacter{9124}{\rceil}
\DeclareUnicodeCharacter{9126}{\rfloor}
\DeclareUnicodeCharacter{9655}{\triangleright}
\DeclareUnicodeCharacter{9665}{\triangleleft}
\DeclareUnicodeCharacter{9671}{\diamond}
\DeclareUnicodeCharacter{9675}{\circ}
\DeclareUnicodeCharacter{10178}{\bot}
\DeclareUnicodeCharacter{10214}{\llbracket} 
\DeclareUnicodeCharacter{10215}{\rrbracket} 
\DeclareUnicodeCharacter{10229}{\longleftarrow}
\DeclareUnicodeCharacter{10230}{\longrightarrow}
\DeclareUnicodeCharacter{10231}{\longleftrightarrow}
\DeclareUnicodeCharacter{10232}{\Longleftarrow}
\DeclareUnicodeCharacter{10233}{\Longrightarrow}
\DeclareUnicodeCharacter{10234}{\Longleftrightarrow}
\DeclareUnicodeCharacter{10236}{\longmapsto}
\DeclareUnicodeCharacter{10238}{\Longmapsto} 
\DeclareUnicodeCharacter{10503}{\Mapsto}    
\DeclareUnicodeCharacter{10971}{\mathrel{\not\hspace{-0.2em}\cap}}
\DeclareUnicodeCharacter{65294}{\ldotp}
\DeclareUnicodeCharacter{65372}{\mid}

\usepackage{lmodern}
\usepackage{amsmath,ifthen,eurosym}
\usepackage{wrapfig,cite}
\usepackage{latexsym,amssymb,url}
\usepackage{subcaption}
\usepackage{cleveref}

\usepackage[svgnames]{xcolor}

\usepackage{color}

\definecolor{MidnightBlack}{rgb}{0.1,0.1,0.30}
\definecolor{MidnightBlue}{rgb}{0.1,0.1,0.44}
\definecolor{Black}{rgb}{0,0, 0}
\definecolor{Blue}{rgb}{0, 0 ,1}
\definecolor{Red}{rgb}{1, 0 ,0}
\definecolor{White}{rgb}{1, 1, 1}
\definecolor{Grey}{rgb}{.6, .6, .6}
\definecolor{Mygreen}{rgb}{.0, .7, .0}
\definecolor{Yellow}{rgb}{.55,.55,0}
\definecolor{Mustard}{rgb}{1.0, 0.86, 0.35}
\definecolor{applegreen}{rgb}{0.55, 0.71, 0.0}
\definecolor{darkturquoise}{rgb}{0.0, 0.81, 0.82}
\definecolor{celestialblue}{rgb}{0.29, 0.59, 0.82}
\definecolor{green_yellow}{rgb}{0.68, 1.0, 0.18}
\definecolor{crimsonglory}{rgb}{0.75, 0.0, 0.2}
\definecolor{darkmagenta}{rgb}{0.30, 0.0, 0.30}
\definecolor{internationalorange}{rgb}{1.0, 0.31, 0.0}
\definecolor{darkorange}{rgb}{1.0, 0.55, 0.0}

\newcommand{\blue}[1]{{\color{Blue}#1}}
\newcommand{\red}[1]{{\color{Red}#1}}
\newcommand{\green}[1]{{\color{Mygreen}#1}}

\newcommand{\orange}[1]{{\color{darkorange}#1}}

\usepackage{tikz}
\usetikzlibrary{calc,3d}
\usetikzlibrary{intersections,decorations.pathmorphing,shapes,decorations.pathreplacing,fit}

\usepackage{todonotes}

\usepackage{amsmath,latexsym,amsthm,amsfonts,amssymb,%
	 ifthen,color,wrapfig,hyperref,rotate,lmodern,aliascnt,%
	 datetime,graphicx,enumerate,enumitem,%
	todonotes,cleveref,bm,tabularx,colortbl,xspace%
}

\newcommand{\remove}[1]{}

\newcommand{\bigmid}{\;\big|\;}
\newcommand{\cupall}{\pmb{\pmb{\bigcup}}}

\newcommand{\pretp}{\preceq_\mathsf{tm}}
\newcommand{\prem}{\preceq_\mathsf{m}}

\newcounter{func}
\newcommand{\newfun}[1]{f_{\refstepcounter{func}\label{#1}\thefunc}}
\newcommand{\funref}[1]{\hyperref[#1]{f_{\ref*{#1}}}} 

\newcounter{con}

\newcommand{\conref}[1]{\hyperref[#1]{c_{\ref*{#1}}}} 

\newcommand{\bound}[1]{\textbf{#1}}

\usepackage{dsfont}

\usepackage{float}
\usepackage{tikz,pgf}
\usetikzlibrary{backgrounds}
\usetikzlibrary{calc,3d}

\tikzset{red node/.style={draw=red, circle, fill = red, minimum size = 3pt, inner sep = 0pt}}
\tikzset{yellow node/.style={draw=yellow, circle, fill = yellow, minimum size = 4pt, inner sep = 0pt}}
\tikzset{blue node/.style={draw=celestialblue, circle, fill =celestialblue, minimum size = 3pt, inner sep = 0pt}}
\tikzset{green node/.style={draw=Mygreen, circle, fill =Mygreen, minimum size = 3pt, inner sep = 0pt}}
\tikzset{triangle/.style = { regular polygon, regular polygon sides=3, rotate=180}}
\tikzset{small red/.style={draw=red, triangle, fill = red, minimum size = 2pt, inner sep = 0pt}}

\tikzset{black node/.style={draw, circle, fill = black, minimum size = 3pt, inner sep = 0pt}}
\tikzset{small black node/.style={draw, circle, fill = black, minimum size = 3pt, inner sep = 0pt}}
\tikzset{model node/.style={draw=celestialblue, circle, fill = celestialblue, minimum size = 5pt, inner sep = 0pt}}
\tikzset{model node small/.style={draw=celestialblue, circle, fill = celestialblue, minimum size = 3pt, inner sep = 0pt}}
\tikzset{rep node/.style={draw=red, circle, fill = red, minimum size = 3pt, inner sep = 0pt}}
\tikzset{track node 1/.style={draw, circle, fill = black, minimum size = 2pt, inner sep = 0pt}}
\tikzset{track node 2/.style={draw=black!30!white, circle, fill = black!30!white, minimum size = 2pt, inner sep = 0pt}}
\tikzset{track node 3/.style={draw=black!10!white, circle, fill = black!10!white, minimum size = 2pt, inner sep = 0pt}}

\newcommand{\mynewtheorem}[2]{
	\newaliascnt{#1}{dummy}
	\newtheorem{#1}[#1]{#2}
	\aliascntresetthe{#1}
}

\theoremstyle{plain}
\mynewtheorem{theorem}{Theorem}
\mynewtheorem{proposition}{Proposition}
\mynewtheorem{corollary}{Corollary}
\mynewtheorem{lemma}{Lemma}

\theoremstyle{definition}
\mynewtheorem{definition}{Definition}
\theoremstyle{remark}
\mynewtheorem{sublemma}{Sublemma}
\mynewtheorem{claim}{Claim}
\mynewtheorem{remark}{Remark}
\mynewtheorem{fact}{Fact}
\mynewtheorem{conjecture}{Conjecture}
\mynewtheorem{question}{Question}
\mynewtheorem{observation}{Observation}
\mynewtheorem{problem}{Problem}
\mynewtheorem{note}{Note}

\usepackage{titling}
\thanksmarkseries{arabic}
\newcommand*\samethanks[1][\value{footnote}]{\footnotemark[#1]}

\newtheorem{innercustomgeneric}{\customgenericname}
\providecommand{\customgenericname}{}

\begin{document}

\title{$k$-apices of minor-closed graph classes.\! I. Bounding the obstructions\thanks{Some of the results of this paper appeared in the \emph{Proceedings of the 47th International Colloquium on Automata, Languages and Programming (\textbf{ICALP}), volume 168 of LIPICs, pages 95:1--95:20, \textbf{2020}}.}}

\author{\bigskip Ignasi Sau\thanks{LIRMM, Univ Montpellier, CNRS, Montpellier, France. {Supported}  by   the ANR projects DEMOGRAPH (ANR-16-CE40-0028), ESIGMA (ANR-17-CE23-0010), ELIT (ANR-20-CE48-0008), and the French-German Collaboration ANR/DFG Project UTMA (ANR-20-CE92-0027).  Emails:  \texttt{ignasi.sau@lirmm.fr}, \texttt{giannos.stamoulis@lirmm.fr}, \texttt{sedthilk@thilikos.info}.}\and
	Giannos Stamoulis\samethanks[2]
	\and
	Dimitrios  M. Thilikos\samethanks[2]}
\date{}

\maketitle

\begin{abstract}
\noindent Let $\mathcal{G}$ be a minor-closed graph class.
We say that a graph $G$ is a \emph{$k$-apex} of $\mathcal{G}$ if  $G$ contains a  set $S$ of at most $k$ vertices such that  $G\setminus S$ belongs to $\mathcal{G}.$ We denote by $\mathcal{A}_k (\mathcal{G})$ the set of all graphs that are $k$-apices of $\mathcal{G}.$
We prove that every graph in the obstruction set of $\mathcal{A}_k (\mathcal{G}),$  i.e., the minor-minimal set of graphs not belonging to $\mathcal{A}_k (\mathcal{G}),$  has order
at most $2^{2^{2^{2^{\mathsf{poly}(k)}}}},$ where $\mathsf{poly}$ is a polynomial function whose degree depends on the order of the minor-obstructions of $\mathcal{G}.$
This bound drops to $2^{2^{\mathsf{poly}(k)}}$ when $\mathcal{G}$ excludes some apex graph as a minor.
\bigskip

\end{abstract}

\newpage

\tableofcontents

\newpage

\section{Introduction}
A graph class $\mathcal{G}$ is
\emph{minor-closed} if every minor\footnote{A graph $H$ is a \emph{minor} of $G$ if a graph isomorphic to $H$
	can be obtained from some subgraph of $G$ after applying edge contractions. As in this paper we consider only simple graphs, we insist
	that in case multiple edges are created after a contraction, then these edges are automatically suppressed to simple edges, while in the case that loops are created, they are automatically removed from the graph.}
of a graph in $\mathcal{G}$
is also a member of $\mathcal{G}.$
Given a
graph class $\mathcal{G},$  its \emph{minor obstruction set}, denoted by $\mathbf{obs}(\mathcal{G}),$  is defined as the set of all minor-minimal graphs not in $\mathcal{G},$ called
	\emph{minor obstructions of} $\mathcal{G}.$
Given a set of graphs $\mathcal{F},$ we denote by $\mathbf{excl}(\mathcal{F})$ as the set containing every graph $G$ that excludes all graphs in $\mathcal{F}$ as minors.
Clearly, for every minor-closed graph class $\mathcal{G},$ $\mathcal{G}=\mathbf{excl}(\mathbf{obs}(\mathcal{G})).$ This implies
that  the obstruction set  $\mathbf{obs}(\mathcal{G})$ can be seen
as a complete characterization of $\mathcal{G}$ in terms of excluded minors.

An algorithmic consequence of the above
concerns the  \textsc{Membership in $\mathcal{G}$} problem that asks,
given an $n$-vertex graph $G,$ whether $G\in \mathcal{G}.$
It follows that if $\mathcal{G}$ is minor-closed then
the \textsc{Membership in $\mathcal{G}$} problem
is equivalent to checking whether $G$ excludes as a minor
all the graphs in the  set $\mathbf{obs}(\mathcal{G}),$ and therefore
is reduced to the \textsc{Minor Checking} problem that asks, given two graphs $G$
and $H,$  whether $H$ is a minor of $G.$

\subsection{Obstruction sets}
According to the celebrated Robertson and Seymour's theorem~\cite{RobertsonSXX} \textsl{there is no infinite set of graphs where every pair of graphs is non-comparable by the minor relation}.
This result implies that,  for every graph class $\mathcal{G},$ the set $\mathbf{obs}(\mathcal{G})$ is \textsl{finite}.
Moreover, the seminal algorithmic result
of the Graph Minors series
is an algorithm solving the \textsc{Minor Checking} problem  in  $f(h)\cdot n^3$-time, where $h$ is the order of $H$ and $f$ is some function of {$h$}. This
algorithm has been improved to a quadratic one in~\cite{KawarabayashiKR11thed} and this, along with the finiteness
of $\mathbf{obs}(\mathcal{G}),$ implies that, for every  minor-closed graph class $\mathcal{G},$
there \textsl{exists} a quadratic-time algorithm for the \textsc{Membership in $\mathcal{G}$} problem. However, this does not mean that we can actually \textsl{construct}
such an algorithm, as this requires first to construct the set  $\mathbf{obs}(\mathcal{G}).$

Interestingly, the proof of Robertson and Seymour's theorem is not constructive.
Friedman, Robertson, and Seymour proved in~\cite{FriedmanRS87them} that the bounded\footnote{The ``bounded version'' of Robertson and Seymour's theorem is the one where the graphs in its statement are restricted to have bounded treewidth.} version of the
Robertson and Seymour's theorem
is equivalent to the extended Kruskal's theorem which  is proof-theoretically stronger than $\Pi^1_1$-CA$_{0}$ (see the work of Krombholz and Rathjen~\cite{krombholz2019upper} for recent  results on the  meta-mathematics of
Robertson and Seymour's theorem~\cite{FriedmanRS87them}).
This rules out the existence of
a proof yielding a way to construct $\mathbf{obs}(\mathcal{G}).$
Moreover, in the same impossibility direction, Fellows and Langston~\cite{FellowsL88nonc}
proved, using a reduction from the \textsc{Halting} problem, that  there is no algorithm that
given a finite description of a minor-closed class $\mathcal{G},$
outputs $\mathbf{obs}(\mathcal{G}).$
Additional conditions, mostly related to logic, that can
guaranty the computability of obstruction sets
have been extensively investigated in \cite{FellowsL94onse,LagergrenA91mini,Lagergren91anup,Lagergren98,GuptaI97boun,AbrahamsonF93,AdlerGK08comp,CourcelleDF97,Bienstock95,FellowsL88nonc,ArnborgPS90tree,DowneyF95survey,GuptaI97boun}.

\medskip

The  study of $\mathbf{obs}(\mathcal{G})$ for distinct instantiations of minor-closed graph classes $\mathcal{G}$
is an active topic in graph theory. In the best
of the cases, such results
achieve the \textsl{complete identification} of
the obstruction set \cite{LeivaditisSSTTV20mino,Holst02onth,RobertsonST95sach,BodlaenderThil99,Thilikos00,ArnborgPC90forb,Archdeacon06akur,DinneenX02mino,KinnersleyL94,FioriniHJV17thee,DinneenV12obst,Wagner37uber}
or resort to a partial characterization, by identifying
subsets of them with particular properties \cite{Pierce14PhDThesis,Yu06more,LiptonMMPRT16sixv,JobsonK21allm,GagarinMC09theb,MoharS12obst}.
A more general  line of research is to study
	\textsl{parameterized} minor-closed graph classes.
A parameterized minor-closed graph class
is a collection $\{\mathcal{G}_{i}\mid i\geq 0\},$ where
$\mathcal{G}_i$ is typically defined
as the set of graphs for which  the value of some minor-closed graph parameter\footnote{A \emph{graph parameter} is a function
	mapping graphs to non-negative integers and is \emph{minor-closed} if it cannot increase when taking minors.} is bounded by $i.$
To our knowledge, no known result exactly identifies $\mathbf{obs}(\mathcal{G}_i),$ for all $i\in\mathbb{N}.$
However, there are results that either identify
all graphs in  $\mathbf{obs}(\mathcal{G}_i)$ with some particular property, implying –typically huge– lower bounds on the cardinality of $|\mathbf{obs}(\mathcal{G}_i)|,$ as a function of $i$ \cite{Dinneen97,MattmanP16thea,RueST12oute,KoutsonasTY14oute,BienstockD92ono,Ramachandramurthi97thes,DinerGST21block}, or provide upper bounds on the order of the graphs in $\mathbf{obs}(\mathcal{G}_i)$ as a function of $i$ \cite{FellowsL89anan,NishimuraRT05para,ChatzidimitriouTZ20spar,Lagergren98,GuptaI97boun,FominLMS12plan,DinneenL07prop,Seymour94abou}. These {latter} results are interesting as they
yield the computability of $\mathbf{obs}(\mathcal{G}_i)$ and directly  imply the \textsl{constructibility}
of fixed-parameter tractable algorithms for
the corresponding graph parameters.
For  more on the interrelation between obstruction sets and parameterized algorithms, see~\cite{FellowsJ13fpti,FellowsL88nonc,FellowsL94onse}.
We wish to stress that the references that we give above
are indicative and by no means complete. See also~\cite{Mattman16forb,Adler08open} for  related surveys.

\subsection{Apices of minor-closed classes}

Given a non-negative integer $k$ and a graph class $\mathcal{G},$ we say that
a graph  $G$ is a \emph{$k$-apex of} $\mathcal{G}$ if it can be transformed to a member of $\mathcal{G}$ after removing at most $k$  vertices.
We denote the set of all $k$-apices of $\mathcal{G}$ by  $\mathcal{A}_{k}(\mathcal{G}).$
The study of
$k$-apices of graph classes is quite extensive both in combinatorics and algorithms. The (meta)problem
\textsc{Vertex Deletion to $\mathcal{G}$} asking, given a graph $G$  and an integer $k,$ whether $G$ is a $k$-apex of $\mathcal{G},$
is  part  of the wider family of \textsl{Graph Modification Problems}
and
can be seen as the prototypical setting of the
``\textsl{small distance from triviality}'' question~\cite{GHN04,FominSM15grap}.

It is easy to see that if $\mathcal{G}$ is minor-closed, then $\mathcal{A}_{k}(\mathcal{G})$ is also minor-closed for every $k\geq 0.$ Clearly, $\mathbf{obs}(\mathcal{A}_{k}(\mathcal{G}))$  constitutes a complete characterization of the class of the $k$-apices of $\mathcal{G}$ and in many cases it characterizes several known graph parameters. For instance, graphs with a vertex cover of size at most $k$ are the graphs in
$\mathcal{A}_{k}(\mathbf{excl}(\{K_{2}\})),$ graphs
with a  feedback vertex set of size at most $k$ are
the graphs in $\mathcal{A}_{k}(\mathbf{excl}(\{K_{3}\})),$ and $k$-apex planar graphs  (also known as \textsl{ apex graphs})
are the graphs in $\mathcal{A}_{k}(\mathbf{excl}(\{K_{5},K_{3,3}\})).$\medskip

Given a finite collection of graphs $\mathcal{F}$
and a non-negative integer $k,$ we set $$\mathcal{F}^{(k)}=\mathbf{obs}(\mathcal{A}_{k}({\mathbf{excl}}(\mathcal{F}))).$$
Notice that if $\mathcal{G}$ is a minor-closed graph class and $\mathcal{F}=\mathbf{obs}(\mathcal{G}),$ then
$\mathcal{F}^{(k)}=\mathbf{obs}(\mathcal{A}_{k}(\mathcal{G})).$

Adler, Grohe, and Kreutzer made an important  step in~\cite{AdlerGK08comp} (see also~\cite{FellowsL94onse}) on the algorithmic study of $\mathcal{F}^{(k)}$
by proving  that it is effectively computable: there is a
Turing Machine that receives $\mathcal{F}$ and $k$
as input and, after \textsl{some finite number of steps}, outputs the set $\mathcal{F}^{(k)}.$ Here we need to stress that no  bound for the running time function of such a Turing Machine is given in~\cite{AdlerGK08comp}. This can be overcome  by a proof of an explicit  combinatorial bound on the size of
$\mathcal{F}^{(k)}.$

Up to now, the most general combinatorial bound on $\mathcal{F}^{(k)}$ is given in~\cite{FominLMS12plan},
where it is proven that if $\mathcal{F}$ contains some \textsl{planar} graph, then every graph in $\mathcal{F}^{(k)}$
has $\mathcal{O}(k^{h})$ vertices, where $h$ is some constant depending (non-constructively) on $\mathcal{F}$ (see \cite{NishimuraRT05para,DinneenL07prop,ZorosPhD17} for  low polynomial bounds on  special cases
of this result). \smallskip

Apart from the above general results, a lot
of work   has been devoted to the identification
of $\mathcal{F}^{(k)}$  for particular instantiations of $\mathcal{F}$ and $k.$ In this direction,  $\{K_{2}\}^{(k)}$ has been identified for $k\in\{1,\ldots,5\}$ in~\cite{CattellDDFL00onco}, for $k=6$ in~\cite{DinneenX02mino}, and for $k=7$ in~\cite{DinneenV12obst}, while the graphs in $\{K_{3}\}^{(i)}$ have been identified in~\cite{DinneenCF01forb} for $i\in\{1,2\}.$
In~\cite{Ding2016}, Ding and Dziobiak identified  the 57 graphs in $\{K_{4},K_{2,3}\}^{(1)},$ i.e., the obstruction
set for apex-outerplanar graphs, and the 25 graphs in $\{K_{4}^{-}\}^{(1)},$   i.e., the obstruction
set for apex-cactus graphs (as announced in~\cite{DziobiakD2013}).
Recently,  the 29 obstructions for 1-apex sub-unicyclic graphs and the 33 obstructions for 1-apex pseudoforests have been identified   in~\cite{LeivaditisSSTTV20mino}
and~\cite{LeivaditisSSTT2019mino}, respectively.

A landmark problem that  attracted
particular attention (see e.g.,~\cite{LiptonMMPRT16sixv,Mattman16forb,Yu06more}) is the one of identifying $\{K_{5},K_{3,3}\}^{(1)},$ i.e., characterizing 1-apex planar graphs.  In this direction, Mattman and Pierce conjectured that
$\{K_{5},K_{3,3}\}^{(n)}$ contains the $Y\Delta Y$-families of $K_{n+5}$ and $K_{3^2,2^n}$ and provided evidence towards this in~\cite{MattmanP16thea}.
Recently,  Jobson and Kézdy  identified \textsl{all} graphs  in $\{K_{5},K_{3,3}\}^{(1)}$ of {connectivity} two in \cite{JobsonK21allm}, where they also reported that $|\{K_{5},K_{3,3}\}^{(1)}|\geq 401.$

\subsection{Our bounds} In this paper we
provide the first general combinatorial upper bound
on the order of the graphs in $\mathcal{F}^{(k)},$ as a function of $\mathcal{F}$ and $k.$ To specify the bound and its attributes, we define two constants depending on $\mathcal{F}.$ We set $s_{\mathcal{F}}$ as the maximum number of vertices of a graph in $\mathcal{F}.$
We also define $a_{\mathcal{F}}$ as the minimum apex number of a graph in $\mathcal{F},$ where the \emph{apex number} of a graph $G$ is the minimum $i$ such that
$G$ is $i$-apex planar, i.e., $a_{\mathcal{F}}= \min\{i\mid\mathcal{F}\cap \mathcal{A}_{i}(\mathbf{excl}(\{K_{5},K_{3,3}\})) \neq \emptyset\}.$
In~\autoref{label_administrando}, we define a third constant depending on $\mathcal{F}$, namely $\ell_\mathcal{F}$, that is the maximum detail of the graphs in $\mathcal{F}$, where the \emph{detail} of a graph is the maximum is the maximum between the size of its vertex and its edge set.
The constant $\ell_\mathcal{F}$ is redundant for the presentation of the results and the proof outline in the rest of this section.

There are several graph classes where
$a_{\mathcal{F}}=1$ such as single-crossing minor-free graph classes \cite{RobertsonS93} and surface-embeddable graphs.
Finally, $a_{\mathcal{F}}=0$ if and only if  $\mathbf{excl}(\mathcal{F})$ has bounded treewidth, because of  the grid exclusion theorem~\cite{RobertsonS86GMV}.
\smallskip

The  main contribution of this paper is the following\footnote{In this paper we adopt the notation $f(\alpha,\beta)=\mathcal{O}_{\alpha}(\beta)$ (resp. $f(\alpha,\beta)=\Omega_{\alpha}(\beta)$) in order to denote that $f(\alpha,\beta)$ is upper-bounded (resp. lower-bounded) by the product of a function of $\alpha$ and a linear function of $β.$}.

\begin{theorem}\label{label_miseryfrightened}
There exists a function $\newfun{label_circunscriben}: \mathbb{N}^3\to\mathbb{N}$
such that if $\mathcal{F}$ is a finite collection of
graphs and $k\in\mathbb{N},$ then every graph
in $\mathcal{F}^{(k)}=\mathbf{obs}(\mathcal{A}_k (\mathbf{excl}(\mathcal{F})),$
has at most $\funref{label_circunscriben}(a_{\mathcal{F}},s_{\mathcal{F}},k)$ vertices.\medskip

\noindent Moreover, if $a=a_{\mathcal{F}}\geq 1$ and $s=s_{\mathcal{F}},$ then
$\funref{label_circunscriben}(a,s,k)=2^{2^{2^{ \log k \cdot 2^{\mathcal{O}_s (k^{a -1})}} }}.$
In particular,  $\funref{label_circunscriben}(a,s,k)= 2^{2^{q\cdot 2^{\log(c \cdot k) \cdot 2^{k^{a -1}\cdot 2^{{\mathcal{O}(s^2\log s)}}}}}},$
where $q:=q(s^2)=2^{2^{2^{c^{2^{2^{\mathcal{O}(s^2\log s)}}}}}},$ and $c:=f_{\mathsf{ul}} (s^2).$
\end{theorem}
\smallskip

\noindent In the above theorem. $f_{\mathsf{ul}}$ is the bounding function of the Unique Linkage Theorem from~\cite{KawarabayashiW10asho} (see also \cite{RobertsonS09XXI,RobertsonSGM22}).  We stress the function  $f_{\mathsf{ul}}$ is introduced twice in our proofs, namely in  \autoref{label_panlatinismo} and  \autoref{label_desenfadadamente}.\medskip

Notice that the general bound of \autoref{label_miseryfrightened} is 4-fold exponential in some polynomial of $k.$
It is worth to observe that
in the case where that  $\mathcal{F}$ contains some apex  graph (or equivalently, $a_{\mathcal{F}}=1$) this bound becomes   $2^{2^{q\cdot  k^{{c'}}}},$
where ${c'}:=2^{2^{\mathcal{O}(s^2 \log s)}},$ i.e.,  is double-exponential in some polynomial of $k,$ whose degree depends on $s.$
For the case where  $a_{\mathcal{F}}=0$ (that is, when $\mathcal{F}$ contains only planar graphs), as we already mentioned, there exists a better bound than the one of \autoref{label_miseryfrightened}, i.e., the obstructions in $\mathcal{F}^{(k)}$
have order polynomial in $k$~\cite{FominLMS12plan}, with degree depending on $s$ as well.
\medskip

\autoref{label_miseryfrightened} implies that one can construct an algorithm that receives as input $\mathcal{F}$ and $k$ and outputs the obstruction set $\mathcal{F}^{(k)}.$ This is done by enumerating all graphs of at most $\funref{label_circunscriben}(a,s,k)$
vertices and filtering out those that are not minor-minimal members of $\mathcal{A}_{k}({\mathbf{excl}}(\mathcal{F})).$ The running time  of this algorithm can be bounded by  $f'(k,s)=2^{2^{2^{2^{2^{k^{\mathcal{O}_{s}(1)}}}}}}.$

\subsection{Proof outline}
Our proof has two parts and, in both of them, the protagonist  is the graph parameter \textsl{treewidth}.
Treewidth is a cornerstone parameter in both structural and algorithmic graph theory and, roughly speaking, can be seen
as a measure of the topological
resemblance of a graph to the structure of a tree (see \autoref{label_dispensaries} for the formal definition).
Our first aim is to bound the treewidth of the graphs in  $\mathcal{F}^{(k)}$ by a function of $k$ and, in the second step,
we use this bound in order to bound the order of the graphs in $\mathcal{F}^{(k)}.$ This two-stage approach
is not new. It dates back to the celebrated \textsl{irrelevant vertex technique}, introduced in~\cite{RobertsonS95XIII}
for the design of a polynomial-time algorithm for the \textsc{Disjoint Paths} problem.
The same technique  has been  used  in~\cite{AdlerGK08comp}
for computing obstruction sets (see also~\cite[Section  7.9.1]{DowneyF99para} and~\cite{CattellDDFL00onco})
and   for the design
of parameterized algorithms recognizing $k$-apices of certain minor-closed graph classes \cite{MarxS07obta,KociumakaP19dele,JansenLS14anea,Kawarabayashi09plan,SauST20anfp,SauST21kapiII}.
In the rest of this subsection we outline how this technique is applied in order to obtain the bounds in~\autoref{label_miseryfrightened}.

\paragraph{The Flat Wall Theorem.}
The main combinatorial tool for our proof is the  Flat Wall Theorem.
This theorem was proved by Robertson and Seymour in~\cite{RobertsonS95XIII} and served as  the combinatorial base for the application (and also the invention) of the \textsl{irrelevant vertex technique}.
In a nutshell, this theorem asserts that every graph in $\mathbf{excl}(\mathcal{F})$
that has ``big enough'' treewidth  contains a vertex set $A,$ whose size depends
on $\mathcal{F},$ such that $G\setminus A$ contains some ``flat wall''.
Intuitively, a flat wall $W$ is contained in a larger subgraph of $G\setminus A,$ its \textsl{compass},
that  is separated from the rest of  $G\setminus A$  via a separator $S$ that is a ``suitably chosen'' part of  the
``bordering cycle'' of $W$ and is arranged in a ``flat way'' inside this cycle.

To deal with flat walls, we use the combinatorial framework
recently introduced in~\cite{SauST21amor} that, in turn, is based on the improved version of the Flat Wall Theorem
proved by
Kawarabayashi, Thomas, and Wollan in~\cite{KawarabayashiTW18anew} (see also~\cite{Chuzhoy15impr,GiannopoulouT13}).
This framework is presented in \autoref{label_dispensaries} and provides the formal definitions of a series of
combinatorial concepts such as paintings and renditions (\autoref{label_inappropriate}), flatness pairs and  tilts (\autoref{label_exceptionalness}), as well
as a notion of wall homogeneity (\autoref{label_backwardness})
alternative to the one given in~\cite{RobertsonS95XIII}.
All these concepts are extensively used in our proofs.

\paragraph{Bounding the treewidth of the obstructions.}
Given a graph $G,$
we call a set $S\subseteq V(G)$ \emph{$\mathcal{F}$-hitting set} of
$G$ if $G\setminus S\in\mathbf{excl}(\mathcal{F}).$

Our proof for bounding the treewidth of the graphs in $\mathcal{F}^{(k)}$
departs from the fact that
every obstruction $G\in\mathcal{F}^{(k)}$ contains an $\mathcal{F}$-hitting set $R$
with $k+1$ vertices. We assume, towards a contradiction, that $G$ has treewidth $\Omega_{s}(k^{2^{\Omega_s(k^{a -1})}}),$ where $s=s_{\mathcal{F}}$ and $a=a_{\mathcal{F}}.$

We next apply the Flat Wall Theorem (as stated in~\autoref{label_aldobrandesco}) on $G\setminus R$
and, after removing an additional set $A,$ we know that $G\setminus (R\cup A),$
contains a flat wall $W$
of height $k^{2^{\Omega_s(k^{a -1})}},$ while  $|R\cup A|\leq k+a-4.$
We then
consider the set $Q$ consisting of the vertices of $R\cup A$ that have
neighbors  in $\Omega_{s}(k^3)$  elements of  some so-called ``canonical partition'' of $G\setminus (R\cup A),$ introduced in~\autoref{label_villebrequin},
into ``zones of influence'' around the vertices of $W.$ Intuitively, this partition  is defined so to respect the ``bidimensional'' structure of $W.$

We next   identify inside the compass $K$ of $W$
some  flat wall $W'$ of height $\Omega_{s}(k)$ with compass $K'$ such that
\begin{itemize}
	\item[(i)]  none of the vertices in $(R\cup A)\setminus Q$ has any neighbor inside  $K',$ and
	\item[(ii)] $W'$ is ``homogeneous with respect to all subsets of $Q$ of size $a-1$''. Here we use the generalized notion of  homogeneity   that  is defined in~\cite{SauST21amor} and used in~\cite{BasteST20acom}.
\end{itemize}

The existence of a wall $W'$ as above  is supported by \autoref{label_anticipation}, \autoref{label_overestimation},
and finally  \autoref{lemma_bidim_branch},
whose proof occupies the whole~\autoref{label_nominalistic}.

We stress at this point that the price we pay for obtaining  a homogeneous  wall of height  $\Omega_{s}(k)$
was  to demand  that $W$ has height $k^{2^{\Omega_s(k^{a -1})}},$ where the  term $\Omega_s(k^{a -1})$ comes from the  number of subsets of $Q$ of size $a-1$ and the fact that, in the worst case, $Q=R\cup A.$
This is the source of the double-exponentiality in $k$ of the bound for treewidth (which becomes polynomial in $k$ in case $a=0$).

Let now $G^-=G\setminus v$ where $v$ is a ``central vertex'' of $W'.$
As $G\in{\mathcal{F}}^{(k)},$ $G^-$ has an $\mathcal{F}$-hitting set $S$ with $|S|\leq k.$
Our next step is to prove that any such $S$ must intersect all but at most $a-1$ vertices of $Q.$ Then using (i), (ii), and the main combinatorial
result of~\cite{BasteST20acom} (\autoref{label_panlatinismo}), we prove (\autoref{label_acquaintances})  that  $S$ is also an $\mathcal{F}$-hitting set of $G,$ a contradiction to the fact that $G\in{\mathcal{F}}^{(k)}$ (see the proof of \autoref{label_transfiguration}).

\paragraph{Bounding the order of the obstructions.}
Given that the graphs in $\mathcal{F}^{(k)}$ have treewidth bounded by a function that  is double-exponential in $k,$
the next step is to bound their order.
For this, we apply the technique
introduced by Lagergren in~\cite{Lagergren98} (see also~\cite{Lagergren91anup,LagergrenA91mini}).
This technique has been used in order to bound the order of obstructions for several width parameters:  minor obstructions for treewidth and pathwidth in~\cite{Lagergren98}, immersion obstructions for cutwidth in~\cite{GiannopoulouPRT19cutw} and   tree-cutwidth in~\cite{GiannopoulouKRT19lean}, and vertex-minor obstructions for linear rankwidth in~\cite{KanteKKO23obst,KanteK18line}. Also, similar ideas were used for the identification of immersion obstructions in~\cite{GiannopoulouPRT17line}. \smallskip

In our case, we consider in \autoref{label_agathyrsians} a graph $G\in\mathcal{F}^{(k)}$ and we  assume that its treewidth is upper-bounded by $t=t(k,s).$  Then we consider a special
type of  tree decomposition defined in~\cite{RobertsonS86GMV,Thomas90amen}, called \emph{linked decomposition}   (see also~\cite{BellenbaumD02twos,Erde18auni}). We also assume that this decomposition is  ``binary'',
in the sense that the associated tree $T$ rooted at each of its nodes has at most two children. We then consider, for each bag
$X_{i}$ of the decomposition, the graph
that is ``dangling'' below $X_{i}$ and
see it as a ``boundaried'' graph $\mathbf{G}_{i},$ whose boundary   is the set $X_{i}.$
An important part of the proof
is to assign to each
boundaried graph $\mathbf{G}_{i}$
a ``set of characteristics'', expressing all ways
$\mathcal{F}$-hitting sets of size at most $k$ may intersect partial minor models of the graphs in $\mathcal{F}$  inside $\mathbf{G}_{i}$ (see \autoref{label_profesionales}). These characteristics are defined using the algorithmic results of~\cite{BasteST20hittI,BasteST20acom}
and play a role similar to that of defining a ``finite congruence'' in~\cite{LagergrenA91mini}.
Using the combinatorial  bounds of~\cite{BasteST20acom}, we prove that they are no more than $2^{\mathcal{O}_{s}(t\log t)}$ different characteristics.
Also, in \autoref{label_profesionales} we introduce an ordering between such characteristics and show that if $X_{i}$ and $X_{j}$ are nodes in the same path from the root to some leaf of $T,$ then their corresponding characteristics are comparable with respect to this ordering (\autoref{label_substantiality}). Next we use the fact that $G$ is an obstruction to prove that
the characteristics should be properly ordered along such paths (see the proof of \autoref{label_substanceless}).
This implies that each path of $T$ has  $2^{\mathcal{O}_{s}(t\log t)}$ nodes which, in turn, yields that
the binary tree $T$ has $2^{2^{\mathcal{O}_{s}(t\log t)}}$ nodes in total. This bound on the size of the tree of the tree decomposition implies that the same bound holds for the order of $G$ as well. \medskip

Given the above discussion, we conclude that the order of the graphs in $\mathcal{F}^{(k)}$ is double-exponential in their treewidth and that their treewidth is double-exponential in $k.$ These altogether yield the claimed 4-fold exponential bound.

\section{Preliminaries}\label{label_dispensaries}

In this section we present some basic definitions
as well as the combinatorial framework of \cite{SauST21amor} concerning flat walls.
Namely, in \autoref{label_administrando} we give some basic definitions on graphs and
in \autoref{label_domesticated} we define walls and some basic notions concerning them.
In \autoref{label_inappropriate} we present the notions of paintings and renditions, which we use in \autoref{label_exceptionalness} in order to define flat walls, flatness pairs, and their tilts.
In \autoref{label_backwardness} we define the notion of homogeneous flatness pairs.

\subsection{Basic definitions}\label{label_administrando}

\paragraph{Sets and integers.}\label{label_unverallgemeinerten}
We denote by $\mathbb{N}$ the set of non-negative integers.
Given two integers $p$ and $q,$ the set $[p,q]$ refers to the set of every integer $r$ such that $p\leq r\leq q.$
For an integer $p\geq 1,$ we set $[p]=[1,p]$ and $\mathbb{N}_{\geq p}=\mathbb{N}\setminus [0,p-1].$
Given a non-negative integer $x,$
we denote by $\mathsf{odd}(x)$ the minimum odd number that is not smaller than $x.$
For a set $S,$ we denote by $2^{S}$ the set of all subsets of $S$ and, given an integer $r\in[|S|],$
we denote by $\binom{S}{r}$ the set of all subsets of $S$ of size $r$ and by $\binom{S}{\leq r}$ the set of all subsets of $S$ of size at most $r.$
If ${\mathcal{S}}$ is a collection of objects where the operation $\cup$ is defined,
then we denote $\cupall\mathcal{S}=\bigcup_{X\in\mathcal{S}}X.$
Given two sets $A,B$ and a function $f: A\to B,$
for a subset $X\subseteq A$ we use $f(X)$ to denote the set $\{f(x)\mid x\in X\}.$

\paragraph{Basic concepts on graphs.}\label{label_ricominciavan}
All graphs considered in this paper are undirected, finite, and without loops or multiple edges.
We use standard graph-theoretic notation and we refer the reader to
\cite{Diestel10grap} for any undefined terminology.
Let $G$ be a graph.
We call $|V(G)|$ the \emph{order} of $G$.
Also, we say that a pair $(L,R)\in 2^{V(G)}\times 2^{V(G)}$ is a \emph{separation} of $G$
if $L\neq R$, $L\cup R=V(G)$, and there is no edge in $G$ between $L\setminus R$ and $R\setminus L.$
Given a vertex $v\in V(G),$ we denote by $N_{G}(v)$ the set of vertices of $G$ that are adjacent to $v$ in $G.$
Also, given a set $S\subseteq V(G),$ we set $N_G (S) = \bigcup_{v\in S} N_G (v).$
For $S \subseteq V(G),$ we set $G[S]=(S,E\cap{\binom{S}{2}} )$ and use the shortcut $G \setminus S$ to denote $G[V(G) \setminus S].$
Given a graph $G,$ we define the \emph{detail} of $G,$ denoted by $\mathsf{detail}(G),$ to be the maximum among $|E(G)|$ and $|V(G)|.$
Given a finite collection $\mathcal{F}$ of graphs, we set $\ell_{\mathcal{F}}= \max\{\mathsf{detail}(H)\mid H\in\mathcal{F}\}.$
Given a vertex $v\in V(G)$ of degree two with neighbors $u$ and $w,$ we define the \emph{dissolution} of $v$
to be the operation of deleting $v$ and, if $u$ and $w$ are not adjacent, adding the edge $\{u,w\}.$
Given two graphs $H,G,$ we say that $H$ is a \emph{dissolution} of $G$
if a graph isomorphic to $H$ can be obtained from $G$ after dissolving vertices of $G.$
Given an edge $e=\{u,v\}\in E(G),$ we define the \emph{subdivision} of $e$
to be the operation of deleting $e,$ adding a new vertex $w$ and making it adjacent to $u$ and $v.$
Given two graphs $H,G,$ we say that $H$ is a \emph{subdivision} of $G$
if a graph isomorphic to $H$ can be obtained from $G$ after possibly subdividing edges of $G.$

\paragraph{Treewidth.}
A \emph{tree decomposition} of a graph~$G$
is a pair~$(T,\chi)$ where $T$ is a tree and $\chi: V(T)\to 2^{V(G)}$
such that
\begin{itemize}
	\item $\bigcup_{t \in V(T)} \chi(t) = V(G),$
	\item for every edge~$e$ of~$G$ there is a $t\in V(T)$ such that
	      $\chi(t)$
	      contains both endpoints of~$e,$ and
	\item for every~$v \in V(G),$ the subgraph of~${T}$
	      induced by $\{t \in V(T)\mid {v \in \chi(t)}\}$ is connected.
\end{itemize}
The \emph{width} of $(T,\chi)$ is equal to $\max\big\{\left|\chi(t)\right|-1 \bigmid t\in V(T)\big\}$ and the \emph{treewidth} of $G$ is the minimum width over all tree decompositions of $G.$

\paragraph{Contractions and minors.}\label{label_senzillament}
The \emph{contraction} of an edge $e = \{u,v\}$ of a simple graph $G$ results in a simple graph $G'$
obtained from $G \setminus \{u,v\}$ by adding a new vertex $uv$ adjacent to all the vertices
in the set $N_G(u) \cup N_G(v)\setminus \{u,v\}.$
A graph $H$ is a \emph{minor} of a graph $G,$ denoted by $H\prem G,$
if we can obtain from $G$ a graph $G'$
by a sequence of vertex removals, edge removals, and edge contractions
such that $H$ is isomorphic to $G'$.
If only edge contractions are allowed, we say that $H$ is a \emph{contraction} of $G.$
Given two graphs $H,G,$ if $H$ is a minor of $G$ then for every vertex $v\in V(H)$ there is
a set of vertices in $G$ that are the endpoints of the edges of $G$ contracted towards creating
the vertex $\rho(v)$ of $V(G')$, where $\rho$ is an isomorphism from $H$ to $G'$.
We call this set the \emph{model} of $v$ in $G.$
Given a finite collection of graphs $\mathcal{F}$ and a graph $G,$ we use the notation $\mathcal{F}\prem G$ to denote that some graph in $\mathcal{F}$ is a minor of $G.$

\subsection{Walls and subwalls}\label{label_domesticated}

\paragraph{Walls.}\label{label_mistreatment}
Let  $k,r\in\mathbb{N}.$ The
\emph{$(k\times r)$-grid} is the
graph whose vertex set is $[k]\times[r]$ and two vertices $(i,j)$ and $(i',j')$ are adjacent if and only if $|i-i'|+|j-j'|=1.$
In the rest of this paper, we always assume that each vertex $(i,j)\in[k]\times [r]$ of a $(k\times r)$-grid is embedded at the point $(i,j)$ in a coordinate system whose horizontal axis refers to the first coordinate, whose vertical axis refers to the second coordinate, and each edge of the grid is represented by a straight line segment.
An  \emph{elementary $r$-wall}, for some odd integer $r\geq 3,$ is the graph obtained from a
$(2 r\times r)$-grid
with vertices $(x,y)
	\in[2r]\times[r],$
after the removal of the
``vertical'' edges $\{(x,y),(x,y+1)\}$ for odd $x+y,$ and then the removal of
all vertices of degree one.
This definition is slightly different than other definitions in the literature
(i.e., we require $r$ to be odd), but we adopt this one for technical reasons.
Notice that, as $r\geq 3,$  an elementary $r$-wall is a planar graph
that has a unique (up to topological isomorphism) embedding in the plane $\mathbb{R}^{2}$
such that all its finite faces are incident to exactly six
edges.
The \emph{perimeter} of an elementary $r$-wall is the cycle bounding its infinite face,
while the cycles bounding its finite faces are called \emph{bricks}.
Also, the vertices
in the perimeter of an elementary $r$-wall that have degree two are called \emph{pegs},
while the vertices $(1,1), (2,r), (2r-1,1), (2r,r)$ are called \emph{corners} (notice that the corners are also pegs).

An \emph{$r$-wall} is any graph $W$ obtained from an elementary $r$-wall $\bar{W}$
after subdividing edges (see \autoref{label_determinable}). A graph $W$ is a \emph{wall} if it is an $r$-wall for some odd $r\geq 3$
and we refer to $r$ as the \emph{height} of $W.$ Given a graph $G,$
a \emph{wall of} $G$ is a subgraph of $G$ that is a wall.
We insist that, for every $r$-wall, the number $r$ is always odd.

We call the vertices of degree three of a wall $W$ \emph{3-branch vertices}.
A cycle of $W$ is a \emph{brick} (resp. the \emph{perimeter}) of $W$
if its 3-branch vertices are the vertices of a brick (resp. the perimeter) of $\bar{W}.$
We denote by ${\mathcal{C}}(W)$ the set of all cycles of $W.$
We  use $D(W)$ in order to denote the perimeter of the  wall $W.$
A brick of $W$ is \emph{internal} if it is disjoint from $D(W).$

\begin{figure}[ht]
	\begin{center}
		\includegraphics[width=11cm]{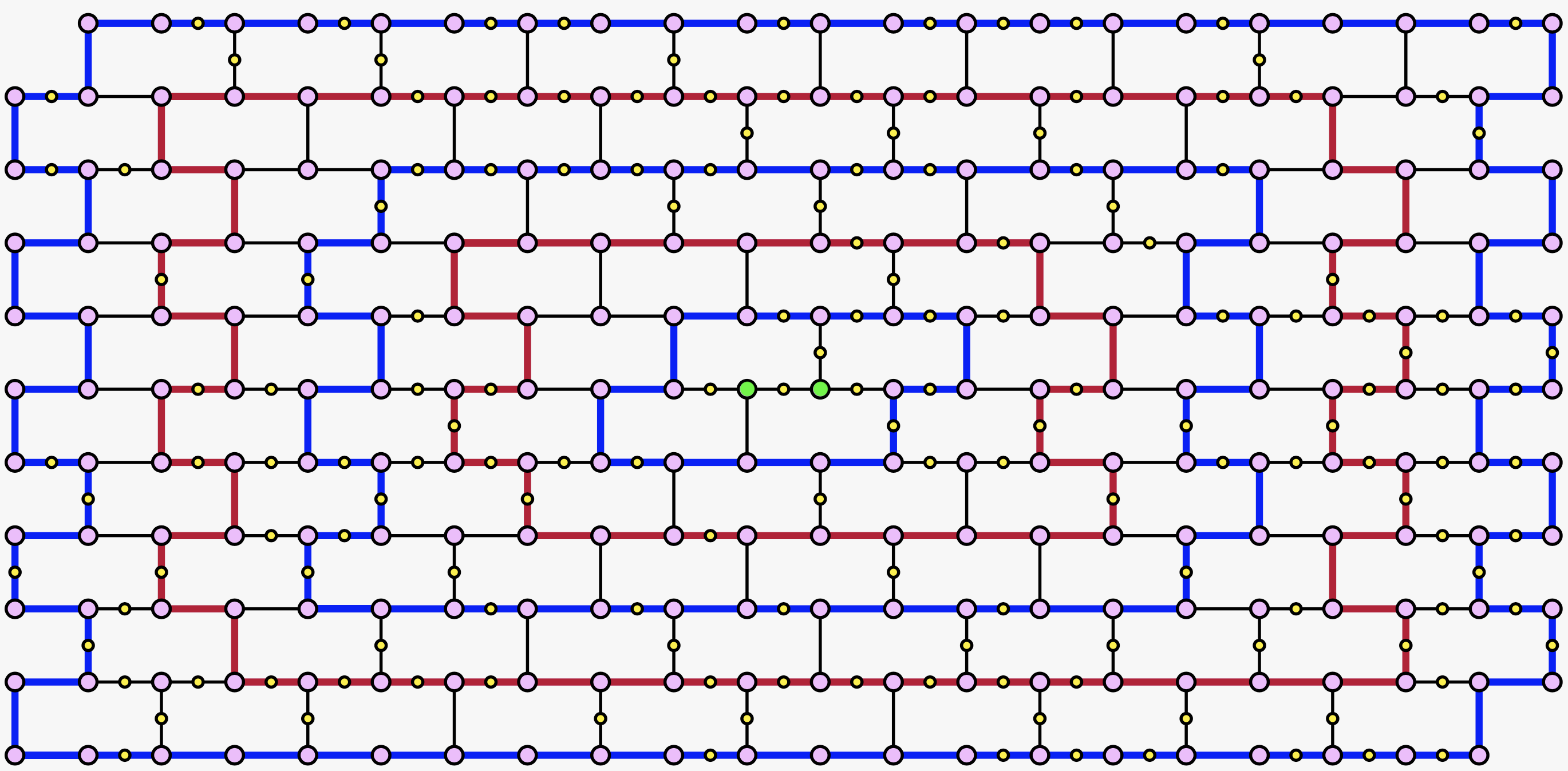}
	\end{center}
	\caption{An $11$-wall and its five layers, depicted in alternating red and blue. The central vertices of the wall are depicted in green.}
	\label{label_determinable}
\end{figure}

\paragraph{Subwalls.}
Given an elementary $r$-wall $\bar{W},$ some odd $i\in \{1,3,\ldots,2r-1\},$ and $i'=(i+1)/2,$
the \emph{$i'$-th  vertical path} of $\bar{W}$  is the one whose
vertices, in order of appearance, are $(i,1),(i,2),(i+1,2),(i+1,3),
	(i,3),(i,4),(i+1,4),(i+1,5),
	(i,5),\ldots,(i,r-2),(i,r-1),(i+1,r-1),(i+1,r).$
Also, given some $j\in[2,r-1]$ the \emph{$j$-th horizontal path} of $\bar{W}$
is the one whose
vertices, in order of appearance, are $(1,j),(2,j),\ldots,(2r,j).$

A \emph{vertical} (resp. \emph{horizontal}) path of $W$ is one
that is a subdivision of a  vertical (resp. horizontal) path of $\bar{W}.$
Notice that the perimeter of an $r$-wall $W$
is uniquely defined regardless of the choice of the elementary $r$-wall $\bar{W}.$
A \emph{subwall} of $W$ is any subgraph $W'$ of  $W$
that is an $r'$-wall, with $r' \leq r,$ and such the vertical (resp. horizontal) paths of $W'$ are subpaths of the
	{vertical} (resp. {horizontal}) paths of $W.$

\paragraph{Layers.}
The \emph{layers} of an $r$-wall $W$  are recursively defined as follows.
The first layer of $W$ is its perimeter.
For $i=2,\ldots,(r-1)/2,$ the $i$-th layer of $W$ is the $(i-1)$-th layer of the subwall $W'$
obtained from $W$ after removing from $W$ its perimeter and
removing recursively all consequent vertices of degree one.
We refer to the $(r-1)/2$-th layer as the \emph{inner layer} of $W.$
The \emph{central vertices} of an $r$-wall are its two branch vertices  that do not belong to any of its layers and that are connected by a path of $W$ that does not intersect any layer.
See \autoref{label_determinable} for an illustration of the notions defined above.

\paragraph{Central walls.}
Given an $r$-wall $W$ and an odd $q\in\mathbb{N}_{\geq 3}$ where $q\leq r,$
we define the \emph{central $q$-subwall} of $W,$ denoted by $W^{(q)},$
to be the $q$-wall obtained from $W$ after removing
its first $(r-q)/2$ layers and all consequent vertices of degree one.

\paragraph{Tilts.}
The \emph{interior} of a wall $W$ is the graph obtained
from $W$ if we remove from it all edges of $D(W)$ and all vertices
of $D(W)$ that have degree two in $W.$
Given two walls $W$ and $\tilde{W}$ of a graph $G,$
we say that $\tilde{W}$ is a \emph{tilt} of $W$ if $\tilde{W}$ and $W$ have identical interiors.

\subsection{Paintings and renditions}
\label{label_inappropriate}
In this subsection we present the notions of renditions and paintings, originating in the work of Robertson and Seymour \cite{RobertsonS95XIII}.
The definitions presented here were introduced by Kawarabayashi et al. \cite{KawarabayashiTW18anew} (see also~\cite{SauST21amor}).
\paragraph{Paintings.}
A \emph{closed} (resp. \emph{open}) \emph{disk} is a set homeomorphic to the set
$\{(x,y)\in \mathbb{R}^{2}\mid x^{2}+y^{2}\leq 1\}$ (resp. $\{(x,y)\in \mathbb{R}^{2}\mid x^{2}+y^{2}< 1\}$).
Let $\Delta$ be a closed disk.
Given a subset $X$ of $\Delta,$ we
denote its closure by $\bar{X}$ and its boundary by $\mathsf{bd}(X).$
A \emph{{$\Delta$}-painting} is a pair $\Gamma=(U,N)$
where
\begin{itemize}
	\item  $N$ is a finite set of points of $\Delta,$
	\item $N \subseteq U \subseteq \Delta,$ and
	\item $U \setminus  N$ has finitely many arcwise-connected  components, called \emph{cells}, where, for every cell $c,$
	      \begin{itemize}
		      \item[$\circ$] the closure $\bar{c}$ of $c$
		            is a closed disk
		            and
		      \item[$\circ$]  $|\tilde{c}|\leq 3,$ where $\tilde{c}:=\mathsf{bd}(c)\cap N.$
	      \end{itemize}
\end{itemize}
We use the  notation $U(\Gamma) := U,$
$N(\Gamma) := N$  and denote the set of cells of $\Gamma$
by $C(\Gamma).$
For convenience, we may assume that each cell  of $\Gamma$ is an open disk of $\Delta.$
Notice that, given a $\Delta$-painting $\Gamma,$
the pair $(N(\Gamma),\{\tilde{c}\mid c\in C(\Gamma)\})$  is a hypergraph whose hyperedges have cardinality at most three and  $\Gamma$ can be seen as a plane embedding of this hypergraph in $\Delta.$

\paragraph{Renditions.}
Let $G$ be a graph and let $\Omega$ be a cyclic permutation of a subset of $V(G)$ that we denote by $V(\Omega).$ By an \emph{$\Omega$-rendition} of $G$ we mean a triple $(\Gamma, \sigma, \pi),$ where
\begin{itemize}
	\item[(a)] $\Gamma$ is a $\Delta$-painting for some closed disk $\Delta,$
	\item[(b)] $\pi: N(\Gamma)\to V(G)$ is an injection, and
	\item[(c)] $\sigma$ assigns to each cell $c \in  C(\Gamma)$ a subgraph $\sigma(c)$ of $G,$ such that
	      \begin{enumerate}
		      \item[(1)] $G=\bigcup_{c\in C(\Gamma)}\sigma(c),$
		      \item[(2)]  for distinct $c, c' \in  C(\Gamma),$  $\sigma(c)$ and $\sigma(c')$  are edge-disjoint,
		      \item[(3)] for every cell $c \in  C(\Gamma),$ $\pi(\tilde{c}) \subseteq V (\sigma(c)),$
		      \item[(4)]  for every cell $c \in  C(\Gamma),$
		            $V(\sigma(c)) \cap \bigcup_{c' \in  C(\Gamma) \setminus  \{c\}}V(\sigma(c')) \subseteq \pi(\tilde{c}),$ and
		      \item[(5)]  $\pi(N(\Gamma)\cap \mathsf{bd}(\Delta))=V(\Omega),$ such that the points
		            in $N(\Gamma)\cap \mathsf{bd}(\Delta)$ appear in $\mathsf{bd}(\Delta)$ in the same ordering
		            as their images, via $\pi,$ in $\Omega.$
	      \end{enumerate}
\end{itemize}

\subsection{Flatness pairs}
\label{label_exceptionalness}
In this subsection we define the notion of a flat wall, originating in the work of Robertson and Seymour \cite{RobertsonS95XIII} and later used in \cite{KawarabayashiTW18anew}.
Here, we define flat walls as in \cite{SauST21amor}.

\paragraph{Flat walls.}
Let $G$ be a graph and let $W$ be an $r$-wall  of $G,$ for some odd integer $r\geq 3.$
We say that a pair $(P,C)\subseteq D(W)\times D(W)$ is a \emph{choice
		of pegs and corners for $W$} if $W$ is the subdivision of an  elementary $r$-wall $\bar{W}$
where $P$ and
$C$ are the pegs and the corners of $\bar{W},$ respectively (clearly, $C\subseteq P$).
To get more intuition, notice that a wall $W$ can occur in several ways from the elementary wall $\bar{W},$
depending on the way the edges in the perimeter of $\bar{W}$ are subdivided.
Each of them gives a different selection $(P,C)$ of pegs and corners of $W.$

We say that $W$ is a \emph{flat $r$-wall}
of $G$ if there is a separation $(X,Y)$ of $G$ and a choice  $(P,C)$
of pegs and corners for $W$ such that:
\begin{itemize}
	\item $V(W)\subseteq Y,$
	\item  $P\subseteq X\cap Y\subseteq V(D(W)),$ and
	\item  if $\Omega$ is the cyclic ordering of the vertices $X\cap Y$ as they appear in $D(W),$
	      then there exists an $\Omega$-rendition $(\Gamma,\sigma,\pi)$ of  $G[Y].$
\end{itemize}
We say that $W$ is a \emph{flat wall}
of $G$ if it is a flat $r$-wall for some odd integer $r \geq 3.$

\paragraph{Flatness pairs.}
Given the above, we  say that  the choice of the 7-tuple $\mathfrak{R}=(X,Y,P,C,\Gamma,\sigma,\pi)$
\emph{certifies that $W$ is a flat wall of $G$}.
We call the pair $(W,\mathfrak{R})$ a \emph{flatness pair} of $G$ and define
the \emph{height} of the pair $(W,\mathfrak{R})$ to be the height of $W.$
We use the term \emph{cell of} $\mathfrak{R}$ in order to refer to the cells of $\Gamma.$

We call the graph $G[Y]$ the \emph{$\mathfrak{R}$-compass} of $W$ in $G,$
denoted by $\mathsf{Compass}_{\mathfrak{R}}(W).$
It is easy to see that there is a connected component of $\mathsf{Compass}_{\mathfrak{R}}(W)$ that contains the wall $W$ as a subgraph.
We can assume that $\mathsf{Compass}_{\mathfrak{R}} (W)$ is connected, updating $\mathfrak{R}$ by removing from $Y$ the vertices of all the connected components of $\mathsf{Compass}_\mathfrak{R} (W)$
except of the one that contains $W$ and including them in $X$ ($\Gamma, \sigma, \pi$ can also be easily modified according to the removal of the aforementioned vertices from $Y$).
We define the  \emph{flaps} of the wall $W$ in $\mathfrak{R}$ as
$\mathsf{Flaps}_{\mathfrak{R}}(W):=\{\sigma(c)\mid c\in C(\Gamma)\}.$
Given a flap $F\in \mathsf{Flaps}_{\mathfrak{R}}(W),$ we define its \emph{base}
as $\partial F:=V(F)\cap \pi(N(\Gamma)).$
A  cell $c$ of ${\mathfrak{R}}$ is \emph{untidy} if  $\pi(\tilde{c})$ contains a vertex
$x$ of ${W}$ such that two of the edges of ${W}$ that are incident to $x$ are edges of $\sigma(c).$
Notice that if $c$ is untidy then  $|\tilde{c}|=3.$
A cell $c$ of $\mathfrak{R}$ is \emph{tidy} if it is not untidy.
The notion of tidy/untidy cell as well as the notions that
we present in the rest of this subsection have been introduced in~\cite{SauST21amor}.

\paragraph{Cell classification.}
Given a cycle $C$ of $\mathsf{Compass}_{\mathfrak{R}}(W),$ we say that
$C$ is \emph{$\mathfrak{R}$-normal} if it is \textsl{not} a subgraph of a flap $F\in \mathsf{Flaps}_{\mathfrak{R}}(W).$
Given an $\mathfrak{R}$-normal cycle $C$ of $\mathsf{Compass}_{\mathfrak{R}}(W),$
we call a cell $c$ of $\mathfrak{R}$ \emph{$C$-perimetric} if
$\sigma(c)$ contains some edge of $C.$
Since every $C$-perimetric cell $c$ contains some edge of $C$ and $|\partial\sigma(c)|\leq 3,$ we observe the following.
\begin{observation}\label{label_grundfalscher}
	For every pair $(C,C')$ of $\mathfrak{R}$-normal cycles of $\mathsf{Compass}_{\mathfrak{R}} (W)$ such that $V(C)\cap V(C')=\emptyset,$ there is no cell of $\mathfrak{R}$ that is both $C$-perimetric and $C'$-perimetric.
\end{observation}
Notice that if $c$ is $C$-perimetric, then $\tilde{c}$ contains two points $p,q\in N(\Gamma)$
such that  $\pi(p)$ and $\pi(q)$ are vertices of $C$ where one,
say $P_{c}^\mathrm{in},$ of the two $(\pi(p),\pi(q))$-subpaths of $C$ is a subgraph of $\sigma(c)$ and the other,
denoted by $P_{c}^\mathrm{out},$  $(\pi(p),\pi(q))$-subpath contains at most one vertex of $\sigma(c)$ that is internal to $P_{c}^\mathrm{out}$ and
which, if it exists, must be the (unique) vertex $z$ in $\partial\sigma(c)\setminus\{\pi(p),\pi(q)\}.$
We pick a $(p,q)$-arc $A_{c}$ in $\hat{c}:={c}\cup\tilde{c}$ such that  $\pi^{-1}(z)\in A_{c}$ if and only if
$P_{c}^\mathrm{in}$ contains
the vertex $z$ as an internal vertex.

We consider the simple closed curve
$K_{C}=\cupall\{A_{c}\mid \mbox{$c$ is a $C$-perimetric cell of $\mathfrak{R}$}\}$
and we denote by $\Delta_{C}$ the closed disk bounded by $K_{C}$  that is contained in  $\Delta.$
A cell $c$ of $\mathfrak{R}$ is called \emph{$C$-internal} if $c\subseteq \Delta_{C}$
and is called \emph{$C$-external} if $\Delta_{C}\cap c=\emptyset.$
Notice that  the cells of $\mathfrak{R}$ are partitioned into  $C$-internal,  $C$-perimetric, and  $C$-external cells.

Let $c$ be a tidy $C$-perimetric cell of $\mathfrak{R}$ where $|\tilde{c}|=3.$ Notice that $c\setminus A_{c}$ has two arcwise-connected components and one of them is an open disk $D_{c}$ that is a subset of $\Delta_{C}.$
If the closure $\overline{D}_{c}$  of $D_{c}$ contains only two points of $\tilde{c}$ then we call the cell $c$ \emph{$C$-marginal}.
We refer the reader to \cite{SauST21amor} for figures illustrating the above notions.

\paragraph{Influence.}
For every $\mathfrak{R}$-normal cycle $C$ of $\mathsf{Compass}_{\mathfrak{R}}(W)$ we define the set
$$\mathsf{Influence}_{\mathfrak{R}}(C)=\{\sigma(c)\mid \mbox{$c$ is a cell of $\mathfrak{R}$ that is not $C$-external}\}.$$

A wall $W'$  of $\mathsf{Compass}_{\mathfrak{R}}(W)$  is \emph{$\mathfrak{R}$-normal} if $D(W')$ is $\mathfrak{R}$-normal.
Notice that every wall of $W$ (and hence every subwall of $W$) is an $\mathfrak{R}$-normal wall of $\mathsf{Compass}_{\mathfrak{R}}(W).$ We denote by ${\mathcal{S}}_{\mathfrak{R}}(W)$ the set of all $\mathfrak{R}$-normal walls of $\mathsf{Compass}_{\mathfrak{R}}(W).$ Given a wall $W'\in\mathcal{S}_{\mathfrak{R}}(W)$ and a cell $c$ of $\mathfrak{R},$
we say that $c$ is \emph{$W'$-perimetric/internal/external/marginal} if $c$ is  $D(W')$-perimetric/internal/external/marginal, respectively.
We also use $K_{W'},$ $\Delta_{W'},$ $\mathsf{Influence}_{\mathfrak{R}}(W')$ as shortcuts
for $K_{D(W')},$ $\Delta_{D(W')},$ $\mathsf{Influence}_{\mathfrak{R}}(D(W')),$ respectively.

\paragraph{Regular flatness pairs.}
We call a  flatness pair $(W,\mathfrak{R})$ of a graph $G$ \emph{regular}
if none of its cells is $W$-external, $W$-marginal, or untidy.

\paragraph{Tilts of flatness pairs.}
Let $(W,\mathfrak{R})$ and $(\tilde{W}',\tilde{\mathfrak{R}}')$  be two flatness pairs of a graph $G$
and let $W'\in\mathcal{S}_{\mathfrak{R}}(W).$
We assume that ${\mathfrak{R}}=(X,Y,P,C,\Gamma,\sigma,\pi)$
and $\tilde{\mathfrak{R}}'=(X',Y',P',C',\Gamma',\sigma',\pi').$
We say that   $(\tilde{W}',\tilde{\mathfrak{R}}')$   is a \emph{$W'$-tilt}
of $(W,\mathfrak{R})$ if
\begin{itemize}
	\item $\tilde{\mathfrak{R}}'$ does not have $\tilde{W}'$-external cells,
	\item  $\tilde{W}'$ is a tilt of $W',$
	\item  the set of $\tilde{W}'$-internal  cells of  $\tilde{\mathfrak{R}}'$ is the same as the set of $W'$-internal
	      cells of ${\mathfrak{R}}$ and their images via $\sigma'$ and ${\sigma}$ are also the same,
	\item $\mathsf{Compass}_{\tilde{\mathfrak{R}}'}(\tilde{W}')$ is a subgraph of $\cupall\mathsf{Influence}_{{\mathfrak{R}}}(W'),$ and
	\item if $c$ is a cell in $C(\Gamma') \setminus C(\Gamma),$ then $|\tilde{c}| \leq 2.$
\end{itemize}

The next observation follows from the third item above and the fact that the cells corresponding to flaps
containing a central vertex of $W'$ are all internal (recall that the height of a wall is always at least three).

\begin{observation}\label{label_surreptitiously}
	Let $(W,\mathfrak{R})$ be a flatness pair of a graph $G$ and $W'\in{\mathcal{S}}_{\mathfrak{R}}(W).$
	For every $W'$-tilt $(\tilde{W}',\tilde{\mathfrak{R}}')$ of $(W,\mathfrak{R}),$ the central vertices of $W'$ belong to the vertex set of $\mathsf{Compass}_{\tilde{\mathfrak{R}}'}(\tilde{W}').$
\end{observation}

Also, given a regular flatness pair $(W,\mathfrak{R})$ of a graph $G$ and a $W'\in\mathcal{S}_{\mathfrak{R}}(W),$
for every $W'$-tilt $(\tilde{W}', \tilde{\mathfrak{R}}')$ of $(W,\mathfrak{R}),$ by definition, none of its cells is $\tilde{W}'$-external, $\tilde{W}'$-marginal, or untidy -- thus, $(\tilde{W}', \tilde{\mathfrak{R}}')$ is regular.
Therefore, regularity of a flatness pair is a property that its tilts ``inherit''.

\begin{observation}\label{label_expressionism}
	If $(W,\mathfrak{R})$ is a regular flatness pair of a graph $G,$ then for every $W'\in\mathcal{S}_{\mathfrak{R}}(W),$ every $W'$-tilt of $(W,\mathfrak{R})$ is also regular.
\end{observation}

We next present the two main results of~\cite{SauST21amor}.
In fact, in \cite{SauST21amor} we
provide two algorithms that,
given a flatness pair $(W,\mathfrak{R})$
of a graph $G,$ compute a $W'$-tilt of $(W,\mathfrak{R}),$ for some given subwall $W'$ of $W,$ and a regular flatness pair of $G,$ respectively.
Here, we use the non-algorithmic version of these results.

\begin{proposition}\label{label_prosperously}
	Let $G$ be a graph and $(W,\mathfrak{R})$ be a flatness pair of $G.$
	For every wall $W'\in\mathcal{S}_{\mathfrak{R}}(W),$ there is a flatness pair $(\tilde{W}',\tilde{\mathfrak{R}}')$ that is a $W'$-tilt of $(W,\mathfrak{R}).$
\end{proposition}

\begin{proposition}
	\label{label_aberenjenado}
	Let $G$ be a graph
	and $({W},{\mathfrak{R}})$ be a flatness pair of $G.$
	There is a regular flatness pair $({W}^{\star},{\mathfrak{R}}^{\star})$ of $G,$ with the same height as $({W},{\mathfrak{R}}),$  such that $\mathsf{Compass}_{\mathfrak{R}^{\star}}(W^{\star})\subseteq \mathsf{Compass}_{\mathfrak{R}}(W).$
\end{proposition}

We conclude this subsection with the Flat Wall Theorem and, in particular, the version proved by Chuzhoy \cite{Chuzhoy15impr}, restated in our framework (see \cite[Proposition 7]{SauST21amor}).

\begin{proposition}\label{label_aldobrandesco}
	There exist two functions  $\newfun{label_everyevening}:\mathbb{N}\to \mathbb{N}$  and
	$\newfun{label_entretenerme}:\mathbb{N}\to \mathbb{N},$ where the images of $\funref{label_everyevening}$ are odd numbers, such that if $r \in \mathbb{N}_{\geq 3}$ is an odd integer, $t\in\mathbb{N}_{\geq 1},$
	$G$ is a graph that does not contain $K_t$ as a minor,  and  $W$ is an $\funref{label_everyevening}(t)\cdot r$-wall of $G,$
	then there is a set $A\subseteq V(G)$ with $|A|\leq \funref{label_entretenerme}(t)$
	and a flatness pair $(\tilde{W}',\tilde{\mathfrak{R}}')$ of $G\setminus A$ of height $r.$
	Moreover, $\funref{label_everyevening}(t)=\mathcal{O}(t^{2})$ and $\funref{label_entretenerme}(t)=t-5.$
\end{proposition}

\subsection{Homogeneous walls}\label{label_backwardness}
We first present some definitions on boundaried graphs and folios that will be used to define the notion of homogeneous walls.
Following this, we present some results concerning homogeneous walls that are key ingredients in our proofs.

\paragraph{Boundaried graphs.}
Let $t\in\mathbb{N}.$
A \emph{$t$-boundaried graph} is a triple $\bound{G} = (G,B,\rho)$ where $G$ is a graph, $B \subseteq V(G),$ $|B| = t,$ and
$\rho : B \to [t]$ is a bijection.
We call $B$ the \emph{boundary} of $\mathbf{G}$ and the vertices of $B$ \emph{the boundary vertices} of $\mathbf{G}.$
For $B'\subseteq B,$ we define the bijection $\rho[B']:B'\to [|B'|]$ such that for every $v\in B',$ $\rho[B'](v) = |\{u\in B' \mid \rho(u)\leq \rho(v)\}|.$
Also, for $S\subseteq V(G),$ we denote by $\mathbf{G}\setminus S$ the $t$-boundaried graph $(G\setminus S, B\setminus S, \rho[B\setminus S]).$
We  say that  $\textbf{G}_1=(G_1,B_1,\rho_1)$ and $\textbf{G}_{2}=(G_2,B_2,\rho_2)$
are \emph{isomorphic} if there is an isomorphism from $G_{1}$ to $G_{2}$
that extends the bijection $\rho_{2}^{-1}\circ \rho_{1}.$
The triple $(G,B,\rho)$ is a \emph{boundaried graph} if it is a $t$-boundaried graph for some $t\in\mathbb{N}.$
As in~\cite{RobertsonS95XIII} (see also \cite{BasteST20acom}), we define the \emph{detail} of a boundaried graph
$\bound{G} = (G,B,\rho)$ as  $\mathsf{detail}(\bound{G}):=\max\{|E(G)|,|V(G)\setminus B|\}.$
We denote by ${\mathcal{B}}^{(t)}$ the set of all (pairwise non-isomorphic)  $t$-boundaried graphs and by ${\mathcal{B}}_{\ell}^{(t)}$ the set of all (pairwise non-isomorphic) $t$-boundaried graphs with detail at most $\ell.$
We also set ${\mathcal{B}}=\bigcup_{t\in\mathbb{N}}{\mathcal{B}}^{(t)}.$

\paragraph{Topological minors of boundaried graphs.}
We say that $(M,T)$   is a \emph{\textsf{tm}-pair} if $M$ is  a graph, $T\subseteq V(M),$ and  all vertices in
$V(M)\setminus T$ have degree two. We denote by $\mathsf{diss}(M,T)$ the graph obtained
from  $M$ by {dissolving} all vertices  in $V(M)\setminus T.$
A \emph{\textsf{tm}-pair} of a graph $G$  is a  \emph{\textsf{tm}-pair}  $(M,T)$ where $M$ is a subgraph of $G.$
We call the vertices in $T$ \emph{branch} vertices of $(M,T).$
We need to deal with topological minors for the notion of homogeneity defined below, on which  the statement of~\cite[Theorem 5.2]{BasteST20acom} relies. This result will be crucial in the proof of \autoref{label_acquaintances}.

If $\textbf{M}=(M,B,\rho)\in{\mathcal{B}}$ and   $T\subseteq V(M)$ with $B\subseteq T,$ we  call  $(\textbf{M},T)$ a \emph{\textsf{btm}-pair}
and we  define  $\mathsf{diss}(\textbf{M},T)=(\mathsf{diss}(M, T),B,\rho).$ Note that we do not permit dissolution of boundary vertices, as we consider all of them to be branch vertices. If $\textbf{G}=(G,B,\rho)$ is a boundaried graph and $(M,T)$ is a \textsf{tm}-pair of $G$
where $B\subseteq T,$  then we say that
$(\textbf{M},T),$ where $\textbf{M}=(M,B,\rho),$ is a   \emph{\textsf{btm}-pair} of $\textbf{G}=(G,B,\rho).$
Let $\textbf{G}_{1},\mathbf{G}_{2}$ be two boundaried graphs.
We say that $\textbf{G}_{1}$ is a \emph{topological minor}
of $\textbf{G}_{2},$ denoted by $\textbf{G}_{1}\pretp\textbf{G}_{2},$ if
$\textbf{G}_{2}$ has a \textsf{btm}-pair $(\textbf{M},T)$
such that  $\mathsf{diss}(\textbf{M},T)$ is isomorphic to $\textbf{G}_{1}.$

\paragraph{Folios.}
Given a $\textbf{G}\in\mathcal{B}$ and a positive integer $\ell,$ we define the \emph{$\ell$-folio} of $\mathbf{G}$
as
$${\ell}\textsf{-folio}(\textbf{G})=\{\textbf{G}'\in\mathcal{ B} \mid \textbf{G}'\pretp \textbf{G} \mbox{~and $\textbf{G}'$ has detail at most $\ell$}\}.$$

The number of distinct $\ell$-folios of $t$-boundaried graphs is indicated in the following result, proved first in~\cite{BasteST20hittI} and used also in~\cite{BasteST20acom}.
\begin{proposition}\label{label_objectivement}
	There exists a function $\newfun{label_dogmatizador}: \mathbb{N}^{2} \to \mathbb{N}$ such that for every $t,\ell\in \mathbb{N},$ $|\{\ell\textsf{-folio}(\mathbf{G}) \mid \mathbf{G}\in\mathcal{B}_{\ell}^{(t)}\}|\leq \funref{label_dogmatizador}(t,\ell).$ Moreover, $\funref{label_dogmatizador}(t,\ell)=2^{2^{\mathcal{O}((t+\ell)\cdot\log(t+\ell))}}.$
\end{proposition}

\paragraph{Augmented flaps.}
Let $G$ be a graph, $A$ be a subset of $V(G)$ of size $a,$ and $(W,\mathfrak{R})$ be a flatness pair of $G\setminus A.$
For each flap $F\in \mathsf{Flaps}_{\mathfrak{R}}(W)$ we consider an injective labeling  $\lambda_{F}: \partial F\rightarrow\{1,2,3\}$ such that
the set of labels assigned by $\lambda_{F}$ to $\partial F$ is  one of $\{1\},$ $\{1,2\},$ $\{1,2,3\}.$
Also, let $\tilde{a}\in[a].$
For every set ${\tilde{A}}\in \binom{A}{\tilde{a}},$ we consider a bijection  $\rho_{\tilde{A}}: \tilde{A}\to [\tilde{a}].$
The labellings in $\mathcal{L}=\{\lambda_{F} \mid F\in \mathsf{Flaps}_{\mathfrak{R}}(W)\}$ and the labellings in $\{\rho_{\tilde{A}} \mid \tilde{A} \in \binom{A}{\tilde{a}}\}$
will be useful for defining a set of boundaried graphs that we will call augmented flaps. We first need some more definitions.

Given a flap $F\in\mathsf{Flaps}_{\mathfrak{R}}(W),$ we define an ordering
$\Omega(F)=(x_{1},\ldots,x_{q}),$ with $q\leq 3,$ of the vertices of $\partial{F}$
so that
\begin{itemize}
	\item $(x_{1},\ldots,x_{q})$ is a  counter-clockwise cyclic ordering of the vertices of $\partial F$ as they appear in the corresponding cell of $C(\Gamma).$ Notice that this cyclic ordering is significant  only when $|\partial F|=3,$
	      in the sense that $(x_{1},x_{2},x_{3})$ remains invariant under shifting, i.e., $(x_{1},x_{2},x_{3})$ is the same as $ (x_{2},x_{3},x_{1})$ but not  under inversion, i.e.,   $(x_{1},x_{2},x_{3})$ is not the same as $(x_{3},x_{2},x_{1}),$ and
	\item   for $i\in[q],$ $\lambda_{F}(x_{i})=i.$
\end{itemize}
Notice that the second condition is necessary for completing the definition of the ordering $\Omega(F),$
and this is the reason why we set up the labellings in $\mathcal{L}.$\medskip\medskip

For each set $\tilde{A} \in \binom{A}{\tilde{a}}$ and each $F\in \mathsf{Flaps}_{\mathfrak{R}}(W)$ with $t_{F}:=|\partial F|,$
we fix $\rho_{F}: \partial F\to [\tilde{a}+1,\tilde{a}+t_F]$ such that
$(\rho^{-1}_{F}(\tilde{a}+1),\ldots,\rho^{-1}_{F}(\tilde{a}+t_F))= \Omega(F).$
Also, we define the boundaried graph $$\textbf{F}^{\tilde{A}}:=(G[\tilde{A}\cup F],\tilde{A}\cup \partial F,\rho_{\tilde{A}}\cup \rho_F)$$
and we denote by $F^{\tilde{A}}$ the underlying graph of $\textbf{F}^{\tilde{A}}.$ We call $\textbf{F}^{\tilde{A}}$ an \emph{$\tilde{A}$-augmented flap} of the flatness pair $(W,\mathfrak{R})$ of $G\setminus A$
in $G.$

\paragraph{Palettes and homogeneity.}
For each $\mathfrak{R}$-normal cycle $C$ of $\mathsf{Compass}_\mathfrak{R} (W)$ and each set $\tilde{A}\in 2^A,$ we define $(\tilde{A},\ell)\textsf{-palette}(C)=\{\ell\textsf{-folio}(\mathbf{F}^{\tilde{A}})\mid F\in \mathsf{Influence}_{\mathfrak{R}}(C)\}.$
Given a set $\tilde{A}\in 2^A,$ we say that the flatness pair $(W,\mathfrak{R})$  of $G\setminus A$ is \emph{$\ell$-homogeneous with respect to $\tilde{A}$} if every  \textsl{internal} brick of ${W}$ has the \textsl{same} $(\tilde{A},\ell)$\textsf{-palette} (seen as a cycle of $\mathsf{Compass}_\mathfrak{R} (W)$).
Also, given a collection ${\mathcal{S}}\subseteq 2^A,$ we say that the flatness pair $(W,\mathfrak{R})$  of $G\setminus A$ is \emph{$\ell$-homogeneous	with respect to ${\mathcal{S}}$}
if it is $\ell$-homogeneous with respect to every $\tilde{A}\in\mathcal{S}.$

The following observation is a consequence of the fact that, given a wall $W$ and a  subwall $W'$ of $W,$ every internal brick of a tilt $W''$ of $W'$ is also an internal brick of $W.$

\begin{observation}\label{label_convenciones}
	Let $\ell\in\mathbb{N},$ $G$ be a graph, $A\subseteq V(G),$ ${\mathcal{S}}\subseteq 2^A,$ and $(W,\mathfrak{R})$  be a flatness pair of $G\setminus A.$ If $(W,\mathfrak{R})$ is  $\ell$-homogeneous
	with respect to ${\mathcal{S}},$ then for every subwall $W'$ of $W,$ every $W'$-tilt of $(W,\mathfrak{R})$ is also $\ell$-homogeneous
	with respect to ${\mathcal{S}}.$
\end{observation}

\medskip

Let $a,\tilde{a},\ell\in \mathbb{N},$ where $\tilde{a}\leq a.$
Also, let $G$ be a graph, $A$ be a subset of $V(G)$ of size at most $a,$ and $(W,\mathfrak{R})$ be a flatness pair of $G\setminus A.$
For every flap  $F\in\mathsf{Flaps}_\mathfrak{R} (W),$ we define the function
$\mathsf{var}^{(A,\tilde{a},\ell)}_F:\binom{A}{\leq \tilde{a}}\to \{\ell\textsf{-folio}(\mathbf{G}) \mid \mathbf{G}\in \bigcup_{i\in[\tilde{a}+3]}{\mathcal{B}}^{(i)}\}$
that maps each set $\tilde{A}\in\binom{A}{\leq \tilde{a}}$ to the set  $\ell\textsf{-folio}(\mathbf{F}^{\tilde{A}}).$

We next provide an upper bound to the number of different $\ell$-folios of the augmented flaps of a flatness pair $(W,\mathfrak{R}).$

\begin{lemma}\label{label_anticipation}
	There exists a function $\newfun{label_exiieriences}:\mathbb{N}^3\to \mathbb{N}$ such that if $a,\tilde{a},\ell\in \mathbb{N},$ where $\tilde{a}\leq a,$ $G$ is a graph, $A$ is a subset of $V(G)$ of size at most $a,$ and $(W,\mathfrak{R})$ is a flatness pair of $G\setminus A,$ then $$|\{\mathsf{var}^{(A,\tilde{a},\ell)}_F\mid F\in \mathsf{Flaps}_\mathfrak{R} (W)\}|\leq \funref{label_exiieriences}(a,\tilde{a},\ell).$$
	Moreover, $ \funref{label_exiieriences}(a,\tilde{a},\ell)= 2^{a^{\tilde{a}} \cdot 2^{\mathcal{O}((\tilde{a}+\ell )\cdot \log (\tilde{a}+\ell))}}.$
\end{lemma}

The proof of \autoref{label_anticipation} follows directly from \autoref{label_objectivement} combined with the fact that there are $\mathcal{O}(|A|^{\tilde{a}})$ elements in $\binom{A}{\leq \tilde{a}}$.

\autoref{label_anticipation} allows us to define an injective function $\sigma: \{\mathsf{var}^{(A,\tilde{a},\ell)}_F\mid F\in \mathsf{Flaps}_\mathfrak{R} (W)\} \to [\funref{label_exiieriences}(a,\tilde{a},\ell)]$
that maps each function in $\{\mathsf{var}^{(A,\tilde{a},\ell)}_F\mid F\in\mathsf{Flaps}_\mathfrak{R} (W)\}$ to an integer in $[\funref{label_exiieriences}(a,\tilde{a},\ell)].$
Using $\sigma,$ we define a function $\zeta_{A,\tilde{a},\ell}:\mathsf{Flaps}_\mathfrak{R} (W)\to [\funref{label_exiieriences}(a,\tilde{a},\ell)],$ that maps each flap $F\in\mathsf{Flaps}_\mathfrak{R} (W)$ to the integer $\sigma(\mathsf{var}^{(A,\tilde{a},\ell)}_F).$
In \cite{SauST21amor}, given a $w\in\mathbb{N},$ the notion of homogeneity is defined with respect to a \textsl{flap-coloring $\zeta$ of $(W,\mathfrak{R})$ with $w$ colors}, that is a function from $\mathsf{Flaps}_\mathfrak{R} (W)$ to $[w].$
This function gives rise to the $\zeta\textsf{-palette}$ of each $\mathfrak{R}$-normal cycle of $\mathsf{Compass}_{\mathfrak{R}} (W)$ which, in turn, is used to define the notion of a \emph{$\zeta$-homogeneous} flatness pair.
Hence, using the terminology of \cite{SauST21amor},  $\zeta_{A,\tilde{a},\ell}$ is a flap-coloring of $(W,\mathfrak{R})$ with $\funref{label_exiieriences}(a,\tilde{a},\ell)$ colors, that ``colors''  each flap $F\in \mathsf{Flaps}_\mathfrak{R} (W)$ by mapping it to the integer $\sigma(\mathsf{var}^{(A,\tilde{a},\ell)}_F),$ and the notion of $\ell$-homogeneity
with respect to $\binom{A}{\leq {\tilde{a}}}$ defined here can be alternatively interpreted as $\zeta_{A,\tilde{a},\ell}$-homogeneity.
The following result, which is the application of a result of Sau et al.~\cite[Lemma 12]{SauST21amor} for the flap-coloring $\zeta_{A,\tilde{a},\ell},$ provides the conditions that guarantee the existence of a homogeneous flatness pair ``inside'' a given flatness pair of a graph.

\begin{proposition}\label{label_transportado}
	There exists a function $\newfun{label_complications}:\mathbb{N}^4\to \mathbb{N},$ whose images are odd integers,
	such that for every odd integer $r\geq 3,$ every $a,\tilde{a},\ell\in \mathbb{N}$ where $\tilde{a}\leq a,$ if $G$ is a graph, $A$ is a subset of $V(G)$ of size at most $a,$ and $(W,\mathfrak{R})$ is a flatness pair of $G\setminus A$ of height $\funref{label_complications}(r,a,\tilde{a},\ell),$ then $W$ contains some subwall $W'$ of height $r$ such that every $W'$-tilt of $(W,\mathfrak{R})$ is  $\ell$-homogeneous
	with respect to $\binom{A}{\leq {\tilde{a}}}.$ Moreover, $\funref{label_complications}(r,a,\tilde{a},\ell) =\mathcal{O}(r^{\funref{label_exiieriences}(a,\tilde{a},\ell)}).$
\end{proposition}

\section{Obstructions have small treewidth}
The goal of this section is to provide an upper bound on the treewidth of every graph in $\mathbf{obs}(\mathcal{A}_k (\mathbf{excl}(\mathcal{F}))),$ stated in \autoref{label_transfiguration}.
In order to prove this,
in \autoref{label_villebrequin}, we
define the notion of a canonical partition of a graph with respect to a flatness pair and provide some additional results (\autoref{label_overestimation} and \autoref{lemma_bidim_branch}) that will be useful in the proof of \autoref{label_transfiguration}.
Also, in \autoref{label_reconquistasen} we argue how to detect an irrelevant vertex inside a homogeneous flatness pair of ``big enough'' height.
The proof of \autoref{label_transfiguration} is finally presented in \autoref{label_informaremos}. The proof of \autoref{lemma_bidim_branch} is postponed to \autoref{label_nominalistic}.

\subsection{Results on canonical partitions}\label{label_villebrequin}

\paragraph{Canonical partitions.}
Let $r\geq 3$ be an odd integer, let $W$ be an $r$-wall, and let $P_{1}, \ldots, P_{r}$ (resp. $L_{1},\ldots, L_{r}$) be its vertical (resp. horizontal) paths.
For every even  (resp. odd) $i\in[2,r-1]$ and every $j\in[2,r-1],$ we define ${A}^{(i,j)}$ to be the  subpath of $P_{i}$ that starts from a vertex of $P_{i}\cap L_{j}$ and finishes at a neighbor of a vertex in $L_{j+1}$ (resp. $L_{j-1}$), such that $P_{i}\cap L_{j}\subseteq A^{(i,j)}$ and $A^{(i,j)}$ does not intersect $L_{j+1}$ (resp. $L_{j-1}$).
Similarly, for every $i,j\in[2,r-1],$ we define $B^{(i,j)}$ to be the subpath of $L_{j}$ that starts from a vertex of $P_{i}\cap L_{j}$ and finishes at a neighbor of a vertex in $P_{i-1},$ such that $P_{i}\cap L_{j}\subseteq A^{(i,j)}$ and $A^{(i,j)}$ does not intersect $P_{i-1}.$

For every  $i,j\in[2,r-1],$ we denote by $Z^{(i,j)}$ the graph $A^{(i,j)}\cup B^{(i,j)}$ and by ${Z_\mathrm{ext}}$ the graph $W\setminus \bigcup_{i,j\in[2,r-1]} Z_{i,j}.$
Now consider the collection $\mathcal{Q}=\{Z_\mathrm{ext}\}\cup\{Z_{i,j}\mid i,j\in[2,r-1]\}$
and observe that the graphs in $\mathcal{Q}$ are connected subgraphs of $W$ and their vertex sets form a partition of $V(W).$
We call $\mathcal{Q}$ the \emph{canonical partition} of $W.$ Also, we call every $Z_{i,j},$ for $i,j\in[2,r-1],$ an \emph{internal bag} of $\mathcal{Q},$ while we refer to $Z_\mathrm{ext}$ as the \emph{external bag} of $\mathcal{Q}.$ See \autoref{label_bitternesses} for an illustration of the notions defined above.

\begin{figure}[ht]
	\centering
	\includegraphics[width=6cm]{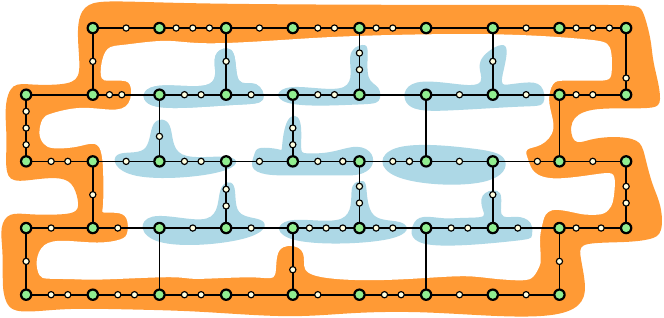}
	\caption{\small A $5$-wall and its canonical partition $\mathcal{Q}.$ The orange bag is the external bag $Z_\mathrm{ext}.$}
	\label{label_bitternesses}
\end{figure}

Let $(W,\mathfrak{R})$ be a flatness pair of a graph $G.$
Consider the canonical partition $\mathcal{Q}$ of $W.$ We enhance the graphs of $\mathcal{Q}$
so to include in them all the vertices of $G$ by applying the following procedure. We set $\tilde{\mathcal{Q}}:=\mathcal{Q}$
and, as long as there is a vertex  $x\in V(\textsf{Compass}_{\mathfrak{R}}(W))\setminus V(\cupall \tilde{\mathcal{Q}})$
that is adjacent to a vertex of a graph $Z\in \tilde{\mathcal{Q}},$  update $\tilde{\mathcal{Q}}:=\tilde{\mathcal{Q}}\setminus \{Z\}\cup \{\tilde{Z}\},$ where $\tilde{Z}=\textsf{Compass}_{\mathfrak{R}}(W)[\{x\}\cup V(Z)].$ Since $\mathsf{Compass}_{\mathfrak{R}}(W)$ is a connected graph, in this way we define a partition of the vertices of $\mathsf{Compass}_{\mathfrak{R}}(W)$ into subsets inducing connected graphs.
We call the $\tilde{Z}\in\tilde{\mathcal{Q}}$ that contains $Z_\mathrm{ext}$ as a subgraph the \emph{external bag} of $\tilde{\mathcal{Q}},$ and we denote it by $\tilde{Z}_\mathrm{ext},$ while we call \emph{internal bags} of $\tilde{\mathcal{Q}}$ all graphs in $\tilde{\mathcal{Q}}\setminus \{\tilde{Z}_\mathrm{ext}\}.$
Moreover, we enhance $\tilde{\mathcal{Q}}$ by adding all vertices of $G\setminus V(\textsf{Compass}_\mathfrak{R} (W)$ in its external bag, i.e., by updating $\tilde{Z}_\mathrm{ext}: = G[V(\tilde{Z}_\mathrm{ext})\cup V(G\setminus V(\textsf{Compass}_\mathfrak{R} (W))].$
We call such a  partition $\tilde{\mathcal{Q}}$ a \emph{$(W,\mathfrak{R})$-canonical partition of $G.$}
Notice that a $(W,\mathfrak{R})$-canonical partition of $G$ is  not unique (since the sets in $\mathcal{Q}$ can be ``expanded'' arbitrarily when introducing vertex $x$).
We stress that every internal bag of a $(W,\mathfrak{R})$-canonical partition of $G$ contains vertices of exactly four bricks of $W.$

Let $(W,\mathfrak{R})$ be a flatness pair of a graph $G$ of height $r,$ for some odd $r\geq 3,$ and let $\tilde{\mathcal{Q}}$ be a $(W,\mathfrak{R})$-canonical partition of $G.$
For every $i\in[(r-1)/2],$ we say that a graph $Z\in \tilde{\mathcal{Q}}$ is an \emph{$i$-internal bag of $\tilde{\mathcal{Q}}$} if $V(Z)$ does not contain any vertex of the first $i$ layers of $W.$
Notice that the $1$-internal bags of $\tilde{\mathcal{Q}}$ are the internal bags of $\tilde{\mathcal{Q}}.$
\medskip

The next result intuitively states that, given a flatness pair $(W,\mathfrak{R})$ of ``big enough'' height and a $(W,\mathfrak{R})$-canonical partition $\tilde{\mathcal{Q}}$ of $G,$ we can find a ``packing'' of subwalls of $W$ that are inside some central part of $W$ and such that the vertex set of every internal bag of $\tilde{\mathcal{Q}}$ intersects the vertices of the flaps in the influence of at most one of these walls.

\begin{lemma}\label{label_overestimation}
	There exists a function $\newfun{label_apreciadores}: \mathbb{N}^3 \to \mathbb{N}$ such that if $p,l\in\mathbb{N}_{\geq 1},$ $x\in\mathbb{N}_{\geq 3}$ is an odd integer, $G$ is a graph, $(W,\mathfrak{R})$ is a flatness pair of $G$ of height at least $\funref{label_apreciadores}(l,x,p),$ and $\tilde{\mathcal{Q}}$ is a $(W,\mathfrak{R})$-canonical partition of $G,$ then
	there is a collection $\mathcal{W}=\{W_1, \ldots, W_{l}\}$ of $x$-subwalls of $W$ such that
	\begin{itemize}
		\item for every $i \in [l],$ $\cupall\mathsf{Influence}_{\mathfrak{R}}(W_i)$ is a subgraph of $\cupall \{Z\mid Z \text{ is a $p$-internal bag of }\tilde{\mathcal{Q}}\}$ and
		\item for every $i,j\in[l]$ with $i\neq j,$ there is no internal bag of $\tilde{\mathcal{Q}}$ that has vertices of both $V(\cupall\mathsf{Influence}_\mathfrak{R} (W_i))$ and $V(\cupall\mathsf{Influence}_\mathfrak{R} (W_j)).$
\end{itemize}
Moreover, $\funref{label_apreciadores}(l,x,p)=\mathcal{O}(\sqrt{l}\cdot x + p)$
and $\mathcal{W}$ can be constructed in time $\mathcal{O}(n+m)$.
\end{lemma}
\begin{proof}
We set $r =\mathsf{odd}(\lceil \sqrt{l} \cdot (x+2)\rceil)$
and $\funref{label_apreciadores}(l,x,p)= r + 2(p+1).$
Recall that $W^{(r)}$ is the central $r$-subwall of $W.$
Since $W$ has height $r + 2(p+1),$ $W^{(r)}$ does not intersect the first $(p+1)$ layers of $W,$
and therefore $V(\cupall\mathsf{Influence}_\mathfrak{R} (W^{(r)}))$ does not contain
any vertex of the first $p$ layers of $W.$
This implies that $\cupall\mathsf{Influence}_\mathfrak{R} (W^{(r)})$
is a subgraph of $\cupall \{Z\mid Z \text{ is a $p$-internal bag of }\tilde{\mathcal{Q}}\}.$
Also, since $r \geq \lceil \sqrt{l} \cdot (x+2)\rceil,$
there exists a collection $\bar{\mathcal{W}}=\{\bar{W}_1, \ldots, \bar{W}_{l}\}$ of $(x+2)$-subwalls
of $W^{(r)}$ (that are also subwalls of $W$), such that for every $i \in [l],$
$\cupall\mathsf{Influence}_{\mathfrak{R}}(\bar{W}_i)$ is a subgraph
of $\cupall\mathsf{Influence}_\mathfrak{R} (W^{(r)})$ and
for every $i,j\in[l]$ with $i\neq j,$ the vertex set of $\cupall\mathsf{Influence}_\mathfrak{R} (\bar{W}_i)$
and the vertex set of $\cupall\mathsf{Influence}_\mathfrak{R} (\bar{W}_j)$ are disjoint.
To see why the latter holds, notice that, due to \autoref{label_grundfalscher},
there are no cells of $\mathfrak{R}$ that are both
$\bar{W}_i$-perimetric and $\bar{W}_j$-perimetric.

Notice now that there may exist an internal bag
$Z\in \tilde{\mathcal{Q}}$ such that $V(Z)$
intersects both $V(\cupall\mathsf{Influence}_\mathfrak{R} (\bar{W}_i))$ and
$V(\cupall\mathsf{Influence}_\mathfrak{R} (\bar{W}_j)),$ for some $i,j\in[l].$
To tackle this, for every $i\in[l]$ we set $W_i$
to be the central $x$-subwall of $\bar{W}_i$ and observe that,
for every $i,j\in[l]$ with $i\neq j,$
if there exists an internal bag $Z\in\tilde{\mathcal{Q}}$ such that $V(Z)$ intersects
both $V(\cupall\mathsf{Influence}_\mathfrak{R} (W_i))$ and $V(\cupall\mathsf{Influence}_\mathfrak{R} (W_j)),$
then $Z$ contains vertices of more than four bricks of $W,$
that is a contradiction to the definition of the canonical partition.
Thus, the collection $\mathcal{W}=\{W_1,\ldots, W_{l}\}$ is the desired one.
\end{proof}

\medskip

The next result provides the conditions to detect a vertex set that should necessarily intersect every set $S\subseteq V(G)$ of size at most $k$ such that $G\setminus S \in \mathbf{excl}(\mathcal{F}).$
In other words, given a graph $G,$ a set $A\subseteq V(G),$
and a $(W,\mathfrak{R})$-canonical partition $\tilde{\mathcal{Q}}$ of $G\setminus A,$
for some flatness pair $(W,\mathfrak{R})$ of $G\setminus A,$
we provide the conditions for a
set of vertices in $A$ with ``big enough'' degree with respect to the bags of $\tilde{\mathcal{Q}}$
to intersect every set $S\subseteq V(G)$ of size at most $k$ such that $G\setminus S \in \mathbf{excl}(\mathcal{F}).$
Recall that $a_{\mathcal{F}}$ is the minimum
apex number of a graph in $\mathcal{F}.$

In fact, we present an even more general formulation that will be needed in future work.
Namely,
instead of considering a set $S\subseteq V(G)$ of at most $k$ vertices such that $G\setminus S\in \mathbf{excl}(\mathcal{F}),$
we can be more ``flexible''
on the ``measure'' required on the set $S$, and we can ask, instead of $S$ having size at most $k$,  that $S$
intersects at most $k$ internal bags of every $(W,\mathfrak{R})$-canonical partition of $G\setminus A.$

\begin{lemma}\label{lemma_bidim_branch}
There exist three functions $\newfun{@prepossession}, \newfun{@proclamation}, \newfun{@engrandecerla}: \mathbb{N}^{3}\to \mathbb{N}$,
	such that if
	$\mathcal{F}$ is a finite collection of graphs,
	$G$ is a graph,
	$k\in\mathbb{N}$,
	$A$ is a subset of $V(G)$,
	$(W,\mathfrak{R})$ is a flatness pair of $G\setminus A$ of height at least $\funref{@prepossession}(a_{\mathcal{F}},s_{\mathcal{F}},k)$,
	$\tilde{\mathcal{Q}}$ is a $(W,\mathfrak{R})$-canonical partition of $G\setminus A$,
	$A'$ is a subset of vertices of $A$ that are adjacent, in $G$, to vertices of at least $\funref{@proclamation}(a_{\mathcal{F}},s_{\mathcal{F}},k)$ $\funref{@engrandecerla}(a_{\mathcal{F}},s_{\mathcal{F}},k)$-internal bags of $\tilde{\mathcal{Q}}$, and $|A'|\geq a_{\mathcal{F}}$,
	then for every set $S\subseteq V(G)$ such that $G\setminus S \in \mathbf{excl}({\mathcal{F}})$ and $S$ intersects at most $k$ internal bags of every $(W,\mathfrak{R})$-canonical partition of $G\setminus A$,
	it holds that $S\cap A'\neq\emptyset$.
	Moreover, $\funref{@prepossession}(a,s,k)=\mathcal{O}(2^a \cdot  s^{2} \cdot k^{2}),$
	$\funref{@proclamation}(a,s,k)=\mathcal{O}(2^a \cdot s^3 \cdot k^3)$, and $\funref{@engrandecerla}(a,s,k)=\mathcal{O}((a^2 +k)\cdot s)$, where $a=a_{\mathcal{F}}$ and $s= s_{\mathcal{F}}$.
\end{lemma}

The proof of \autoref{lemma_bidim_branch} is postponed to \autoref{label_nominalistic}.
The following corollary is an immediate consequence of~\autoref{lemma_bidim_branch}, since every set $S\subseteq V(G)$ of size at most $k$ clearly intersects at most $k$ internal bags of every $(W,\mathfrak{R})$-canonical partition of $G\setminus A.$

\begin{corollary}\label{label_significativos}
There exist three functions $\funref{@prepossession}, \funref{@proclamation},\funref{@engrandecerla}: \mathbb{N}^{3}\to \mathbb{N},$
such that if
$\mathcal{F}$ is a finite collection of graphs,
$G$ is a graph,
$k\in\mathbb{N},$
$A$ is a subset of $V(G),$
$(W,\mathfrak{R})$ is a flatness pair of $G\setminus A$ of height at least $\funref{@prepossession}(a_{\mathcal{F}},s_{\mathcal{F}},k),$
$\tilde{\mathcal{Q}}$ is a $(W,\mathfrak{R})$-canonical partition of $G\setminus A,$
$A'$ is a subset of vertices of $A$ that are adjacent, in $G,$ to vertices of at least $\funref{@proclamation}(a_{\mathcal{F}},s_{\mathcal{F}},k)$ $\funref{@engrandecerla}(a_{\mathcal{F}},s_{\mathcal{F}},k)$-internal bags of $\tilde{\mathcal{Q}},$ and $|A'|\geq a_{\mathcal{F}},$
then for every  set $S\subseteq V(G)$ of size at most $k$ such that $G\setminus S \in \mathbf{excl}(\mathcal{F})$
it holds that  $S\cap A'\neq\emptyset.$
Moreover, $\funref{@prepossession}(a,s,k)=\mathcal{O}(2^a \cdot  s^{2} \cdot k^{2}),$
$\funref{@proclamation}(a,s,k)=\mathcal{O}(2^a \cdot s^3 \cdot k^3),$ and $\funref{@engrandecerla}(a,s,k)=\mathcal{O}((a^2 +k)\cdot s),$ where $a=a_{\mathcal{F}}$ and $s= s_{\mathcal{F}}.$
\end{corollary}

We would like to comment that in the version of
 \autoref{label_significativos} used in~\cite{SauST21kapiII},
we write that $\funref{@prepossession}(a,s,k)=\mathcal{O}(2^a \cdot  s^{5/2} \cdot k^{5/2})$.
The ``improved'' function $\funref{@prepossession}$ that appears here
does not imply any improvement to the asymptotics of the running times of the algorithms in~\cite{SauST21kapiII}.

\subsection{Existence of an irrelevant wall inside a homogeneous flat wall}\label{label_reconquistasen}

The \textsl{irrelevant vertex technique} was introduced in~\cite{RobertsonS95XIII}
for providing an \textsf{FPT}-algorithm for the \textsc{Disjoint Paths} problem.
Moreover, this technique has appeared to be quite versatile and is now a standard tool
of parameterized algorithm design (see e.g.,~\cite{CyganFKLMPPS15para,ThilikosBDFM12grap}).

The fact that in the compass of a ``large enough'' homogeneous flat wall there exists a flat wall whose compass is irrelevant is asserted  by \autoref{label_acquaintances} and this subsection is devoted to its proof.
We first give some additional definitions and present a result that we derive from \cite{BasteST20acom}.
\medskip

Let $G$ be a graph and let $\ell\in \mathbb{N}.$
We say that a vertex set $X\subseteq V(G)$ is
	\emph{$\ell$-irrelevant} if every graph $H$ with detail at most $\ell$ that is a minor of $G$ is also a minor of $G\setminus X.$\smallskip

We state the following result from \cite{BasteST19hittIV} (see also \cite{BasteST20acom}).
In fact, \autoref{label_panlatinismo} is stated in \cite[Theorem 23]{BasteST19hittIV} for boundaried graphs. \autoref{label_panlatinismo} is derived by the same proof
if we consider graphs with empty boundary.

\begin{proposition}\label{label_panlatinismo}
	There exist two functions
	$\newfun{label_conversational}: \mathbb{N}^3\to\mathbb{N}$ and
	$\newfun{label_unbelievability}: \mathbb{N}^2\to\mathbb{N},$
	where the images of $\funref{label_conversational}$ are odd numbers,
	such that,
	for every $a,\ell\in\mathbb{N},$ every odd $q\in\mathbb{N}_{\geq 3},$ and  every graph $G,$
	if  $A\subseteq V(G),$ where $|A|\leq a,$ and  $(W,\mathfrak{R})$ is  a regular
	flatness pair of $G\setminus A$ of height at least $\funref{label_conversational}(a,\ell,q)$ that is $\funref{label_unbelievability}(a,\ell)$-homogeneous with respect to $A,$
	then the vertex set of the compass of every $W^{(q)}$-tilt of $(W,\mathfrak{R})$ is $\ell$-irrelevant.
	Moreover, it holds that $\funref{label_conversational}(a,\ell,q)=\mathcal{O}( (f_\mathsf{ul}(16a+12\ell))^3   + q)$
	and $\funref{label_unbelievability}(a,\ell)=a+\ell+3,$ where $f_\mathsf{ul}$ is the function of the Unique Linkage Theorem.
\end{proposition}

Based on the above result, we prove that, given a graph $G,$ a set $A\subseteq V(G),$
and a ``big enough'' flatness pair $(W,\mathfrak{R})$ of $G\setminus A$ that is homogeneous with respect to $\binom{A}{\leq a},$ for some integer $a\leq |A|$,
there is a flatness pair that is a tilt of a central subwall of $W$
and its compass is ``irrelevant'' to the fact that $G\in \mathcal{A}_{k} (\mathbf{excl}(\mathcal{F})).$
In the following result, we demand homogeneity with respect to $\binom{A}{\leq a}$ although, for the proofs of this paper, it would suffice to demand homogeneity with respect to a set $\tilde{A}\in  \binom{A}{\leq a}$ that ``avoids'' every set $S\subseteq V(G)$ of at most $k$ vertices such that $|A\setminus S|\leq a$ and $G\setminus S\in \mathbf{excl}(\mathcal{F}).$
We insist on this more general formulation because it is used in~\cite{SauST21kapiII}.
In fact, we present an even more general formulation that will be needed in future work.
Namely,
instead of considering a set $S\subseteq V(G)$ of at most $k$ vertices such that $G\setminus S\in \mathbf{excl}(\mathcal{F}),$
we can be more ``flexible''
on the ``measure'' required on the set $S$, and we can ask, instead of $S$ having size at most $k$,  that $S$  intersects at most $k$ internal bags of every $(W,\mathfrak{R})$-canonical partition of $G\setminus A.$
Recall that $\ell_{\mathcal{F}}=\max\{\mathsf{detail}(H)\mid H\in{\mathcal{F}}\}.$

\begin{lemma}\label{lemma_biiid_irr}
There exists a function $\newfun{label_interpersonal}: \mathbb{N}^{4}\to \mathbb{N},$
whose images are odd numbers, such that given
$k, q, a\in\mathbb{N}$, with odd $q\geq 3$,
a finite collection $\mathcal{F}$ of graphs,
a graph $G$,
a subset $A\subseteq V(G)$,
and a regular flatness pair $(W,\mathfrak{R})$ of $G\setminus A$ of height at least $\funref{label_interpersonal}(a,\ell_{\mathcal{F}},q,k)$ that is $\funref{label_unbelievability}(a,\ell_{\mathcal{F}})$-homogeneous with respect to $\binom{A}{\leq a}$,
it holds that for every $W^{(q)}$-tilt $(W',\mathfrak{R}')$ of $(W,\mathfrak{R})$ and
for every set $S\subseteq V(G)$ that intersects at most $k$ internal bags of every $(W,\mathfrak{R})$-canonical partition of $G\setminus A$
and $|A\setminus S|\leq a$,
it holds that $G\setminus S\in\mathbf{excl}(\mathcal{F})$ if and only if $G\setminus (S\setminus V(\textsf{Compass}_{\mathfrak{R}'}(W')))\in\mathbf{excl}(\mathcal{F})$.
Moreover, $\funref{label_interpersonal}(a,\ell_{\mathcal{F}},q,k)=\mathcal{O}(k\cdot \funref{label_conversational}(a,\ell_{\mathcal{F}},q) ).$
\end{lemma}

\begin{proof}
	Let $z=\funref{label_conversational}(a,\ell_{\mathcal{F}},q),$ $r$ be the smallest odd integer that is not smaller than $(k+1)\cdot (z+2)+q,$ and $\funref{label_interpersonal}(a,\ell_{\mathcal{F}},q,k)=r.$
	We also set $d:= \funref{label_unbelievability}(a,\ell_{\mathcal{F}}).$
	Let $G$ be a graph, $A$ be a subset of $V(G)$ of size at most $a,$ and $(W,\mathfrak{R})$ be a regular flatness pair of $G\setminus A$ of height at least $\funref{label_interpersonal}(a,\ell_{\mathcal{F}},q,k)$ that is $d$-homogeneous
	with respect to $\binom{A}{\leq a}.$

	For every $i\in[r],$ we denote by $P_{i}$ (resp. $Q_{i}$) the $i$-th vertical (resp. horizontal) path of $W.$ Let $z'=\frac{z+1}{2}$
	and observe that, since the images of the function $\funref{label_conversational}$ of \autoref{label_panlatinismo} are always odd numbers, then $z'\in \mathbb{N}.$
	We also define, for every $i\in[k+1]$ the graph
	\[B_{i}:=\bigcup_{j\in [z'-1]}P_{j+(i-1)\cdot z'} \cup\bigcup_{j\in [z']} P_{j+(k+1-i)\cdot z'}\cup \bigcup_{j\in [z'-1]}Q_{j+(i-1)\cdot z'} \cup \bigcup_{j\in [z']} Q_{j+(k+1-i)\cdot z'}.\]
	For every $i\in[k+1],$ we define ${W}_{i}$ to be the graph obtained from $B_{i}$
	after repeatedly removing from $B_{i}$ all vertices of degree one (see \autoref{label_desasosegado} for an example).
	Since $z=2z'-1,$   for every $i\in[k+1]$  ${W}_{i}$ is a $z$-subwall of $W.$
	For every $i\in[k+1],$  we set $L^{i}_\mathsf{inn}$ to be the inner layer of $W_{i}.$
	Notice that $L^{i}_\mathsf{inn},$ for $i\in[k+1],$ and $D(W^{(q)})$ are $\mathfrak{R}$-normal cycles of $\mathsf{Compass}_{\mathfrak{R}}(W).$
We stress that, by the definition of $W_i$'s, for every $i\in[k+1]$, $V(W_i)$ does not intersect the vertices $(j\cdot z')$-th layer of $W$.
This implies that, for every $i\in[k+1]$ and for every   $(W,\mathfrak{R})$-canonical partition $\tilde{\mathcal{Q}}$ of $G\setminus A$, there is no bag of $\tilde{\mathcal{Q}}$ that intersects both a vertex of $L^{i}_\mathsf{inn}$ and of $\cupall\mathsf{Influence}_{\mathfrak{R}}(W_i)$.
Intuitively, this means that, in~\autoref{label_desasosegado}, there is no bag of any  $(W,\mathfrak{R})$-canonical partition of $G\setminus A$ that contains both a vertex of the innermost orange layer and a vertex of a flap containing vertices of the outermost blue layer.

	\begin{figure}[ht]
		\centering
		\scalebox{1}[1]{\includegraphics[width=11cm]{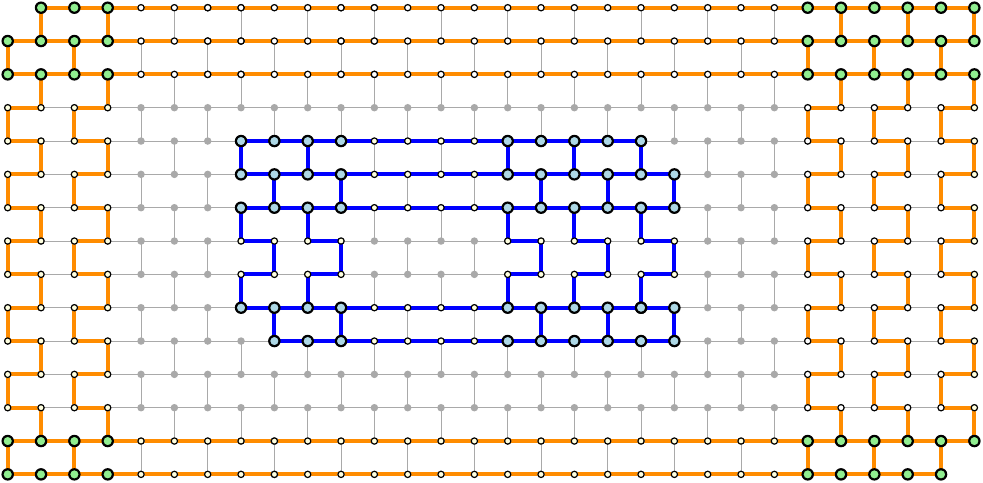}}
		\caption{A 15-wall and the 5-walls $W_1$ and $W_2$ as in the proof of  \autoref{lemma_biiid_irr}, depicted in orange and blue, respectively. The white vertices are subdivision vertices of the walls $W_1$ and $W_2.$}
		\label{label_desasosegado}
	\end{figure}
Let $\tilde{\mathcal{Q}}$ be a $(W,\mathfrak{R})$-canonical partition of $G\setminus A$.
	Let $(\hat{W},\hat{\mathfrak{R}})$ be a $W^{(q)}$-tilt of $(W,\mathfrak{R}).$
We set $Y:=V(\textsf{Compass}_{\hat{\mathfrak{R}}}(\hat{W}).$
	By the definition of a tilt of a flatness pair, it holds that $Y$ is a subgraph of $\cupall\mathsf{Influence}_{\mathfrak{R}} (W^{(q)}).$
	Moreover, for every $i\in[k],$ the fact that $r\geq (k+1)\cdot (z+2)+q$ implies that
	$\cupall\mathsf{Influence}_{\mathfrak{R}} (W^{(q)})$ is a subgraph of $\cupall\mathsf{Influence}_{\mathfrak{R}} (L_{\mathsf{inn}}^{i}).$
	Hence, for every $i\in[k+1],$ we have that
$Y\mbox{~is a subgraph of~}\cupall\mathsf{Influence}_{\mathfrak{R}}(L_{\mathsf{inn}}^{i}).$

Let $S\subseteq V(G)$ that intersects at most $k$ internal bags of $\tilde{\mathcal{Q}}$.
We aim to prove that $G\setminus S\in \mathbf{excl}(\mathcal{F})$ if and only if $G\setminus (S\setminus Y) \in  \mathbf{excl}(\mathcal{F}).$
It is easy to see that if $G\setminus (S\setminus Y) \in  \mathbf{excl}(\mathcal{F})$, then $G\setminus S\in \mathbf{excl}(\mathcal{F})$.
So, it remains to prove that, if $G\setminus S\in \mathbf{excl}(\mathcal{F}),$ then $G\setminus (S\setminus Y) \in  \mathbf{excl}(\mathcal{F}).$

	Suppose, towards a contradiction, that $\mathcal{F}\prem {G\setminus (S\setminus Y)}$
	and let $H$ be a graph in $\mathcal{F}$ that is a minor of ${G\setminus (S\setminus Y)}.$
	For every $i\in[k+1],$ let $(W_{i}',\mathfrak{R}_{i}')$ be a flatness pair of $G\setminus A$ that is a $W_{i}$-tilt of $(W,\mathfrak{R})$ ({which exists due to \autoref{label_prosperously}})
	and keep in mind that $W_i '$ has height $z.$
	Also, note that, for every $i\in[k+1],$  $L_{\mathsf{inn}}^i$ is the inner layer of $W_{i}'$ and therefore it is
	an $\mathfrak{R}_i^\prime$-normal cycle of $\mathsf{Compass}_{\mathfrak{R}_i^\prime}(W_i^\prime).$
	Additionally, for every $i\in[k+1],$ $(W_{i}',\mathfrak{R}_{i}')$ is
	$d$-homogeneous with respect to $2^A$ due to \autoref{label_convenciones}, and, due to \autoref{label_expressionism}, $(W_{i}',\mathfrak{R}_{i}')$ is also regular.

	For every $i\in[k+1],$ we set
	$D_{i} := V(\textsf{Compass}_{\mathfrak{R}_i '}(W_i '))\setminus V(\cupall\mathsf{Influence}_{\mathfrak{R}_i '}(L_\mathsf{inn}^{i}))$
	and observe that, since every flap in $\mathsf{Flaps}_{\mathfrak{R}_i '}(W_i ')$ belongs to $\mathsf{Influence}_{\mathfrak{R}_{i-1} '}(W_{i-1} ')$ and $\mathsf{Influence}_{\mathfrak{R}_{i-1} '}(W_{i-1} ')\subseteq \mathsf{Influence}_{\mathfrak{R}_i '}(L_\mathsf{inn}^{i}),$ the vertex sets $D_i,$ $i\in[k+1]$ are pairwise disjoint and no bag of $\tilde{\mathcal{Q}}$ intersects both $D_i$ and $D_j$, for every $i,j\in[k+1], i\neq j.$
	Therefore, since $S$ intersects at most $k$ internal bags of $\tilde{\mathcal{Q}}$, there exists a $j\in [k+1]$ such that
	$S\cap D_j=\emptyset.$
	We set $S_\mathsf{in}:= S\cap V(\textsf{Compass}_{\mathfrak{R}_j '}(W_j '))$ and $S_{\mathsf{out}}:=S\setminus S_\mathsf{in}$ and observe that $S_\mathsf{in}\subseteq V(\cupall\mathsf{Influence}_{\mathfrak{R}_j '}(L_\mathsf{inn}^j)),$ while $S_{\mathsf{out}}\cap V(\textsf{Compass}_{\mathfrak{R}_j '}(W_j '))=\emptyset.$
	It is also easy to see that $S\cap Y \subseteq S_\mathsf{in}$ and therefore $S_{\mathsf{out}}\subseteq S\setminus Y.$
	
Also, we set $A':= A\setminus S_{\mathsf{out}}.$
Since $(W_i ',\mathfrak{R}_i ')$ is $d$-homogeneous with respect to $\binom{A}{\leq a}$
and $|A'|= |A\setminus S_{\mathsf{out}}|\leq a$, and therefore
$A'\in \binom{A}{\leq a}$,
it follows that $(W_i ',\mathfrak{R}_i ')$ is $d$-homogeneous with respect to $A'.$
Since $S_{\mathsf{out}}\cap V(\textsf{Compass}_{\mathfrak{R}_j '}(W_j '))=\emptyset,$
by removing the vertices of $S_{\mathsf{out}}$ from $G,$
	we obtain a flatness pair $(W_j ',\mathfrak{R}_j '')$ of $(G\setminus S_{\mathsf{out}}) \setminus A',$
	where, if $\mathfrak{R}_j ' = (X,Y,P,C,\Gamma,\sigma,\pi),$ then $\mathfrak{R}_j ''$ is obtained from $\mathfrak{R}_j '$ by removing the set $S_{\mathsf{out}}$ from $X.$
	Notice that the $\mathfrak{R}_j''$-compass and the $\mathfrak{R}_j '$-compass of $W_j '$ are identical,
	which implies that $(W_j ',\mathfrak{R}_j '')$ is a regular
	flatness pair of $(G\setminus S_{\mathsf{out}}) \setminus A'$ that is $d$-homogeneous with respect to $A'.$
	Also, recall that $W_j '$ has height $z = \funref{label_conversational}(a,\ell_{\mathcal{F}},q).$

	We are now in position to apply \autoref{label_panlatinismo} on $G\setminus S_{\mathsf{out}},$
	$A',$ and $({W}_{j}',\mathfrak{R}_{j}''),$
	which implies that for every $W_j^{\prime (q)}$-tilt $(\tilde{W},\tilde{\mathfrak{R}})$
	of $({W}_{j}',\mathfrak{R}_{j}''),$
	the vertex set of $\mathsf{Compass}_{\tilde{\mathfrak{R}}}(\tilde{W})$ is $\ell_{\mathcal{F}}$-irrelevant.
	Observe that since $H$ is a minor of $G\setminus (S\setminus Y)$ and $S_{\mathsf{out}}\subseteq S\setminus Y,$ $H$ is also a minor of $G\setminus S_{\mathsf{out}}.$
	This, in addition to the fact that 
	the vertex set of $\mathsf{Compass}_{\tilde{\mathfrak{R}}}(\tilde{W})$ is $\ell_{\mathcal{F}}$-irrelevant
	and $H$ has detail at most $\ell_{\mathcal{F}},$
	implies that $H$ is also a minor of $G\setminus (S_{\mathsf{out}} \cup V(\textsf{Compass}_{\tilde{\mathfrak{R}}}(\tilde{W}))).$
	
	Also, it is easy to observe that
$\cupall\mathsf{Influence}_{\mathfrak{R}_{j}'}(L^{j}_\mathsf{inn})\mbox{~is a subgraph of~}\textsf{Compass}_{\tilde{\mathfrak{R}}}(\tilde{W}).$
	Using the fact that
	$S_\mathsf{in}\subseteq V(\cupall\mathsf{Influence}_{\mathfrak{R}_j '}(L_\mathsf{inn}^j))\subseteq V(\textsf{Compass}_{\tilde{\mathfrak{R}}}(\tilde{W})),$
	we derive that $H$ is a minor of $G\setminus (S\cup V(\textsf{Compass}_{\tilde{\mathfrak{R}}}(\tilde{W}))).$
	Therefore $H$ is a minor of $G\setminus S,$
	which contradicts the initial assumption that $G\setminus S\in \mathbf{excl}(\mathcal{F}).$
\end{proof}

The following corollary is an immediate consequence of~\autoref{lemma_biiid_irr}, since every set $S\subseteq V(G)$ of size at most $k$ clearly intersects at most $k$ internal bags of every $(W,\mathfrak{R})$-canonical partition of $G\setminus A.$

\begin{corollary}\label{label_acquaintances}
	There exists a function $\funref{label_interpersonal}: \mathbb{N}^{4}\to \mathbb{N},$
	{whose images are odd numbers,} such that for every $a,k\in\mathbb{N},$ every odd $q\in\mathbb{N}_{\geq 3},$ and every graph $G,$
	if $A$ is a subset of $V(G)$ of size at most $a$
	and
	$(W,\mathfrak{R})$ is a regular flatness pair of $G\setminus A$ of height at least $\funref{label_interpersonal}(a,\ell_{\mathcal{F}},q,k)$ that is  $\funref{label_unbelievability}(a,\ell_{\mathcal{F}})$-homogeneous
	with respect to $2^A,$
	then for every  $W^{(q)}$-tilt $(\hat{W},\hat{\mathfrak{R}})$ of $(W,\mathfrak{R}),$
	it holds that  $G\in \mathcal{A}_{k} (\mathbf{excl}(\mathcal{F}))$ if and only if $G\setminus V(\textsf{Compass}_{\hat{\mathfrak{R}}}(\hat{W}))\in  \mathcal{A}_{k} (\mathbf{excl}(\mathcal{F})).$
	Moreover, $\funref{label_interpersonal}(a,\ell_{\mathcal{F}},q,k)=\mathcal{O}(k\cdot \funref{label_conversational}(a,\ell_{\mathcal{F}},q) ).$
\end{corollary}

We will use a ``light'' version of \autoref{label_acquaintances}, namely for $q=3.$
Combining it with \autoref{label_surreptitiously} we obtain a central vertex $v\in V(G)$ of $W^{(3)}$ such that $G\in \mathcal{A}_{k} (\mathbf{excl}(\mathcal{F}))$ if and only if $G\setminus v\in  \mathcal{A}_{k} (\mathbf{excl}(\mathcal{F})).$

\begin{corollary}\label{label_descalabrado}
	Let $a,k\in\mathbb{N},$ $G$ be a graph,
	$A$ be a subset of $V(G)$ of size at most $a,$
	and
	$(W,\mathfrak{R})$ be a  regular flatness pair of $G\setminus A$ of height at least $\funref{label_interpersonal}(a,\ell_{\mathcal{F}},3,k)$  that is $\funref{label_unbelievability}(a,\ell_{\mathcal{F}})$-homogeneous with respect to $2^A.$
	If $v$ is a central vertex of  $W,$ then $G\in \mathcal{A}_{k} (\mathbf{excl}(\mathcal{F}))$ if and only if $G\setminus v\in  \mathcal{A}_{k} (\mathbf{excl}(\mathcal{F})).$
\end{corollary}

\subsection{Bounding the treewidth of an obstruction}\label{label_informaremos}
We conclude this section by proving the following result that will be useful in the proof of \autoref{label_miseryfrightened}.

\begin{lemma}\label{label_transfiguration}
	There exists a function $\newfun{label_explicacions}:\mathbb{N}^2\to\mathbb{N}$ such that if
	$\mathcal{F}$ is a finite collection of graphs
	and
	$G\in\mathbf{obs}(\mathcal{A}_k (\mathbf{excl}(\mathcal{F}))),$
	then $\mathsf{tw}(G)\leq \funref{label_explicacions}(a_{\mathcal{F}},s_{\mathcal{F}},k).$
	Moreover, $\funref{label_explicacions}(a, s,k)=2^{\log( k \cdot c) \cdot2^{k^{a -1}\cdot   2^{\mathcal{O}(s^2\log s)}}},$
	where $a=a_{\mathcal{F}},$ $s= s_{\mathcal{F}},$ $c=f_{\mathsf{ul}} (s^2),$
	and $f_{\mathsf{ul}}$ is the function of the Unique Linkage Theorem.
\end{lemma}

In order to prove \autoref{label_transfiguration}, we also need the following result of Kawarabayashi  and  Kobayashi~\cite{KawarabayashiK20line}, that
provides a \textsl{linear} relation between the treewidth and the height of a largest wall in a minor-free graph.

\begin{proposition}\label{label_wahrscheinlichkeitssatz}
	There exists a function $\newfun{label_aromatiserez}:\mathbb{N}\to \mathbb{N}$ such that, for every $t,r\in \mathbb{N}$ and every
	graph $G$ that does not contain $K_{t}$ as a minor,
	if $\mathsf{tw}(G)\geq \funref{label_aromatiserez}(t)\cdot r$ then $G$ contains an $r$-wall.
	In particular, one may choose $\funref{label_aromatiserez}(t)=2^{\mathcal{O}(t^{2}\log t)}.$
\end{proposition}

We are now in position to prove \autoref{label_transfiguration}.

\begin{proof}[Proof of \autoref{label_transfiguration}.]
	For simplicity, we use $s,a,$ and $\ell$ instead of $s_{\mathcal{F}}, a_{\mathcal{F}},$ and $\ell_{\mathcal{F}},$ respectively.
	Keep in mind that $\ell =\mathcal{O}(s^2).$
	We set $\tilde{a} = a-1,$
	\begin{align*}
		b:=                                  & \ \funref{label_interpersonal}(\tilde{a},\ell, 3, k),              &
		d :=                                 & \ \funref{label_unbelievability}(\funref{label_entretenerme}(s),\ell), &
		z: =                                 & \ \funref{label_entretenerme}(s)+k+1,                                 \\
		m:=                                  & \ \funref{@prepossession}(a,s,k+1),                         &
		x:=                                  & \ \funref{@proclamation}(a,s,k+1),                     &
		l :=                                 & \ z\cdot x,                                                           \\
		p:=                                  & \ \funref{@engrandecerla}(a,s,k+1),                             &
		h:=                                  & \ \funref{label_apreciadores}(l,b,p),                            &
		r:=                                  & \ \mathsf{odd}(\max\{m,h\}),                                             \\
		w:=                                  & \ \funref{label_complications}(r,z, \tilde{a},d),                   &
		q:=                                  & \funref{label_everyevening}(s)\cdot w, \mbox{ and }              &
		\funref{label_explicacions}(a,s,k):= & \ \funref{label_aromatiserez}(s)\cdot q+k+1.
	\end{align*}
	It is easy to verify that
	$\funref{label_explicacions}(a, s,k)=
		2^{\log( k \cdot c) \cdot2^{k^{a -1}\cdot   2^{\mathcal{O}(s^2\log s)}}},$
	where $c=f_{\mathsf{ul}}(s^2)$ and $f_{\mathsf{ul}}$ is the function of the Unique Linkage Theorem.

	Suppose towards a contradiction that $\mathsf{tw}(G)> \funref{label_explicacions}(a,s,k).$
	Since $G\in \mathbf{obs}(\mathcal{A}_k (\mathbf{excl}(\mathcal{F}))),$
	there exists a set $S\subseteq V(G)$ such that $|S|=k+1$
	and $G\setminus S\in \mathbf{excl}(\mathcal{F}).$
	Therefore, $G\setminus S$ does not contain $K_s$ as a minor
	and $\mathsf{tw}(G\setminus S)> \funref{label_aromatiserez}(s)\cdot q,$
	so, due to \autoref{label_wahrscheinlichkeitssatz},
	$G\setminus S$ contains  a $q$-wall $W^\bullet.$
	Since $q= \funref{label_everyevening}(s)\cdot w$
	and $G\setminus S$ does not contain $K_s$ as a minor,
	by \autoref{label_aldobrandesco}
	there is set $A\subseteq V(G\setminus S)$
	of size at most $\funref{label_entretenerme}(s)$
	and a flatness pair $(\hat{W}, \hat{\mathfrak{R}})$
	of $G\setminus (S\cup A)$ of height $w.$
	By using \autoref{label_aberenjenado},
	we can obtain a regular flatness pair $(\hat{W}',\hat{\mathfrak{R}}')$
	of $G\setminus (S\cup A)$ of height $w.$
	Since $|S\cup A|\leq z$ and $(\hat{W}', \hat{\mathfrak{R}}')$
	is of height $w=\funref{label_complications}(r,z, \tilde{a},d),$
	by \autoref{label_transportado} there exists a subwall $\hat{W}'$ of $\hat{W}$
	of height $r$ such that every $\hat{W}'$-tilt of $(\hat{W}, \hat{\mathfrak{R}})$
	is $d$-homogeneous with respect to $\binom{S\cup A}{\leq \tilde{a}}.$
	We consider a $\hat{W}'$-tilt of $(\hat{W}, \hat{\mathfrak{R}}),$
	say $(W,\mathfrak{R}),$ and keep in mind that it is a flatness pair of $G\setminus (S\cup A)$
	of height $r$ that is $d$-homogeneous with respect to $\binom{S\cup A}{\leq \tilde{a}}.$
	Moreover, by \autoref{label_expressionism}, $(W,\mathfrak{R})$ is regular.

	Let $\tilde{\mathcal{Q}}$ be a $(W,\mathfrak{R})$-canonical partition of $G\setminus (S\cup A).$
	Let $$A^\star = \{v\in S\cup A \mid  v\text{ is adjacent, in }G,\text{ to vertices of at least $x$ $p$-internal bags of }\tilde{\mathcal{Q}}\}.$$
	We consider a family
	$\mathcal{W}=\{{W}_{1}, \ldots, {W}_{l}\}$ of $l$ $b$-subwalls of $W$
	such that for every $i \in [l],$ $\cupall\mathsf{Influence}_{\mathfrak{R}}(W_i)$ is a subgraph of
	$\cupall \{Z\mid Z \text{ is a $p$-internal bag of }\tilde{\mathcal{Q}}\}$ and
	for every $i,j\in[l]$ with $i\neq j,$ there is no internal bag $Z\in \tilde{\mathcal{Q}}$
	that contains vertices of both $V(\cupall\mathsf{Influence}_\mathfrak{R} (W_i))$ and
	$V(\cupall\mathsf{Influence}_\mathfrak{R} (W_j)).$
	The existence of $\mathcal{W}$ follows from the fact that
	$r\geq h = \funref{label_apreciadores}(l,b,p)$ and \autoref{label_overestimation}.
	Notice that
	the set $N_G((S\cup A)\setminus A^\star)$ intersects the vertex set at most
	$(x-1)\cdot |(S\cup A)\setminus A^\star|\leq (x-1)\cdot (\funref{label_entretenerme}(s)+k+1)< l$
	$p$-internal bags of $\tilde{\mathcal{Q}}.$
	Hence, taking into account the aforementioned properties of the walls $W_1,\ldots, W_{l},$
	there exists an $i\in [l]$ such that no
	vertex in $(S\cup A)\setminus A^\star$ is
	adjacent to vertices of $\cupall\mathsf{Influence}_{\mathfrak{R}}({W}_{i}).$
	In other words, if there exists a vertex $v\in V(\cupall\mathsf{Influence}_{\mathfrak{R}}({W}_{i}))$
	that is adjacent, in $G,$ to a vertex $u\in S\cup A,$ then $u\in A^\star.$

	Let $v$ be a central vertex of $W_i.$
	Since $G\in \mathbf{obs}(\mathcal{A}_k(\mathbf{excl}(\mathcal{F}))),$ it holds that $G\setminus v\in \mathcal{A}_k(\mathbf{excl}(\mathcal{F})).$
	Thus, there is a set $S'\subseteq V(G\setminus v)$ of size $k$ such that $G\setminus (S'\cup \{v\})\in \mathbf{excl}(\mathcal{F}).$
	We set $A^\star_\mathrm{hit}=A^\star\cap S'$ and $A^\star_\mathrm{free}=A^\star\setminus S'.$

	We claim that $|A^\star_\mathrm{free}|\leq \tilde{a}.$ To prove this,
	suppose to the contrary that $|A^\star_\mathrm{free}|> \tilde{a},$ or equivalently $|A^\star_\mathrm{free}|\geq a.$
	We have that $(W,\mathfrak{R})$ is a flatness pair of $G\setminus (S\cup A)$ of height $r\geq m$ and
	$A^\star_\mathrm{free}$ is a subset of $S\cup A$ such that every $v\in A^\star_\mathrm{free}$ is
	adjacent, in $G,$ to vertices of
	at least $x$ $p$-internal bags of $\tilde{\mathcal{Q}}$ and
	$|A^\star_\mathrm{free}|\geq a,$
	where $m= \funref{@prepossession}(a,s,k+1),$
	$x=\funref{@proclamation}(a,s,k+1),$ and
	$p=\funref{@engrandecerla}(a,s,k+1).$
	Therefore, by \autoref{label_significativos} applied to
	$\mathcal{F},$ $G,$ $k+1,$ $S\cup A,$ $(W,\mathfrak{R}),$ $\tilde{\mathcal{Q}},$ and $A^\star_\mathrm{free},$
	it holds that every set $S\subseteq V(G)$ of size at most $k+1$
	such that $G\setminus S\in \mathbf{excl}(\mathcal{F})$ intersects $A^\star_\mathrm{free}.$
	As $S'\cup \{v\}$ has size $k+1$ and
	$G\setminus (S'\cup \{v\})\in \mathbf{excl}(\mathcal{F}),$ it follows that
	$S'$ intersects $A^\star_\mathrm{free},$ a contradiction
	to the definition of $A^\star_\mathrm{free}.$

	Let $(\tilde{W}_i, \tilde{\mathfrak{R}}_i)$ be a $W_i$-tilt of $(W,\mathfrak{R}),$
	which exists due to \autoref{label_prosperously}.
	Since $(\tilde{W}_i,\tilde{\mathfrak{R}}_i)$ is a $W_i$-tilt of $(W,\mathfrak{R}),$
	$\mathsf{Compass}_{\tilde{\mathfrak{R}}_i} (\tilde{W}_i)$ is a subgraph
	of $\cupall \mathsf{Influence}_\mathfrak{R} (W_i).$
	Consequently, if $A_i$ is the set of vertices in $S\cup A$ that are adjacent to
	vertices of $\mathsf{Compass}_{\tilde{\mathfrak{R}}_i} (\tilde{W}_i)$ in $G,$ then $A_i\subseteq A^\star.$
	Let $A'=A_i \cap A^\star_\mathrm{free}$ and notice that, if we remove the vertices of $A^\star_\mathrm{hit}$ from $G,$
	we can obtain a flatness pair $(\tilde{W}_i, \tilde{\mathfrak{R}}_i ')$ of
	$(G\setminus A^\star_\mathrm{hit})\setminus A',$ where, if $\tilde{\mathfrak{R}}_i = (X,Y,P,C,\Gamma,\sigma,\pi),$ then $\tilde{\mathfrak{R}}_i'$ is obtained from $\tilde{\mathfrak{R}}_i$ by removing the set $A^\star_\mathrm{hit}$ from $X.$
	Notice that
	$\mathsf{Compass}_{\tilde{\mathfrak{R}}_{i}}(\tilde{W}_{i})=\textsf{Compass}_{\tilde{\mathfrak{R}}_i '}(\tilde{W}_i).$
	Moreover, since $(W, \mathfrak{R})$ is a flatness pair
	of $G\setminus (S\cup A)$ that is $d$-homogeneous with respect to $\binom{S\cup A}{\leq \tilde{a}}$ and $\tilde{W}_i$ is a tilt of a subwall of $W,$
	by \autoref{label_convenciones}
	$(\tilde{W}_{i},\tilde{\mathfrak{R}}_{i}')$ is also a flatness pair
	of $G\setminus (S\cup A)$ that is $d$-homogeneous with respect to $\binom{S\cup A}{\leq \tilde{a}}.$
	The latter, together with the fact that $|A'|\leq|A^\star_\mathrm{free}|\leq \tilde{a},$ imply that
	$(\tilde{W}_{i},\tilde{\mathfrak{R}}_{i}')$ is a flatness pair
	of $(G\setminus A^\star_\mathrm{hit})\setminus A'$ that is $d$-homogeneous with respect to $2^{A'}$ (as $2^{A'}\subseteq \binom{S\cup A}{\leq \tilde{a}}$).
	Also, since $(W,\mathfrak{R})$ is regular,
	by \autoref{label_expressionism} we have that $(\tilde{W}_{i},\tilde{\mathfrak{R}}_{i})$ and therefore $(\tilde{W}_{i},\tilde{\mathfrak{R}}_{i}')$ are regular as well.
	Since $|A'|\leq \tilde{a}$ and the height of $(\tilde{W}_{i},\tilde{\mathfrak{R}}_{i}')$ is
	$b= \funref{label_interpersonal}(\tilde{a},\ell, 3, k),$
	by \autoref{label_descalabrado} applied to $\tilde{a},$ $k-|A^\star_\mathrm{hit}|,$ $G\setminus A^\star_\mathrm{hit},$ $A',$ and
	$(\tilde{W}_{i},\tilde{\mathfrak{R}}_{i}'),$ we conclude that $G\in \mathcal{A}_k (\mathbf{excl}(\mathcal{F})),$
	a contradiction to the hypothesis that $G\in \mathbf{obs}(\mathcal{A}_k (\mathbf{excl}(\mathcal{F}))).$
\end{proof}

\section{From small treewidth to small order}
\label{label_agathyrsians}

In this section we aim to prove the following result that provides an upper bound, in terms of treewidth, of the order of a graph in $\mathbf{obs}(\mathcal{A}_k (\mathbf{excl}(\mathcal{F})).$
\autoref{label_miseryfrightened} follows from \autoref{label_transfiguration} and \autoref{label_substanceless}.

\begin{lemma}\label{label_substanceless}
	There exists a function $\newfun{label_schwankenden}:\mathbb{N}^2\to\mathbb{N}$ such that if
	$\mathcal{F}$ is a finite collection of graphs and
	$G$ is a graph in $\mathbf{obs}(\mathcal{A}_k (\mathbf{excl}(\mathcal{F}))$ of treewidth  $\mathsf{tw},$
	then $|V(G)|\leq \funref{label_schwankenden}(\mathsf{tw},\ell_{\mathcal{F}}).$
	Moreover, $\funref{label_schwankenden}(\mathsf{tw},\ell)=2^{2^{\mathcal{O}(q(\ell)\cdot  \mathsf{tw}^{{3}} \cdot \log \mathsf{tw})}},$ where $\ell=\ell_{\mathcal{F}},$ $q(\ell)=2^{2^{2^{c^{2^{2^{\mathcal{O}(\ell \cdot \log \ell)}}}}}},$  {$c= f_{\mathsf{ul}} (\ell)$}, and $f_{\mathsf{ul}}$ is the function of the Unique Linkage Theorem.
\end{lemma}

In order to prove \autoref{label_substanceless}, in \autoref{label_profesionales} we provide some useful definitions and preliminary results on boundaried graphs,
in \autoref{label_hyperbolical} we define tree decompositions of boundaried graphs, and, finally, in \autoref{label_commemorative} we prove \autoref{label_substanceless}.

\subsection{Representatives of boundaried graphs}\label{label_profesionales}

\paragraph{Minors of boundaried graphs.} We say that a $t$-boundaried graph $\mathbf{G}_{1}=(G_1,B_1,\rho_1)$ is a \emph{minor} of a $t$-boundaried graph $\mathbf{G}_{2}=(G_2,B_2,\rho_2),$ denoted by $\mathbf{G} _{1}\prem\mathbf{G}_{2},$
if there is a sequence of  removals of non-boundary vertices, edge removals, and edge contractions in $G_2,$ not allowing contractions of edges with both endpoints  in $B_{2},$ that  transforms $\mathbf{G}_{2}$ to a boundaried graph that is isomorphic to $\mathbf{G}_{1}$ (during edge contractions, boundary vertices persist). Note that this extends the usual definition of minors in graphs without boundary.

\paragraph{Compatible boundaried graphs.}
We say that two boundaried graphs $\mathbf{G}_{1}=(G_1,B_1,\rho_1)$ and $\mathbf{G}_{2}=(G_2,B_2,\rho_2)$ are \emph{compatible} if $\rho_{2}^{-1}\circ \rho_{1}$  is an isomorphism from $G_{1}[B_{1}]$ to $G_{2}[B_{2}].$
Given two compatible boundaried graphs $\mathbf{G}_{1}=(G_1,B_1,\rho_1)$
and  $\mathbf{G}_{2}=(G_2,B_2,\rho_2),$  we
define $\mathbf{G}_{1}\oplus\mathbf{G}_{2}$ as the graph obtained
if we take the disjoint union of $G_{1}$ and $G_{2}$
and, for every $i\in[|B_{1}|],$ we identify vertices $\rho_{1}^{-1}(i)$ and $\rho_{2}^{-1}(i).$

\paragraph{Equivalent boundaried graphs and representatives.}

Given an $h\in\mathbb{N}$ and two boundaried graphs $\mathbf{G}_1$ and $\mathbf{G}_2,$ we say that $\mathbf{G}_1\leqslant_h \mathbf{G}_2$ if $\mathbf{G}_1$ and $\mathbf{G}_2$ are compatible and, for every graph $H$ of detail at most $h$ and every boundaried graph $\mathbf{F}$ that is compatible with $\mathbf{G}_1$ (and therefore with $\mathbf{G}_2$ as well), it holds that
$$ H\preceq_\mathsf{m}  \mathbf{F}\oplus \mathbf{G}_1 \Rightarrow H\preceq_\mathsf{m}\mathbf{F}\oplus \mathbf{G}_2.$$

It is easy to observe the following.
\begin{observation}\label{label_antipathetic}
	If $\mathbf{G}_1\preceq_\mathsf{m} \mathbf{G}_2,$ then $\mathbf{G}_1 \leqslant_{h} \mathbf{G}_2$ for every $h\in \mathbb{N}.$
\end{observation}

Given $h,t\in \mathbb{N},$ we say that two $t$-boundaried graphs $\mathbf{G}_{1}$ and $\mathbf{G}_{2}$ are \emph{$h$-equivalent}, denoted by $\mathbf{G}_{1}\equiv_{h,t} \mathbf{G}_{2},$
if $\mathbf{G}_1\leqslant_h \mathbf{G}_2$ and $\mathbf{G}_2\leqslant_h \mathbf{G}_1.$
Note that  $\equiv_{h,t}$ is an equivalence relation on ${\mathcal{B}}$ and that only boundaried graphs with the same boundary size can be $h$-equivalent.
A minimum-order
element of an  equivalence class of  $\equiv_{h,t}$
is called \emph{representative}
of $\equiv_{h,t}.$
For every $t\in \mathbb{N},$ we define a \emph{set of $t$-representatives} for $\equiv_{h,t}$ to be a collection containing a minimum-order  representative for each equivalence class of $\equiv_{h,t}.$
Given $t,h\in \mathbb{N},$ we denote by  $\mathcal{R}_{h}^{(t)}$  a set of $t$-representatives  for $\equiv_{h,t}.$

We need the following result from \cite{BasteST20acom}.

\begin{proposition}
	\label{label_desenfadadamente}
	There exists a function $\newfun{label_personnellement}:\mathbb{N}\to\mathbb{N}$
	such that for every $t\in\mathbb{N}_{\geq 1},$  $|\mathcal{R}_{h}^{(t)}| \leq 2^{\funref{label_personnellement}(h) \cdot t \cdot \log t}.$ In particular, the relation $\equiv_{h,t}$ partitions  ${\mathcal{B}}^{(t)}$
	into $2^{\funref{label_personnellement}(h) \cdot t \cdot \log t}$ equivalence classes.  Moreover, it holds that {$\funref{label_personnellement}(h) =2^{2^{2^{c^{2^{2^{\mathcal{O}(h \cdot \log h)}}}}}}$}, where $c=f_{\mathsf{ul}} (h)$ and $f_{\mathsf{ul}}$ is the function of the {Unique Linkage Theorem}.
\end{proposition}

\paragraph{Characteristic of boundaried graphs.}
Let $t,h \in\mathbb{N}.$
We denote by $\mathcal{P}_{t,h}$ the set $\{(I,\mathbf{R})\mid I\in 2^{[t]}\mbox{ and }\mathbf{R}\in\mathcal{R}_h^{({t-|I|})}\}.$
Given a $k\in \mathbb{N}$ and a pair $(I,\mathbf{R})\in\mathcal{P}_{t,h},$ we define the function $\mathbf{p}_{k,I,\mathbf{R}}:\mathcal{B}^{(t)}\to [0,k+1]$ such that for every $t$-boundaried graph $\mathbf{G}=(G,B,\rho),$
$$\mathbf{p}_{k,I,\mathbf{R}} (\mathbf{G})= \min\{|S|\mid \mbox{$S\subseteq V(G)$ of size at most $k,$ $\rho(S\cap B)=I,$ and $\mathbf{G}\setminus S\leqslant_h \mathbf{R}$}\}.$$
If such a set $S$ does not exist,  we set $ \mathbf{p}_{k,I,\mathbf{R}} (\mathbf{G})=k+1.$
Also,  we define the \emph{$(k,h)$-characteristic function} of the $t$-boundaried graph $\mathbf{G}$ to be the function $\mathsf{char}_\mathbf{G}^{(k,h)}:\mathcal{P}_{t,h} \to [0,k+1]$ that maps every pair $(I,\mathbf{R})\in\mathcal{P}_{t,h}$  to the integer $\mathbf{p}_{k,I,\mathbf{R}} (\mathbf{G}).$
\smallskip

We now prove the following result.

\begin{lemma}\label{label_substantiality}
	For every $t,k,h\in\mathbb{N},$ if  $\mathbf{G}_1,\mathbf{G}_2$ are two $t$-boundaried graphs such that $\mathbf{G}_1 \prem \mathbf{G}_2,$ then, for every $(I,\mathbf{R})\in\mathcal{P}_{t,h},$ $\mathsf{char}_{\mathbf{G}_{1}}^{(k,h)}(I,\mathbf{R}) \leq \mathsf{char}_{\mathbf{G}_{2}}^{(k,h)}(I, \mathbf{R}).$
\end{lemma}

\begin{proof}
	Let $\mathbf{G}_1 = (G_1, B_1, \rho_1)$ and $\mathbf{G}_2=(G_2, B_2, \rho_2)$
	be two $t$-boundaried graphs such that $\mathbf{G}_1 \prem \mathbf{G}_2.$
	Also, let $(I,\mathbf{R})\in\mathcal{P}_{t,h}$ and $w= \mathbf{p}_{k,I,\mathbf{R}} (\mathbf{G}_2).$
	We will prove that $\mathbf{p}_{k,I,\mathbf{R}} (\mathbf{G}_1)\leq w.$
	In the case where $w=k+1,$
	the inequality holds trivially since, by definition, $\mathbf{p}_{k,I,\mathbf{R}} (\mathbf{G}_1)\in[0,k+1].$
	Also, $w=0$ only if $I=\emptyset,$
	in which case the empty set is a certificate that $\mathbf{p}_{k,I,\mathbf{R}} (\mathbf{G}_2) = 0$
	and since, due to \autoref{label_antipathetic}, $\mathbf{G}_1\preceq_\mathsf{m} \mathbf{G}_2$
	implies that $\mathbf{G}_1 \leqslant_{h} \mathbf{G}_2,$
	it holds that $\mathbf{G}_1\leqslant_h \mathbf{R}$
	and therefore $\mathbf{p}_{k,I,\mathbf{R}} (\mathbf{G}_1) = 0.$
	Thus, we can assume that $w\in[k].$
	Let $S$ be a subset of $V(G_2)$ of size $w$ such that
	$\rho_2(S\cap B_2)=I$ and $\mathbf{G}_2\setminus S\leqslant_h \mathbf{R}.$
	Let $S'$ be the subset of $V(G_1)$ obtained from $S$ after applying in $G_2$
	the operations that transform it to $G_1,$ i.e.,
	the subset of $V(G_1)$ that contains the resulting vertices after the contraction of edges
	with some endpoint in $S$ and the vertices of $S$ that
	are not removed or whose incident edges are not contracted
	while transforming $\mathbf{G}_2$ to $\mathbf{G}_1.$
	On the one hand, since $B_1 = B_2,$
	it holds that
	$\rho_1(S'\cap B_1)=\rho_2(S\cap B_2).$
	On the other hand, $\mathbf{G_1}\setminus S' \prem \mathbf{G_2} \setminus S$ and thus,
	due to \autoref{label_antipathetic}, $\mathbf{G_1}\setminus S' \leqslant_h \mathbf{G_2} \setminus S,$ which, in turn, implies that $\mathbf{G_1}\setminus S' \leqslant_h \mathbf{R}.$
	Hence, $\mathbf{p}_{k,I,\mathbf{R}} (\mathbf{G}_1)\leq w=\mathbf{p}_{k,I,\mathbf{R}} (\mathbf{G}_2).$
\end{proof}

Given $x,y\in \mathbb{N},$ we set $S_{x,y}$ to be the set of vectors of size $y$ whose elements are in $[x].$
Given an $n\in \mathbb{N}$ and two vectors $v=(v_1, \ldots, v_n)$ and $v'=(v_1 ', \ldots, v_n '),$ we say that $v\leq v'$ (resp. $v=v'$) if for every $i\in[n],$ $v_i \leq v_i '$ (resp. $v_i = v_i '$).
We say that a sequence $V= \langle v_1,\ldots, v_m \rangle$ of vectors is \emph{monotone} if for every $i,j\in [m]$ it holds that $i \leq j$ if and only if $v_i \leq v_j.$
It is easy to observe the following.

\begin{observation}\label{label_zusammenhang}
	Let $m,x,y\in \mathbb{N}.$
	For every monotone sequence $V = \langle v_1, \ldots, v_{m}\rangle$ of vectors  in $S_{x,y},$
	if $m\geq x\cdot y +1$ then there exists an $i\in [m-1]$ such that $v_i = v_{i+1}.$
\end{observation}

We next prove that for every chain of ``many enough'' boundaried graphs that are ordered under the minor relation, there exist two boundaried graphs that have the same characteristic function.
\begin{lemma}\label{label_rechtfertigen}
	There exists a function $\newfun{label_ecclesiastes}:\mathbb{N}^2\to \mathbb{N}$ such that for every $k,t,h\in\mathbb{N},$ if $\mathbf{G}_1, \ldots, \mathbf{G}_d$ is a sequence of boundaried graphs
	where, for every $i\in[d-1],$ $\mathbf{G}_i \prem \mathbf{G}_{i+1}$ and $d\geq \funref{label_ecclesiastes}(k,t,h),$ then there exists an $i\in[d-1]$ such that,
	for every $(I,\mathbf{R})\in\mathcal{P}_{t,h},$ it holds that $\mathsf{char}_{\mathbf{G}_i}^{(k,h)}(I,\mathbf{R})=\mathsf{char}_{\mathbf{G}_{i+1}}^{(k,h)}(I,\mathbf{R}).$
	Moreover, $\funref{label_ecclesiastes}(k,t,h) = k\cdot 2^{\mathcal{O}(\funref{label_personnellement}(h) \cdot t \cdot \log t)}.$
\end{lemma}

\begin{proof}
	We set $y= |\mathcal{P}_{t,h}|$
	and  $\funref{label_ecclesiastes}(k,t,h) = (k+2)\cdot y+1.$
	By \autoref{label_desenfadadamente}, we have that $y = 2^{\mathcal{O}(\funref{label_personnellement}(h) \cdot t \cdot \log t)}.$
	For every $\mathbf{G}_i,$ we set $c_i$ to be the vector corresponding to $\mathsf{char}^{(k,h)}_{\mathbf{G}_i}$ and observe that it is a vector of size $y$ whose coordinates are elements in $[0,k+1].$
	By \autoref{label_substantiality}, $\langle c_1, \ldots, c_d\rangle$ is a monotone sequence and thus, by \autoref{label_zusammenhang}, since $d \geq (k+2)\cdot y+1$ there is an $i\in[d-1]$ such that $c_i = c_{i+1}.$
\end{proof}

\subsection{Tree decompositions of boundaried graphs}\label{label_hyperbolical}
In this subsection, we deal with tree decompositions of boundaried graphs.

Given a tree $T$ and two distinct vertices $a,b$ of $V(T),$ we denote by $aTb$ the unique $(a,b)$-path in $T.$
Given a graph $G$ and two sets $X,Y\subseteq V(G),$ a \emph{collection of $t$ vertex-disjoint paths between $X$ and $Y$} is a set of $t$ paths $P_1,\ldots, P_t,$ where, for every $i\in[t],$ one endpoint of $P_i$ is in $X$ and the other is in $Y$ and for every $i,j\in[t], i\neq j,$ $V(P_i)\cap V(P_j)= \emptyset.$

\paragraph{Rooted trees.}
A \emph{rooted tree} is a pair $(T,r),$ where $T$ is a tree and $r\in V(T).$ We call $r$ the \emph{root} of $T.$
Given two vertices $a,b$ of $T,$ we write $a\leq_{(T,r)} b$ to denote that $a\in V(rTb)$ and, in this case, we say that $b$ is a \emph{descendant} of $a$ in $(T,r).$
Given some $q\in V(T),$ we denote the \emph{set of descendants} of $q$ in $(T,r)$ as $\mathsf{desc}_{T,r} (q).$
The \emph{children} of a vertex $q\in V(T)$ in $(T,r)$ are the descendants of $q$ in $(T,r)$ that are adjacent to $q$ in $T.$
A rooted tree $(T,r)$ is \emph{binary} if every vertex of $T$ has at most two children.

\paragraph{Treewidth of boundaried graphs.}
Let $\mathbf{G}=(G,B,\rho)$ be a boundaried graph.
A \emph{tree decomposition} of $\mathbf{G}$ is a triple $(T,\chi, r)$
where $(T,\chi)$ is a tree decomposition of $G$ and $r$ is a vertex of $T$ such that $\chi(r)=B.$
The \emph{width} of $(T,\chi, r)$ is the width of $(T,\chi).$
The treewidth of a boundaried graph $\mathbf{G}$ is the minimum width over all its tree decompositions and is denoted by $\mathsf{tw}(\mathbf{G}).$

Let $\mathbf{G}=(G,B,\rho)$ be a boundaried graph and $(T,\chi, r)$ be a tree decomposition of $\mathbf{G}.$
For every $q\in V(T),$ we set $T_q= T[\mathsf{desc}_{T,r} (q)]$ and $G_q = G[\bigcup_{w\in V(T_q)} \chi(w)].$
Notice that if $a,b\in V(T)$ and $a\leq_{(T,r)} b,$ then $G_b$ is a subgraph of $G_a.$
We also define the $t_q$-boundaried graph $\bar{\mathbf{G}}_q = (\bar{G}_q, \chi(q), \rho_q),$ where $\bar{G}_q = G\setminus (V(G_q)\setminus \chi(q)).$
Notice that $\mathbf{G}_q$ and $\bar{\mathbf{G}}_q$ are compatible and $\mathbf{G}_q \oplus \bar{\mathbf{G}}_q = G.$
\medskip

Our next step is to use a special type of tree decompositions, namely \textsl{linked tree decompositions}, defined by Robertson and Seymour in~\cite{RobertsonS86GMV}.
Thomas in~\cite{Thomas90amen} proved that every graph $G$ admits a linked tree decomposition of width $\mathsf{tw}(G)$ (see also \cite{BellenbaumD02twos,Erde18auni}).
By combining the result of \cite{Thomas90amen} and \cite[Lemmas 4 and 6]{ChatzidimitriouTZ20spar},
we can consider tree decompositions  as asserted in the following result.
\begin{proposition}\label{label_consecuencia}
	Let $t\in\mathbb{N}_{\geq 1}.$
	For every boundaried graph $\mathbf{G}=(G,B,\rho)$ of treewidth $t-1,$
	there exists a tree decomposition $(T,\chi,r)$ of $\mathbf{G}$ of width $t-1$ such that
\begin{enumerate}

		\item $(T,r)$ is a binary tree,

		\item  for every $a,b\in V(T)$
		      where $a$ is a child of $b$ in $(T,r),$ if $|\chi(a)|=|\chi(b)|$
		      then $G_a$ is a \textsl{proper} subgraph of $G_b,$
		      i.e., $|V(G_a)|<|V(G_b)|,$

		\item for every $s\in\mathbb{N}$ and every pair $u_1, u_2\in V(T),$ where $u_1\leq_{(T,r)} u_2$ and $|\chi(u_1)|=|\chi(u_2)|,$
		      either there is an internal vertex $w$ of $u_1 T u_2$ such that $|\chi(w)|< s,$
		      or there exists a collection of $s$ vertex-disjoint paths in $G$
		      between $\chi(u_1)$ and $\chi(u_2),$ and
		      
		\item $|V(G)|\leq t\cdot |V(T)|.$
	\end{enumerate}
\end{proposition}

In fact, linked tree decompositions are defined as the tree decompositions satisfying only property $(3)$ \cite{RobertsonS86GMV,Thomas90amen}.
In our proofs, we will need the extra properties $(1),$ $(2),$ and $(4)$ that are provided by \cite[Lemmas 4 and 6]{ChatzidimitriouTZ20spar}.

\subsection{Bounding the order of an obstruction of small treewidth}\label{label_commemorative}

In this subsection, we prove \autoref{label_substanceless}.
For this, we also need the following result ({for a proof see e.g. \cite[Lemma 14]{GiannopoulouPRT19cutw}}).

\begin{proposition}\label{label_accroissement}
	Let $r,m\in\mathbb{N}_{\geq 1}$ and $w$ be a word of length $m^r$ over the alphabet $[r].$
	Then there is a number $k\in[r]$ and a subword $u$ of $w$ such that $u$ contains only numbers not smaller than $k$ and $u$ contains the number $k$ at least $m$ times.
\end{proposition}

Note that a word of length $m^r$ over the alphabet $[r]$ can equivalently be seen as an element of~$S_{r,m^r}.$
We are now ready to prove \autoref{label_substanceless}.

\begin{proof}[Proof of \autoref{label_substanceless}]
	Let $G\in \mathbf{obs}(\mathcal{A}_k(\mathbf{excl}(\mathcal{F})).$
	We set $t:=\mathsf{tw}(G)+1.$
	For simplicity, we use $\ell$ instead of $\ell_{\mathcal{F}}.$
	We set
	\begin{align*}
		d:=                                   & \ \funref{label_ecclesiastes}(k,t,\ell), \\
		m:=                                   & \ {(2^{\binom{t}{2}}+1)\cdot d},         \\
		x:=                                   & \ m^{t}, \mbox{ and}                     \\
		\funref{label_schwankenden}(t,\ell):= & \ t\cdot 2^x.
	\end{align*}
	Suppose that $|V(G)|> \funref{label_schwankenden}(t,\ell).$
	Let $(T,\chi)$ be a tree decomposition of $G$ of width $\mathsf{tw}(G)$ and let $r\in V(T).$
	We consider the rooted tree $(T,r)$ and we set $B:= \chi(r)$ and a bijection $\rho:B\to [|B|].$
	We set $\mathbf{G}=(G,B,\rho)$ and observe that $(T,\chi, r)$ is a tree decomposition of $\mathbf{G}$ of width $\mathsf{tw}(G).$
	Since $\mathsf{tw}(\mathbf{G})= \mathsf{tw}(G) = t-1,$ by \autoref{label_consecuencia}, there exists a tree decomposition $(T,\chi,r)$ of $\mathbf{G}$ of width $t-1$ such that Properties (1) to (4) are satisfied.

	Since $|V(G)|> \funref{label_schwankenden}(t,\ell)= t\cdot 2^x,$ Property (4) implies that $|V(T)|> 2^x.$
	Also, by  Property (1), $(T,r)$ is a binary tree and therefore there exists a leaf $u$ of $T$ such that
	$|V(rTu)|\geq x.$
	We set $\ell:=|V(rTu)|.$

	We set $v_1= r$ and for every $i\in [\ell-1],$ we set $v_{i+1}$ to be the child of $v_i$ in $(T,r)$ that belongs to $V(rTu).$
	Keep in mind that $v_\ell = u.$
	For every $i\in [\ell],$ we set $c_i:=|\chi(v_i)|$ and observe that, since $(T,\chi, r)$ has width $t-1,$ $c_i\in[t].$

	Let $C$ be the word $c_1\cdots c_x.$
	Since $x=m^{t}$ and every $c_i \in [t],$ then, due to \autoref{label_accroissement},
	there is a $t'\in[t]$ and a subword $C'$ of $C$ such that,
	for every $c$ in $C',$ $c\geq t'$ and there are at least $m$ numbers in $C'$
	that are equal to $t'.$
	Therefore, there exists a set $\{z_1,\ldots, z_m\}\subseteq V(T)$ such that for every
	$i\in[2,m],$ $z_i$ is a descendant of $z_{i-1}$ in $(T,r),$ for every $z'\in V(z_1 T z_m)$
	it holds that $|\chi (z')|\geq t',$ and, for every $i\in[m],$ $|\chi( z_i )| = t'.$
	Hence, Property (3) of the tree decomposition $(T,\chi,r)$ of $\mathbf{G}$ implies that
	there exists a collection $\mathcal{P}=\{{P}_1,\ldots, {P}_{t'}\}$
	of $t'$ vertex-disjoint paths in $G$ between $\chi (z_1)$ and $\chi (z_m).$

	For every $i\in [m],$ let $\rho_i$ be the function mapping a vertex $v$ in $\chi(z_i)$ to the index of the path of $\mathcal{P}$
	it intersects, i.e., for every $j\in [t'],$ if  $v$ is a vertex in
	$V(P_j)\cap \chi(z_i),$ where $P_{j}\in\mathcal{P},$ then
	$\rho_i (v)= j.$
	Also, for every $i\in[m],$ let $\mathbf{G}_{z_i}$ be the $t'$-boundaried graph
	$(G_{z_i}, \chi(z_i), \rho_i).$
	Since, $m=(2^{\binom{t}{2}}+1)\cdot d,$ there is a set $J\subseteq [m]$ of size $d$ such that for every $i,j\in J,$ the graph $G_{z_i}[\chi(z_i)]$ is isomorphic to the graph $G_{z_j}[\chi(z_j)].$
	Therefore, for every $i,j\in J,$ $\mathbf{G}_{z_i}$ and $\mathbf{G}_{z_j}$ are compatible.
	Furthermore, observe that for every $i,j\in J$ with $i\leq j,$ $\mathbf{G}_{z_i}\prem \mathbf{G}_{z_j}.$
	To see why this holds, for every $i,j\in J$ with $i< j,$ let $\mathcal{P}_{i,j}$ be the collection of subpaths of $\mathcal{P}$ between the vertices of  $\chi(z_i)$ and $\chi(z_j)$ and consider the graph $G_{z_i}[\chi(z_i)] \cup \cupall\mathcal{P}_{i,j} \cup G_{z_j},$ that is a subgraph of $G_{z_i}.$
	By contracting the edges in $\mathcal{P}_{i,j},$ we obtain a boundaried graph isomorphic to $\mathbf{G}_{z_j}.$
	Also, recall that $|J| = d= \funref{label_ecclesiastes}(k,t,\ell).$
	Thus, by \autoref{label_rechtfertigen}, there exist $i,j\in J$ such that $j$ is the smallest element in $J$ that is greater than $i$ and
	$\mathsf{char}_{\mathbf{G}_{z_{i}}}^{(k,\ell)} = \mathsf{char}_{\mathbf{G}_{z_{j}}}^{(k,\ell)}.$
	For simplicity, we set $a:=z_{i}$ and $b:=z_{j}.$
	Notice that, in $G,$ by contracting the edges of the paths in $\mathcal{P}$ and removing the vertices of $G_a$ that are not vertices of $G_b,$ we obtain a graph isomorphic to $\bar{\mathbf{G}}_{a} \oplus \mathbf{G}_{b}.$
	Therefore, $\bar{\mathbf{G}}_{a} \oplus \mathbf{G}_{b}$ is a minor of $G.$
	Furthermore, $|V(\bar{\mathbf{G}}_a\oplus \mathbf{G}_b)|<|V(G)|.$
	To prove this, we argue that $G_b$ is a proper subgraph of $G_a.$
	First recall that for every $y\in V(a T b),$ $|\chi(y)|\geq t'.$
	If there is a $y\in V(aT b)$ such that $|\chi(y)|>t',$ then there is a vertex $v\in \chi(y)$ that is a vertex of $V(G_a)\setminus V(G_b)$ and thus $G_b$ is a proper subgraph of $G_a,$ while in the case where for every $y\in V(a T b),$ $|\chi(y)| = t',$ Property (2) implies that $G_b$ is a proper subgraph of $G_a.$

	Let $G' = \bar{\mathbf{G}}_a \oplus \mathbf{G}_b.$ Since  $|V(G')|<|V(G)|,$ $G'$ is a minor of $G,$ and $G\in \mathbf{obs}(\mathcal{A}_k (\mathbf{excl}(\mathcal{F}))),$ it holds that $G'\in\mathcal{A}_k (\mathbf{excl}(\mathcal{F})).$
	Therefore, there exists a set $S\subseteq V(G')$ of size $k$ such that $G'\setminus S\in \mathbf{excl}(\mathcal{F}).$
	Let $S_\mathrm{in} = S\cap V(G_b)$ and $S_\mathrm{out} = S\setminus S_\mathrm{in}.$
	We set $I_S:=\rho_b(\chi(b)\cap S_\mathrm{in}),$ $\mathbf{R}$ to be the $t'$-boundaried graph in $\mathcal{R}_\ell^{(t')}$ that is $\ell$-equivalent to $\mathbf{G}_b \setminus S_\mathrm{in},$ and $w:=\mathsf{char}_{\mathbf{G}_b}^{(k,\ell)}(I_S,\mathbf{R}).$
	Observe that $w\in[|S_\mathrm{in}|].$
	The fact that
	$\mathsf{char}_{\mathbf{G}_a}^{(k,\ell)} =\mathsf{char}_{\mathbf{G}_b}^{(k,\ell)}$ implies that
	$\mathsf{char}_{\mathbf{G}_b}^{(k,\ell)}(I_S,\mathbf{R}) =w.$
	Therefore, there exists a set $S'\subseteq V(G_a)$ such that $|S'| = w ,$
	$\rho_a (\chi(b)\cap S') = I_S,$ and $\mathbf{G}_a \setminus S' \leqslant_\ell \mathbf{R}.$
	Since $\mathbf{R}\equiv_{\ell,t'} \mathbf{G}_b \setminus S_\mathrm{in},$ the fact that $\mathbf{G}_a \setminus S' \leqslant_\ell \mathbf{R}$ implies that $\mathbf{G}_a \setminus S' \leqslant_\ell \mathbf{G}_b \setminus S_\mathrm{in}.$

	To conclude the proof, we argue that $G\setminus (S_\mathrm{out}\cup S')\in \mathbf{excl}(\mathcal{F}),$ which together with the fact that $|S_\mathrm{out}\cup S'|\leq |S| =k$ implies that $G\in\mathcal{A}_k (\mathbf{excl}(\mathcal{F})),$ a contradiction.
	Indeed, since $G\setminus (S_\mathrm{out}\cup S') = ( \bar{\mathbf{G}}_a \setminus S_\mathrm{out} ) \oplus (\mathbf{G}_a \setminus S')$ and $\mathbf{G}_a \setminus S' \leqslant_\ell \mathbf{G}_b \setminus S_\mathrm{in},$
	every graph $H\in\mathcal{F}$ that is a minor of
	$( \bar{\mathbf{G}}_a \setminus S_\mathrm{out} ) \oplus (\mathbf{G}_a \setminus S')$
	is also a minor of
	$( \bar{\mathbf{G}}_a \setminus S_\mathrm{out} ) \oplus (\mathbf{G}_b \setminus S_\mathrm{in})= G'\setminus S.$
	Consequently, $G'\setminus S\in \mathbf{excl}(\mathcal{F})$ implies that $G\setminus (S_\mathrm{out}\cup S')\in \mathbf{excl}(\mathcal{F}).$
\end{proof}

\section{Proof of \autoref{lemma_bidim_branch}}
\label{label_nominalistic}

In this section we prove a series of combinatorial  results.
In particular, in \autoref{label_acceptability} we prove a lemma (\autoref{label_automatisation}) that will be useful for the proof of
\autoref{label_emporteroient}, presented in \autoref{label_operationszeichen}.
The latter, together with a result proved in \autoref{label_reposadamente}, imply \autoref{lemma_bidim_branch}.

\subsection{Supporting combinatorial result}\label{label_acceptability}

Given a $(k\times r)$-grid $H$ with vertices $(x,y)
	\in[k]\times[r],$ and some $i\in [k],$
the \emph{$i$-th  vertical path} of $H$  is the one whose
vertices, in order of appearance, are $(i,1),(i,2),\ldots,(i,r).$
Also, given some $j\in[r],$ the \emph{$j$-th horizontal path} of $H$
is the one whose
vertices, in order of appearance, are $(1,j),(2,j),\ldots,(k,j).$

Given a $(n\times (2m+1))$-grid $H,$ we refer to the $(m+1)$-th horizontal path of $H$ as the \emph{middle horizontal path} of $H,$ which we denote by $P_{H}.$
Let $(i,j,j')\in[n]\times[-m,m]^2$ with $j\neq j'.$ We denote by $P_{i,j\rightarrow j'}$
the subpath of the $i$-th vertical path of $H$  starting from the vertex $(i, m+1+j)$
and finishing at $(i, m+1+ j').$
Let $(i,i',j)\in[n]^2\times[-m,m]$ with $i\neq i'.$ We denote by $P_{i\rightarrow i',j}$
the subpath of the $(m+1+j)$-th horizontal path of $H$ starting from the vertex $(i,m+1+j)$
and finishing at $(i',m+1+j).$ See \autoref{label_unreasonable} for an illustration of the above definitions.
\smallskip

\begin{figure}[ht]
	\begin{center}
		\includegraphics[width=4cm]{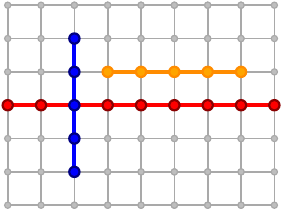}
	\end{center}
	\caption{A $(9\times 7)$-grid $H$ and the paths $\red{P_{H}},$ $\blue{P_{3,-2\to 2}},$ and $\orange{P_{4\to 8,1}}$ (depicted in red, blue, and orange, respectively).}
	\label{label_unreasonable}
\end{figure}

Given a path $P$ and three integers $r,h,d\geq 1,$
we say that a collection  ${\mathcal{C}}$  of subsets of $V(P)$ is \emph{$(r,h,d)$-scattered} in $P,$
if ${\mathcal{C}}=\{C_{1}, \ldots, C_{h}\},$ where for every $i\in[h]$ $C_{i}$ is a subset of $V(P)$ of cardinality $r,$
such that $\forall i,j\in [h],$  $C_{i}\cap C_{j}=\emptyset$ and $\forall u,v\in \cupall_{i\in [h]}C_{i},$ $\mathsf{dist}_{P}(u,v)> d$
\footnote{Given a graph $G$ and two vertices $u,v\in V(G),$ we define the \emph{distance between $u$ and $v$ in $G$}, denoted by $\mathsf{dist}_{G}(u,v),$ as the minimum number of edges in a path with $u,v$ as its endpoints.}.

We use the term $k$-\emph{grid} for the $(k\times k)$-grid. We say that a graph is a \emph{partially triangulated $r$-grid} if it can be obtained from an $r$-grid after adding edges in a way that the remaining graph remains planar.
We extend all above definitions of \emph{vertical, horizontal}, and \emph{middle horizontal path} of a grid to partially triangulated grids.
\medskip

\paragraph{Panchromatic contractions.} The purpose of this subsection and the next one is the proof of a lemma (\autoref{label_emporteroient})
on colored triangulated grids that, we believe, may have independent interest and applications. Our purpose is to prove that
for every $k,$ if $H$ is a big enough triangulated grid whose vertices are colored by some
fixed set of colors, so that each color appears many enough times in the sufficiently internal part
of $H,$ then $H$ can be contracted to a triangulated $k$-grid $R$ in a way that each vertex
of $R$ is the result of a ``panchromatic contraction'', in the sense that it is the result of the contraction
of vertices of \textsl{all} different colors. It also follows that the terms
``big enough'',  ``many enough'',  and ``sufficiently internal'' are quantified by functions that are polynomial in $k.$
This result is the combinatorial core of the proof of \autoref{lemma_bidim_branch}
that will follow in \autoref{label_reposadamente}.\medskip

The rest of this subsection is devoted to the proof of the following result, that intuitively states that given a big-enough grid $H$ and some colors for the vertices of the middle horizontal path of $H,$ if each color appears sufficiently many times in a scattered way, then we can contract $H$ to a large partially triangulated grid $R$ in which each vertex carries all colors.

\begin{lemma}\label{label_automatisation}
	There exist two functions  $\newfun{label_unthinkingly}:\mathbb{N}^3\to \mathbb{N}$ and $\newfun{label_unsuspecting}:\mathbb{N}\to \mathbb{N}$ such that for every $r,a,d\in\mathbb{N},$ with $d\geq 2r^{2},$
	if $H$ is a partially triangulated $(n\times m)$-grid with $n\geq \funref{label_unthinkingly}(r,a,d)$ and $m\geq \funref{label_unsuspecting}(r),$
	and ${\mathcal{C}}=\{C_{1}, \ldots, C_{a}\}$ is a collection of subsets of vertices of $P_{H}$ that  is $(r^{2},a,d)$-scattered in $P_{H},$
	then $H$ contains as a contraction a partially triangulated $r$-grid $R$ such that the model of each vertex of $R$ intersects every $C_{i}, i\in[a].$
	Moreover,  $\funref{label_unthinkingly}(r,a,d)=\mathcal{O}(r^2\cdot a \cdot d)$ and $\funref{label_unsuspecting}(r)=\mathcal{O}(r^{2}).$
\end{lemma}

\begin{proof}
Let $r,a,d\in\mathbb{N}$ such that $d\geq 2r^2.$
We set $\funref{label_unthinkingly}(r,a,d)=r^{2}\cdot a \cdot (d+1)$
and $\funref{label_unsuspecting}(r)=2(r^{2}+r+1)+1.$
Let $H$ be a partially triangulated $(n\times m)$-grid with $n\geq \funref{label_unthinkingly}(r,a,d)$ and $m\geq \funref{label_unsuspecting}(r),$
and ${\mathcal{C}}=\{C_{1}, \ldots, C_{a}\}$ be a collection of subsets of vertices of $P_{H}$ that  is $(r^{2},a,d)$-scattered in $P_{H}.$
Notice that we ask $n\geq \funref{label_unthinkingly}(r,a,d)=r^{2}\cdot a \cdot (d+1),$
in order to allow the existence of the collection ${\mathcal{C}}$ in $P_H.$
Also, keep in mind that the middle horizontal path $P_H$ of $H$ is its $\lceil m/2\rceil$-th horizontal path.

We define a function $p:\cupall\mathcal{C}\to [n]$ that maps every vertex $v\in \cupall\mathcal{C}$ to an integer $i\in[n]$
	if $v$ belongs to the intersection of the $i$-th vertical path of $H$ with $P_{H}.$
	Intuitively, $p(v)$ indicates the position of  vertex $v$ on the middle horizontal path of $H.$
	Observe that since ${\mathcal{C}}$ is  $(r^{2},a,d)$-scattered,
	it follows that for every $ u,v\in\cupall\mathcal{C}$ with $u\neq v,$ it holds that $|p(u)-p(v)|>d.$
	We define the relation $<_{p}$ on the vertices of $\cupall\mathcal{C}$ such that for every $u,v\in \cupall\mathcal{C}, u<_{p} v$ if and only if $p(u)< p(v).$
	For every $i\in[a],$ we fix an ordering of the elements of $C_{i}$ with respect to $<_{p},$ i.e., $C_{i}=\{v_{1}^{i}, \ldots, v_{r^{2}}^{i}\}$ where for every $j,j'\in[r^{2}],$ $j<j'$if and only if $v_{j}^{i}<_{p} v_{j'}^{i}.$
	Intuitively, we can see the set $C_{i}$ as the vertices in $\cupall\mathcal{C}$ colored with color $i$
	and $v_{j}^{i}$ as the $j$-th vertex of color $i$ that we encounter while traversing $P_{H}$ from left to right.

	We now aim to construct the vertices of the desired $r$-grid $R.$
	To do this, we define a collection of pairwise vertex-disjoint trees
	that are subgraphs of $H$ and every tree contains a vertex of every $C_{i}.$
	The edges of each tree will be contracted to a single vertex that will be a vertex of $R.$

	Towards this, we first consider a partition $X_{1},\ldots, X_{r^{2}}$ of \ $\cupall{\mathcal{C}}$
	such that for every $j\in [r^{2}],$ $X_{j}:=\{v_{j}^{1},\ldots v_{j}^{a}\}.$
	Intuitively, each set $X_{j}$ contains the $j$-th vertex (with respect to the ordering defined by $<_{p}$)
	of each color (i.e., of each $C_{i}, i\in[a]$).
	Observe that for every $(i,j)\in[a]\times [r^{2}],$ $C_{i}\cap X_{j}=\{v_{j}^{i}\}.$
	In the eventual grid $R$ that will be constructed, the model of every vertex of the grid will contain a unique set $X_j$ and therefore, as each $X_j$ intersects every $C_i,$ the model of every vertex of the grid $R$ will intersect every $C_i, i\in[a],$ as claimed.

	For every $j\in[r^{2}],$ let $x_{j}^\mathsf{left}$ (resp. $x_{j}^\mathsf{right}$) be the vertex in $X_{j}$
	such that for every $x\in X_{j},$ if $x\neq x_{j}^\mathsf{left}$ (resp. $x\neq x_{j}^\mathsf{right}$)
	then $x_{j}^\mathsf{left}<_{p} x$ (resp. $x<_{p} x_{j}^\mathsf{right}$).
	For every $j\in[r^{2}],$ we set $T_{j}$ to be the graph
	\[P_{p(x_{j}^\mathsf{left})\to p(x_{j}^\mathsf{right}),j} \cup P_{p(x_{j}^\mathsf{left}),j\to r^{2}+1}\cup P_{p(x_{j}^\mathsf{right}),-(r^{2}+1)\to 0}\cup \bigcup_{x \in V(X_{j})}P_{p(x),0\to j} .\]
	Recall that $P_H$ is the $\lceil m/2\rceil$-th horizontal path of $H.$
	We set $s_{j}=(p(x_{j}^\mathsf{left}),\lceil m/2 \rceil + r^2+1)$ and $t_{j}=(p(x_{j}^\mathsf{right}), \lceil m/2 \rceil -(r^2+1)).$
	See \autoref{label_administered} for an illustration of the above definitions.
	Observe that $T_{j}$ is a tree whose leaves are the vertices in $(X_{j}\setminus \{x_{j}^\mathsf{right}\})\cup \{s_{j}, t_{j}\}.$
	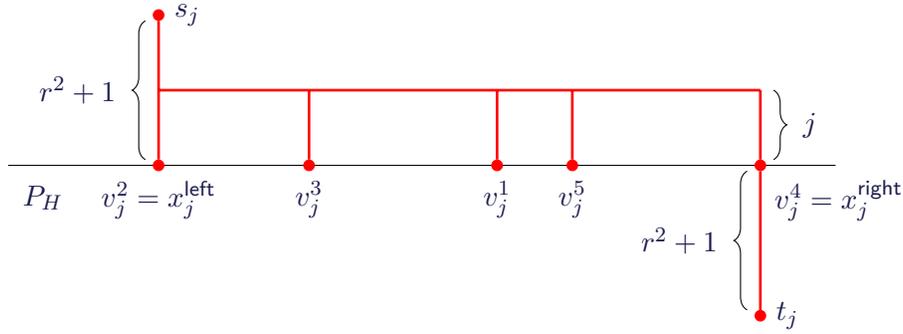
\begin{figure}[H]
		\centering
		\begin{tikzpicture}

			\draw[-] (-1,0) -- (10,0);
			\node[label=below left:{$P_H$}] (PH) at (0,0) {};

			\node[red node, label=right :{$s_j$}] (jup) at (1,2) {};
			\node[red node, label=below :{$v_j^2 = x_j^\mathsf{left}$}] (jleft) at (1,0) {};
			\node[red node, label=below :{$v_j^3$}] (vj3) at (3,0) {};
			\node[red node, label=below :{$v_j^1$}] (vj1) at (5.5,0) {};
			\node[red node, label=below :{$v_j^5$}] (vj5) at (6.5,0) {};
			\node[red node, label=below right:{$v_j^4 = x_j^\mathsf{right}$}] (jright) at (9,0) {};
			\node[red node, label=right :{$t_j$}] (jdown) at (9,-2) {};

			\draw[red, line width=1pt] (jup) -- (jleft) (jdown) -- (jright) (vj3) -- (3,1) (vj1) -- (5.5,1) (vj5) -- (6.5,1) (9,1) -- (jright) (1,1) -- (9,1);

			\draw [decorate,decoration={brace,amplitude=5pt, raise=5pt,mirror}] (jup) -- (jleft) node [black,midway,xshift = -0.3cm, label = left:{$r^2 +1$}] {};

			\draw [decorate,decoration={brace,amplitude=5pt, raise=5pt}] (9,1) -- (jright) node [black,midway,xshift = 0.3cm, label = right:{$j$}] {};

			\draw [decorate,decoration={brace,amplitude=5pt, raise=5pt}] (jdown) -- (jright) node [black,midway,xshift = -0.3cm, label = left:{$r^2 +1$}] {};
		\end{tikzpicture}
		\caption{Visualization of the graph $T_{j}$ (depicted in red) for $h=5.$}
		\label{label_administered}
	\end{figure}
	We stress that we can construct the graphs $T_{j}$ since $m\geq \funref{label_unsuspecting}(r).$
	Observe that every $T_{j}$ is a tree and for $j\neq j',$ $T_{j}$ and $T_{j'}$ are not necessarily vertex-disjoint.
	To get a collection of pairwise vertex-disjoint trees, we have to resolve possible intersections.

	Notice that if $j<j',$ then $T_{j}$  intersects $T_{j'}$ only in the vertices $(p(v_{j'}^{i}),j), i\in[a]$
	where $v_{j'}^{i} <_{p} x_{j}^\mathsf{right}$ (see \autoref{label_producedness}).
	For every $j \in[r^{2}-1]$ we set
	\[I_{j}=\{h\in[n]\mid \exists (i,j')\in[a]\times [j+1,r^{2}] : \ h=p(v_{j'}^{i}) \wedge v_{j'}^{i}<_{p} x_{j}^\mathsf{right} )\}.\]
	Intuitively, these are the positions (in $P_{H}$) of the vertices of every $T_{j'}, j'>j$
	that are on the left of $x_{j}^\mathsf{right}$ (see \autoref{label_producedness}).
	\begin{figure}[ht]
		\centering
		\begin{tikzpicture}

			\draw[-] (0,0) -- (12,0);
			\node[label=below:{$P_H$}] (PH) at (0,0) {};

			\node[red node] (jup) at (1,2) {};
			\node[red node, label=below:{$v_j^2$}] (jleft) at (1,0) {};
			\node[red node, label=below:{$v_j^3$}] (vj3) at (3.25,0) {};
			\node[red node, label=below:{$v_j^1$}] (vj1) at (5.5,0) {};
			\node[red node, label=below:{$v_j^5$}] (vj5) at (6.25,0) {};
			\node[red node, label=below right:{\!\!$v_j^4$}] (jright) at (8.5,0) {};
			\node[red node] (jdown) at (8.5,-2) {};
			\node[label = above :{$\red{T_{j}}$}] () at (1,2) {};
			\draw[red, line width=1pt] (jup) -- (jleft) (jdown) -- (jright) (vj3) -- (3.25,1) (vj1) -- (5.5,1) (vj5) -- (6.25,1) (8.5,1) -- (jright) (1,1) -- (8.5,1);

			\node[blue node] (iup) at (1.75,2) {};
			\node[blue node, label=below:{$v_{j'}^2$}] (ileft) at (1.75,0) {};
			\node[blue node, label=below:{$v_{j'}^3$}] (vi3) at (4,0) {};
			\node[blue node, label=below:{$v_{j'}^1$}] (vi5) at (7,0) {};
			\node[blue node, label=below:{$v_{j'}^5$}] (vi4) at (9.25,0) {};
			\node[blue node, label=below right:{\!\!$v_{j'}^4$}] (iright) at (10.75,0) {};
			\node[blue node] (idown) at (10.75,-2) {};
			\node[label =above:$\blue{T_{j'}}$] () at (1.75,2) {};
			\draw[blue, line width=1pt] (iup) -- (ileft) (idown) -- (iright) (vi3) -- (4,1.25) (vi5) -- (7,1.25) (vi4) -- (9.25,1.25) (10.75,1.25) -- (iright) (1.75,1.25) -- (10.75,1.25);

			\node[green node] (kup) at (2.5,2) {};
			\node[green node, label=below:{$v_{j''}^2$}] (kleft) at (2.5,0) {};
			\node[green node, label=below:{$v_{j''}^3$}] (vk3) at (4.75,0) {};
			\node[green node, label=below:{$v_{j''}^1$}] (vk5) at (7.75,0) {};
			\node[green node, label=below:{$v_{j''}^5$}] (vk4) at (10,0) {};
			\node[green node, label=below right:{\!\!$v_{j''}^4$}] (kright) at (11.5,0) {};
			\node[green node] (kdown) at (11.5,-2) {};
			\node[label = above:{$\green{T_{j''}}$}] () at (2.5,2) {};	
			\draw[Mygreen, line width=1pt] (kup) -- (kleft) (kdown) -- (kright) (vk3) -- (4.75,1.5) (vk5) -- (7.75,1.5) (vk4) -- (10,1.5) (11.5,1.5) -- (kright) (2.5,1.5) -- (11.5,1.5);

			\draw [decorate,decoration={brace,amplitude=4pt, raise=4pt}] (8.5,1) -- (jright) node [black,midway,xshift = 0.2cm, label = right:{\!\!$j$}] {};

			\draw [decorate,decoration={brace,amplitude=4pt, raise=4pt}] (10.75,1.25) -- (iright) node [black,midway,xshift = 0.2cm, label = right:{\!\!$j'$}] {};

			\draw [decorate,decoration={brace,amplitude=4pt, raise=4pt}] (11.5,1.5) -- (kright) node [black,midway,xshift = 0.2cm, label = right:{\!\!$j''$}] {};

		\end{tikzpicture}
		\caption{Visualization of the graphs $T_{j}$ (depicted in red), $T_{j'}$ (depicted in blue), and $T_{j''}$ (depicted in green) for $h=5.$
Here, if we assume that $j < j'<j'',$
then $I_{j}=\{p(v_{j'}^{2}), p(v_{j''}^{2}), p(v_{j'}^{3}), p(v_{j''}^{3}), p(v_{j'}^{1}), p(v_{j''}^{1})\}.$}
		\label{label_producedness}
	\end{figure}
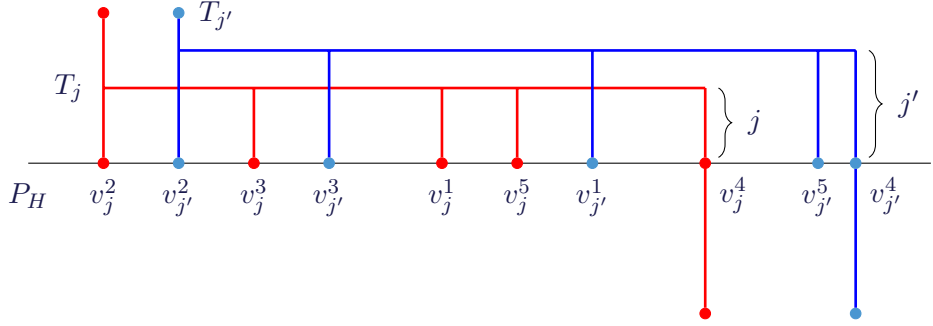
	For every $h\in I_{j},$ we set $h^\mathsf{left}= h-(r^{2}-j),$ $h^\mathsf{right}=h+r^{2}-j,$ and
	$U_{h}^j$ to be the graph depicted  in \autoref{label_establecerla}. More precisely,
	\[U_{h}^j = P_{h^\mathsf{left},-(r^{2}-j)\to j}\cup P_{h^\mathsf{left}\to h^\mathsf{right},-(r^{2}-j)}\cup P_{h^\mathsf{right},-(r^{2}-j)\to j}.\]

\begin{figure}[H]
		\centering
		\begin{tikzpicture}
			\draw[-] (0,0) -- (8,0);
			\node[label=below:{$P_H$}] (PH) at (8,0) {};

			\node[blue node] (upleft) at (2,1) {};
			\node[blue node, label=above left:{$(h^\mathsf{left},\lceil m/2\rceil)$}] (left) at (2,0) {};
			\node[blue node] (downleft) at (2,-1.5) {};
			\node[black node,label=above:{$(h,\lceil m/2\rceil)$}] (center) at (4,0) {};
			\node[blue node] (upright) at (6,1) {};
			\node[blue node,label=above right:{$(h^\mathsf{right},\lceil m/2\rceil)$}] (right) at (6,0) {};
			\node[blue node] (downright) at (6,-1.5) {};

			\draw[blue, line width=1pt] (upleft) -- (left) -- (downleft) -- (downright) -- (right) -- (upright);

			\draw [decorate,decoration={brace,amplitude=5pt, raise=3pt}] (upleft) -- (left) node [black,midway,xshift = 0.1cm, label = right:{$j$}] {};

			\draw [decorate,decoration={brace,amplitude=5pt, raise=5pt,mirror}] (left) -- (downleft) node [black,midway,xshift = -0.3cm, label = left:{$r^2 -j$}] {};

			\draw [decorate,decoration={brace,amplitude=5pt, raise=5pt,mirror}] (left) -- (center) node [black,midway,yshift = -0.1cm, label = below:{$r^2 -j$}] {};

			\draw [decorate,decoration={brace,amplitude=5pt, raise=5pt,mirror}] (center) -- (right) node [black,midway,yshift = -0.1cm, label = below:{$r^2 -j$}] {};
		\end{tikzpicture}

		\caption{Visualization of the graph $U_{h}^j.$
			The vertices $(h^\mathsf{left},\lceil m/2\rceil),$ $(h,\lceil m/2\rceil),$ and $(h^\mathsf{right},\lceil m/2\rceil)$ are the vertices on the intersection of the path $P_H$ with the $h^\mathsf{left}$-th, $h$-th, and $h^\mathsf{right}$-th vertical path of $H,$ respectively.}
		\label{label_establecerla}
	\end{figure}
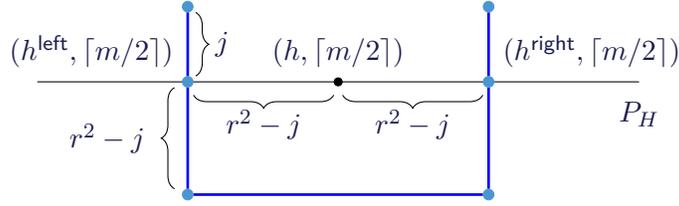

	Also, we set
	$T^{\star}_{j}$ to be the graph
	\[\left(T_{j}\setminus \bigcup_{h\in I_{j}} P_{h^\mathsf{left} \to h^\mathsf{right},j}\right)\cup \bigcup_{h\in I_{j}} U_{h}^j.\]
	Observe that, since ${\mathcal{C}}$ is $(r^{2},a,d)$-scattered and $d\geq 2r^2,$ $T_{1}^{\star}, \ldots, T_{r^{2}}^{\star}$ are pairwise vertex-disjoint trees each containing a vertex of every $C_{i}, i\in[a].$
	Indeed, any possible intersection between, say $T_i$ and $T_j,$ does not exist anymore when we reroute through the graphs $U_h^j, j\in[r^2], h\in I_j.$
	Moreover, no new intersections are created by the addition of the graphs $U_h^j,$ since by the fact that ${\mathcal{C}}$ is $(r^{2},a,d)$-scattered, $d\geq 2r^2,$ and by the construction of $U_h^j,$ every two $U_{h}^{j},U_{h'}^{j'},$ with $j\neq j'$ and $h\neq h',$ are disjoint. See \autoref{label_publications}.
	\begin{figure}[H]
		\centering
		\begin{tikzpicture}

			\draw[-] (0,0) -- (12,0);
			\node[label=below:{$P_H$}] (PH) at (0,0) {};

			\node[red node] (jup) at (1,2) {};
			\node[red node, label=below:{$v_j^2$}] (jleft) at (1,0) {};
			\node[red node, label=below:{$v_j^3$}] (vj3) at (3.25,0) {};
			\node[red node, label=below:{$v_j^1$}] (vj1) at (5.5,0) {};
			\node[red node, label=below:{$v_j^5$}] (vj5) at (6.25,0) {};
			\node[red node, label=below right:{\!\!$v_j^4$}] (jright) at (8.5,0) {};
			\node[red node] (jdown) at (8.5,-2) {};
			\node[red node] (1ul) at (1.45,1) {};
			\node[red node] (1dl) at (1.45,-1) {};
			\node[red node] (1ur) at (2.05,1) {};
			\node[red node] (1dr) at (2.05,-1) {};
			\node[red node] (2ul) at (2.25,1) {};
			\node[red node] (2dl) at (2.25,-1) {};
			\node[red node] (2ur) at (2.75,1) {};
			\node[red node] (2dr) at (2.75,-1) {};
			\node[red node] (3ul) at (3.7,1) {};
			\node[red node] (3dl) at (3.7,-1) {};
			\node[red node] (3ur) at (4.3,1) {};
			\node[red node] (3dr) at (4.3,-1) {};
			\node[red node] (4ul) at (4.5,1) {};
			\node[red node] (4dl) at (4.5,-1) {};
			\node[red node] (4ur) at (5.05,1) {};
			\node[red node] (4dr) at (5.05,-1) {};
			\node[red node] (5ul) at (6.7,1) {};
			\node[red node] (5dl) at (6.7,-1) {};
			\node[red node] (5ur) at (7.3,1) {};
			\node[red node] (5dr) at (7.3,-1) {};
			\node[red node] (6ul) at (7.5,1) {};
			\node[red node] (6dl) at (7.5,-1) {};
			\node[red node] (6ur) at (8.05,1) {};
			\node[red node] (6dr) at (8.05,-1) {};
			\node[label = above :{$\red{T_{j}^{\star}}$}] () at (1,2) {};
			\draw[red, line width=1pt] (jup) -- (jleft) (jdown) -- (jright) (vj3) -- (3.25,1) (vj1) -- (5.5,1) (vj5) -- (6.25,1) (8.5,1) -- (jright) (1,1) -- (1ul) -- (1dl) -- (1dr) -- (1ur) -- (2ul) -- (2dl) -- (2dr) -- (2ur) -- (3ul) -- (3dl) -- (3dr) -- (3ur) -- (4ul) -- (4dl) -- (4dr) -- (4ur) -- (5ul) -- (5dl) -- (5dr) -- (5ur) -- (6ul) -- (6dl) -- (6dr) -- (6ur) -- (8.5,1);

			\node[blue node] (iup) at (1.75,2) {};
			\node[blue node, label=below:{}] (ileft) at (1.75,0) {};
			\node[blue node, label=below:{}] (vi3) at (4,0) {};
			\node[blue node, label=below:{}] (vi5) at (7,0) {};
			\node[blue node, label=below:{}] (vi4) at (9.25,0) {};
			\node[blue node, label=below right:{}] (iright) at (10.75,0) {};
			\node[blue node] (idown) at (10.75,-2) {};
			\node[label =above:$\blue{T_{j'}^{\star}}$] () at (1.75,2) {};
			\node[blue node] (1ulb) at (2.35,1.25) {};
			\node[blue node] (1dlb) at (2.35,-0.75) {};
			\node[blue node] (1urb) at (2.65,1.25) {};
			\node[blue node] (1drb) at (2.65,-0.75) {};
			\node[blue node] (2ulb) at (4.6,1.25) {};
			\node[blue node] (2dlb) at (4.6,-0.75) {};
			\node[blue node] (2urb) at (4.95,1.25) {};
			\node[blue node] (2drb) at (4.95,-0.75) {};
			\node[blue node] (3ulb) at (7.6,1.25) {};
			\node[blue node] (3dlb) at (7.6,-0.75) {};
			\node[blue node] (3urb) at (7.95,1.25) {};
			\node[blue node] (3drb) at (7.95,-0.75) {};
			\node[blue node] (4ulb) at (9.85,1.25) {};
			\node[blue node] (4dlb) at (9.85,-0.75) {};
			\node[blue node] (4urb) at (10.15,1.25) {};
			\node[blue node] (4drb) at (10.15,-0.75) {};
			\draw[blue, line width=1pt] (iup) -- (ileft) (idown) -- (iright) (vi3) -- (4,1.25) (vi5) -- (7,1.25) (vi4) -- (9.25,1.25) (10.75,1.25) -- (iright) (1.75,1.25) -- (1ulb) -- (1dlb) -- (1drb) -- (1urb) -- (2ulb) -- (2dlb) -- (2drb) -- (2urb) -- (3ulb) -- (3dlb) -- (3drb) -- (3urb) -- (4ulb) -- (4dlb) -- (4drb) -- (4urb) -- (10.75,1.25);

			\node[green node] (kup) at (2.5,2) {};
			\node[green node, label=below:{}] (kleft) at (2.5,0) {};
			\node[green node, label=below:{}] (vk3) at (4.775,0) {};
			\node[green node, label=below:{}] (vk5) at (7.775,0) {};
			\node[green node, label=below:{}] (vk4) at (10,0) {};
			\node[green node, label=below right:{}] (kright) at (11.5,0) {};
			\node[green node] (kdown) at (11.5,-2) {};
			\node[label = above:{$\green{T_{j''}}$}] () at (2.5,2) {};	
			\draw[Mygreen, line width=1pt] (kup) -- (kleft) (kdown) -- (kright) (vk3) -- (4.775,1.5) (vk5) -- (7.775,1.5) (vk4) -- (10,1.5) (11.5,1.5) -- (kright) (2.5,1.5) -- (11.5,1.5);

		\end{tikzpicture}
		\caption{The trees $T_{j}^{\star}$ (depicted in red), $T_{j'}^{\star}$ (depicted in blue),  and $T_{j''}$ (depicted in green).}
		\label{label_publications}
	\end{figure}
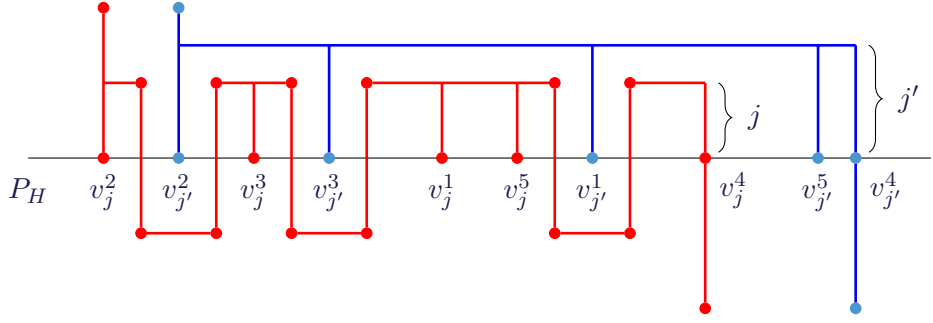

	Towards the construction of the desired $r$-grid, we already mentioned that some trees would be contracted to single vertices.
	These trees are $T_{j}^{\star}, j\in[r^{2}].$
	Our aim now is to ``connect'' these vertices, obtained by the contraction of each $T_{j}^{\star}, j\in[r^{2}],$ in order to form the desired $r$-grid.

	Recall that, for every $j\in[r^2],$ $s_{j}=(p(x_{j}^\mathsf{left}),\lceil m/2 \rceil + r^2+1)$ and $t_{j}=(p(x_{j}^\mathsf{right}), \lceil m/2 \rceil -(r^2+1)).$
	For simplicity, we set $l^{\uparrow} = \lceil m/2 \rceil +r^{2}+1$ and $l^{\downarrow}=\lceil m/2 \rceil -(r^{2}+1).$
	Also, for every $j\in[r^2],$ we set $q_j^\mathsf{left} = p(x_{j}^\mathsf{left})$ and $q_j^\mathsf{right} = p(x_{j}^\mathsf{right}).$
	Therefore, for every $j\in[r^2],$ $s_{j}=(q_j^\mathsf{left},l^{\uparrow})$ and $t_{j}=(q_j^\mathsf{right},l^{\downarrow}).$

	Now, for every odd $i\in [r-1],$ we define $L_{i}^{\uparrow}$ to be the graph
	\begin{align*}
		\bigcup_{j\in[r]}\left(P_{q^\mathsf{left}_{(i-1)\cdot r +j},l^{\uparrow}\to l^{\uparrow}+j} \cup P_{q^\mathsf{left}_{(i-1)\cdot r+j}\to q^\mathsf{left}_{(i+1)\cdot r - j +1},l^{\uparrow}+j} \cup P_{q^\mathsf{left}_{(i+1)\cdot r - j +1},l^{\uparrow}\to l^{\uparrow} + j}\right) \\
		\cup  P_{q^\mathsf{left}_{(i-1)\cdot r +1}\to q^\mathsf{left}_{i\cdot r}, l^{\uparrow}} \cup P_{q^\mathsf{left}_{i\cdot r +1}\to q^\mathsf{left}_{(i+1)\cdot r}, l^{\uparrow}}.
	\end{align*}
	Also, for every even $i\in [r-1],$ we define $L_{i}^{\downarrow}$ to be the graph
	\begin{align*}
		\bigcup_{j\in[r]}\left(P_{q^\mathsf{right}_{(i-1)\cdot r +j},l^{\downarrow}\to l^{\downarrow}-j} \cup P_{q^\mathsf{right}_{(i-1)\cdot r+j}\to q^\mathsf{right}_{(i+1)\cdot r - j +1},l^{\downarrow}-j}\cup P_{q^\mathsf{right}_{(i+1)\cdot r - j +1},l^{\downarrow}\to l^{\downarrow} - j}\right) \\
		\cup  P_{q^\mathsf{right}_{(i-1)\cdot r +1}\to q^\mathsf{right}_{i\cdot r}, l^{\downarrow}} \cup P_{q^\mathsf{right}_{i\cdot r +1}\to q^\mathsf{right}_{(i+1)\cdot r}, l^{\downarrow}}.
	\end{align*}
	Then, we consider the graph $R^{\star}$
	\[
		\bigcup_{j\in[r^{2}]} T_{j}^{\star}\cup \bigcup_{\mbox{{\footnotesize odd} }i\in[r-1]}L_{i}^{\uparrow} \cup\bigcup_{\mbox{{\footnotesize even} }i\in[r-1]}L_{i}^{\downarrow}.
	\]
	See \autoref{label_justification} for an illustration of the above graphs.
	\begin{figure}[H]
		\centering
		\includegraphics[width=14.5cm]{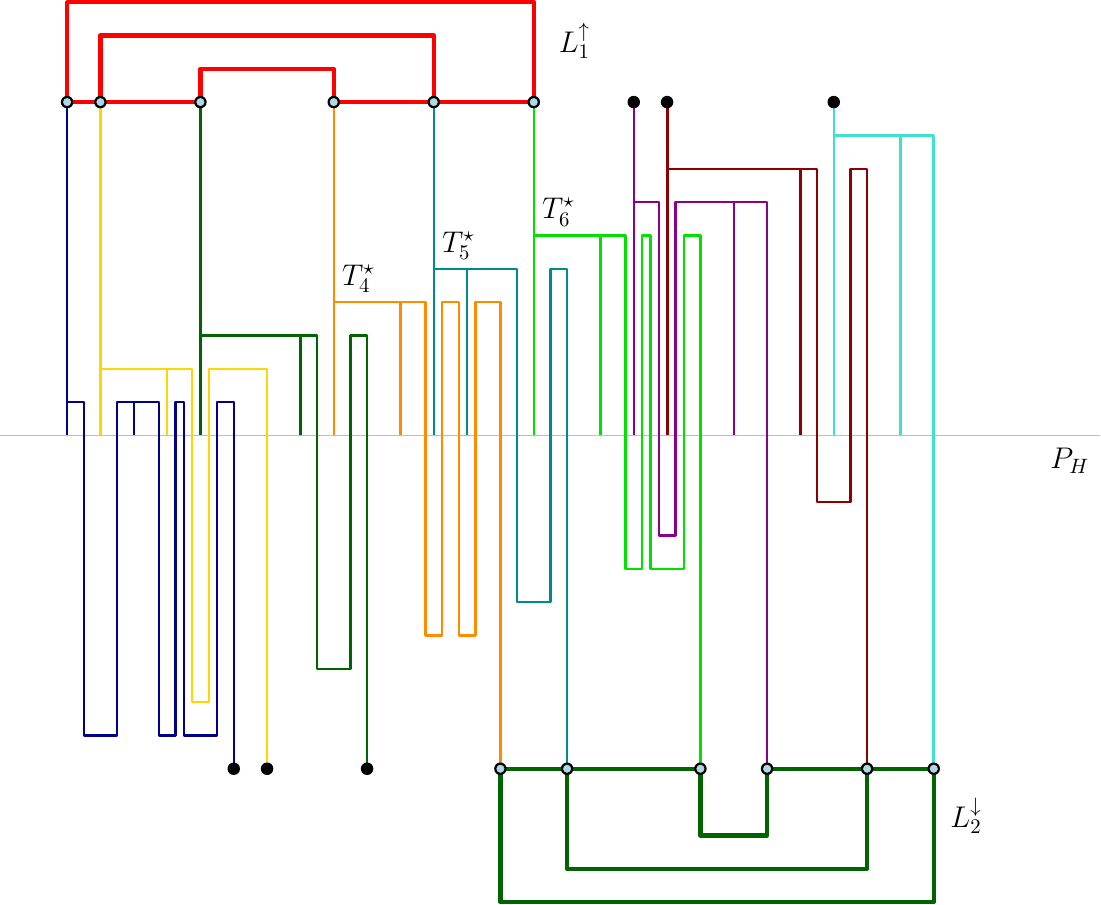}
		\caption{Visualization of the graph $R^{\star}.$ Notice that the trees $T_{4}^{\star}, T_{5}^{\star},$ and $T_{6}^{\star}$ (depicted in orange, blue, and green, respectively) intersect both $L_{1}^{\uparrow}$ and $L_{2}^{\downarrow}.$ 
		}
		\label{label_justification}
	\end{figure}

	We now consider the graph $\tilde{R}$ obtained from $R^{\star}$ if for every $i\in[r^{2}]$ we contract all edges of $T_{i}^{\star}$ and then we contract each path of $\bigcup_{\mbox{{\footnotesize odd} }i\in[r-1]}L_{i}^{\uparrow} \cup\bigcup_{\mbox{{\footnotesize even} }i\in[r-1]}L_{i}^{\downarrow}$ to an edge. We now prove the following:
	\medskip

	\noindent\emph{Claim:}  $\tilde{R}$ is an $r$-grid.\medskip

	\noindent\emph{Proof of the claim:}
	For every odd (resp. even) $i\in[r]$ and every $j\in[r],$ let $w_{i,j}$ be the vertex obtained after contracting the edges of $T_{(i-1)\cdot r+j}^{\star}$ (resp. $T_{i\cdot r -j +1}^{\star}$) and keep in mind that if $i$ is odd (resp. even), then the model of $w_{i,j}$ in $H$ contains $s_{(i-1)\cdot r +j}$ (resp. $t_{i\cdot r -j +1}$).
	We assume that $V(\tilde{R}) = \bigcup_{i,j\in [r]} \{w_{i,j}\}$ and we argue that, for every $i\in [r]$ and every $j\in[r],$ $w_{i,j}$ is adjacent, in $\tilde{R},$ to $w_{i,j-1}$ (if $j>1$), to $w_{i,j+1}$ (if $j<r$), to $w_{i-1,j}$ (if $i>1$), and to $w_{i+1,j}$ (if $i<r$).
	This implies that $\tilde{R}$ is an $r$-grid.
	We first show that for every $i\in[r]$ and every $j\in[r],$ $w_{i,j}$ is adjacent, in $\tilde{R},$ to $w_{i,j-1},$ if $j>1,$ and to $w_{i,j+1},$ if $j<r.$
	For this, observe that if $i$ is even (resp. odd) then the vertex $s_{(i-1)\cdot r+j}$ (resp. $t_{i\cdot r -j+1}$) is connected through $L_i^{\uparrow}$ (resp. $L_i^{\downarrow}$) with the vertex $s_{(i-1)\cdot r + j+ 1}$ (resp. $t_{i\cdot r -j +2}$), if $j<r,$ and  the vertex $s_{(i-1)\cdot r+j-1}$ (resp. $t_{i\cdot r-j}$), if $j>1.$
	Therefore, since if $i$ is even (resp. odd), then the model of $w_{i,j}$ in $H$ contains $s_{(i-1)\cdot r +j}$ (resp. $t_{i\cdot r -j +1}$), we have that $w_{i,j}$ is adjacent, in $\tilde{R},$ to $w_{i,j-1}$ (if $j>1$) and to $w_{i,j+1}$ (if $j<r$).
	Also, notice that, if $i$ is even, the vertex $s_{(i-1)\cdot r+j}$ is connected through $L_i^{\uparrow}$ to the vertex $s_{(i+1)\cdot r-j+1},$ that is a vertex in the model of $w_{i+1,j}$ in $H.$
	If $i$ is odd, then the vertex $t_{i\cdot r -j + 1}$ is connected through $L_i^{\downarrow}$ with $t_{(i +1)\cdot r +j},$ that is a vertex in the model of $w_{i+1,j}$ in $H.$
	The claim follows.\hfill$\diamond$
	\medskip

	By the claim above, $\tilde{R}$ is an $r$-grid.
	If we further contract every edge that is adjacent to a vertex of $V(H)\setminus V(\tilde{R}),$ we obtain a partially triangulated $r$-grid $R$ as the desired one.
\end{proof}

\subsection{Finding a complete apex grid}\label{label_operationszeichen}

\paragraph{Central grids.}
Let  $k,r\in\mathbb{N}_{\geq 2}.$
We define the \emph{perimeter} of a $(k\times r)$-grid to be the unique cycle of the grid of length at least three that does not contain vertices of degree four.

\begin{figure}[ht]
	\centering
	\includegraphics[width=3.5cm]{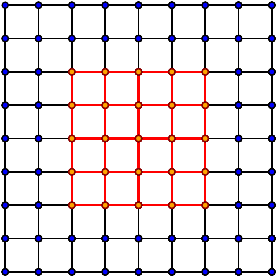}
	\caption{A 9-grid and its central 5-grid.}
	\label{label_administrativas}
\end{figure}

Let $r\in \mathbb{N}_{\geq 2}$ and $H$ be an $r$-grid.
Given an $i\in\lceil \frac{r}{2}\rceil,$ we define the \emph{$i$-th layer} of $H$ recursively as follows.
The first layer of $H$ is its perimeter, while, if $i\geq 2,$ the $i$-th layer of $H$ is
the $(i-1)$-th layer of the grid created if we remove from $H$ its perimeter.
Given two  odd integers $q,r\in\mathbb{N}_{\geq 3}$ such that $q\leq r$ and an $r$-grid $H,$
we define the \emph{central $q$-grid} of $H$ to be the graph obtained from $H$
if we remove from $H$ its $\frac{r-q}{2}$ first layers.
See \autoref{label_administrativas} for an illustration of the notions defined above.
Given a partially triangulated $r$-grid $H,$ we call \emph{central $q$-grid} of $H$ the subgraph of $H$ induced by the vertices of the central $q$-grid of the underlying grid of $H.$

\begin{lemma}\label{label_emporteroient}
	There exist three functions $\newfun{label_impercepbbly}, \newfun{label_einbegreifen}: \mathbb{N}^{2}\to \mathbb{N},$ and $\newfun{label_presupongamos}:\mathbb{N}\to \mathbb{N}$
	such that if $r,a\in \mathbb{N},$ $H$ is a partially triangulated $h$-grid, where $h\geq \funref{label_impercepbbly}(r,a)+2\cdot \funref{label_presupongamos}(r),$ and
	${\mathcal{S}}=\{S_{1},\ldots, S_{a}\}$ is a collection
	of $a$ subsets of vertices in the central $\funref{label_impercepbbly}(r,a)$-grid of $H$ such that  for every $i\in[a], |S_{i}|\geq \funref{label_einbegreifen}(r,a),$
	then $H$ contains as a contraction a partially triangulated $r$-grid $R$
	such that the model of each vertex of $R$ in $H$ intersects every $S_{i},i\in[a].$ Moreover,  $\funref{label_impercepbbly}(r,a)=\mathcal{O}(r^{4}\cdot 2^a),$
	$\funref{label_einbegreifen}(r,a)=\mathcal{O}(r^{6}\cdot 2^{a}),$ and $\funref{label_presupongamos}(r)=\mathcal{O}(r^{2}).$
\end{lemma}

\begin{proof}
Let $\funref{label_unsuspecting}, \funref{label_unthinkingly}$ be the functions of \autoref{label_automatisation}.  We set $\ell: = \max\{2r^{2},\funref{label_unsuspecting}(r)\}=\funref{label_unsuspecting}(r)$ and $n= \max\{\funref{label_unthinkingly}(r,a,\ell),2^{a-1}\cdot r^2 \cdot a\cdot (\ell+1)\}.$ We also set
\begin{align*}
	b:= & \ \ell\cdot (a+1)+2, &
	z:= & \ \lceil\sqrt{n}\rceil, &
	\funref{label_impercepbbly}(r,a):= & \ b \cdot z, \\
	\funref{label_einbegreifen}(r,a):= & \ 2^{a-1}\cdot r^{2}\cdot b^{2},~\mbox{and} &
	\funref{label_presupongamos}(r):= & \ \ell.
\end{align*}
We begin by arguing that the following claim holds:
\medskip

\noindent\emph{Claim 1:} $H$ contains a partially triangulated $(\ell\times n)$-grid $R'$ as a contraction and  there is a collection $\mathcal{V}=\{V_{1},\ldots, V_{a}\}$ of subsets of the vertices of the middle horizontal path of $R',$ where
	\begin{itemize}
		\item for every $i\in[a],$ the model of each vertex $v\in V_{i}$ in $H$ intersects $S_{i},$
		\item for every $i\in[a],$ $|V_{i}|= 2^{a-1}\cdot r^{2},$ and
		\item for every distinct $u,v\in \cupall\mathcal{V}, \mathsf{dist}_{R'}(u,v)> \ell.$
	\end{itemize}

	\noindent\emph{Proof of Claim 1:}
	Let $\breve{H}$ be the central $\funref{label_impercepbbly}(r,a)$-grid of $H$ and keep in mind that $\funref{label_impercepbbly}(r,a)=b\cdot z.$
	Also, let $\mathcal{P}=\{P_{1}, \ldots, P_{\funref{label_impercepbbly}(r,a)}\}$ be the set of
	the vertical paths of $\breve{H},$
	where $P_{i}$ is the $i$-th vertical path of $\breve{H}.$
	For every $j\in[b],$ let $\mathcal{P}_{j}=\bigcup_{i\in[z]}{P_{j+b(i-1)}}.$
	For every $i\in[a],$ let $x_{i}:=\arg\max_{j\in[b]}\{|V(\mathcal{P}_{j})\cap S_{i}|\}.$
	Intuitively, we partition $\mathcal{P}$ into $b$ sets $\mathcal{P}_j,$ $j\in[b],$
	each one consisting of the $j$-th, $(j+b)$-th, $\ldots,$ $(j+(z-1)\cdot b)$-th vertical path of $\breve{H},$ and
	$x_{i}$ is defined as the index $j$ maximizing the size of the  intersection of $V(\mathcal{P}_j)$ with $S_{i}.$
	Observe that, since $|S_{i}|\geq \funref{label_einbegreifen}(r,a),$ by the pigeonhole principle it follows that $|V(\mathcal{P}_{x_{i}})\cap S_{i}|\geq \funref{label_einbegreifen}(r,a)/b.$

	Now, let $\mathcal{L}=\{L_{1}, \ldots, L_{\funref{label_impercepbbly}(r,a)}\}$ be the set of the horizontal paths of $\breve{H},$
	where $L_{i}$ is the $i$-th horizontal path of $\breve{H}.$
	For every $j\in[b],$ let $\mathcal{L}_{j}=\bigcup_{i\in[z]}{L_{j+b(i-1)}}.$
	For every $i\in[a],$ let $y_{i}:=\arg\max_{j\in[b]}\{|V(\mathcal{L}_{j})\cap V(\mathcal{P}_{x_{i}})\cap S_{i}|\}.$
	Intuitively, we partition $\mathcal{L}$ into $b$ sets $\mathcal{L}_j,$ $j\in[b],$
	each one consisting of the $j$-th, $(j+b)$-th, $\ldots,$ $(j+(z-1)\cdot b)$-th horizontal path of $\breve{H},$ and
	$y_{i}$ is defined as the index $j$ maximizing the size of the  intersection of $V(\mathcal{L}_j)$ with $V(\mathcal{P}_{x_{i}})\cap S_{i}.$
	Observe that, since $|V(\mathcal{P}_{x_{i}})\cap S_{i}|\geq \funref{label_einbegreifen}(r,a)/b,$ again by the pigeonhole principle, $|V(\mathcal{L}_{y_{i}})\cap V(\mathcal{P}_{x_{i}})\cap S_{i}|\geq \funref{label_einbegreifen}(r,a)/b^{2}.$
	For every $i\in[a],$ let $Q_{i}:=V(\mathcal{P}_{x_{i}})\cap V(\mathcal{L}_{y_{i}}).$ See \autoref{label_uncultivated} for an illustration of the above.
	\begin{figure}[ht]
		\centering
		\includegraphics[width=7.5cm]{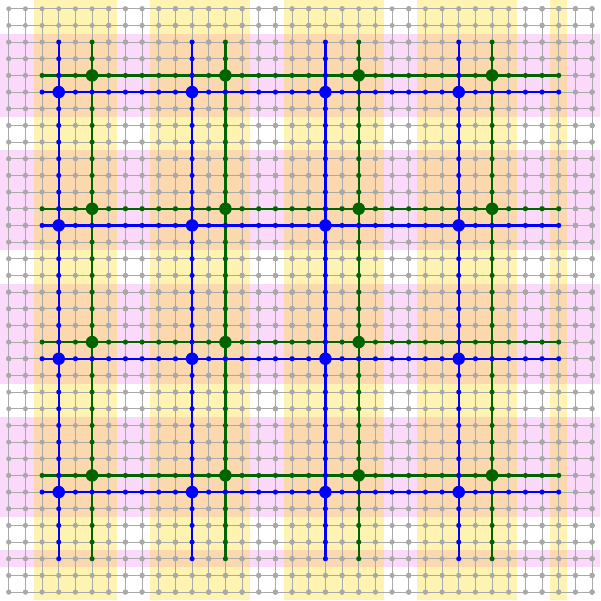}
		\caption{Illustration of a grid $H,$ the set $\mathcal{P}_{x_i}\cup\mathcal{L}_{y_i}$ (depicted in blue), and the set $\mathcal{P}_{x_{i'}}\cup\mathcal{L}_{y_{i'}}$ (depicted in green), with $i\neq i'.$ The grid $\breve{H}$ is the subgrid of $H$ obtained by removing the two first layers of $H.$ The set $Q_{i}$ consists of the blue vertices that are adjacent to four blue vertices. Yellow and pink regions contain the paths of $\hat{\mathcal{P}}$ and $\hat{\mathcal{L}},$ whose edges are contracted in order to obtain the grid $H'.$}
\label{label_uncultivated}
\end{figure}
For further intuition, observe that for every $i\in[a],$
$Q_{i}$ is a set of $z^{2}$ vertices in $\breve{H}$ and for every $u,v\in Q_{i}$ with $u\neq v,$
it holds that $\mathsf{dist}_{\breve{H}}(u,v)\geq b.$
Notice that since $b=\ell\cdot (a+1)+2,$ there is a set of $\ell$ consecutive integers
in $[2,b-1]$ that ``avoid''  every $x_i,$ $i\in[a],$ i.e., there is
a $t\in[2, b-\ell]$ such that  for every $i\in[a],$ $x_{i}\notin[t,t+\ell-1].$
Let $\bar{\mathcal{P}}:=\bigcup_{i\in[t,t+\ell-1]}\mathcal{P}_{i}.$
Also, there exists a $t'\in[2, b-\ell]$ such that  for every $i\in[a],$ $y_{i}\notin[t',t'+\ell-1].$
Let $\bar{\mathcal{L}}:=\bigcup_{i\in[t',t'+\ell-1]}\mathcal{L}_{i}.$
Intuitively, $\bar{\mathcal{P}}$ (resp. $\bar{\mathcal{L}}$) is the union
of $z$ sets of $\ell$ consecutive vertical (resp. horizontal) paths of $\breve{H}$
whose indices ``avoid'' $x_{i}$ (resp. $y_{i}$)
for every $i\in[a].$
In \autoref{label_uncultivated}, $\bar{\mathcal{P}}$ (resp. $\bar{\mathcal{L}}$)
is the set of the vertical (resp. horizontal) paths of $\breve{H}$
that are between yellow (resp. pink) regions.

We denote by $\hat{\mathcal{P}}$ (resp. $\hat{\mathcal{L}}$) the set of the vertical (resp. horizontal) paths of $H$ that contain the paths in $\mathcal{P}\setminus \bar{\mathcal{P}}$ (resp. $\mathcal{L}\setminus \bar{\mathcal{L}}$) as subpaths.
In \autoref{label_uncultivated},  $\hat{\mathcal{P}}$ (resp. $\hat{\mathcal{L}}$) is the set of the vertical (resp. horizontal) paths of $H$ that are drawn inside yellow (resp. pink) regions.
Let $H'$ be the graph obtained from $H$ after contracting every edge of a horizontal (resp. vertical) path of $H$ whose endpoints are in $\hat{\mathcal{P}}$ (resp. $\hat{\mathcal{L}}$).
In \autoref{label_uncultivated}, the graph $H'$ is obtained if we contract every ``horizontal'' edge inside a yellow region and every ``vertical'' edge inside a pink region.
Therefore, $H'$ is a contraction of $H$ and the fact that $H$ is an {$h$-grid, with $h\geq b\cdot z+2\ell$}, implies that $H'$ is a {$q$-grid, with $q\geq (\ell+1) \cdot z+2\ell+1$}.
To get some intuition on this, observe that, in \autoref{label_uncultivated}, the contraction of the ``horizontal'' (resp. ``vertical'') edges inside $(z+1)$-many yellow (resp. pink) regions results into $(z+1)$-many  vertical (resp. horizontal) paths of $H',$ while the rest of vertical or horizontal paths of $H,$ which are at least $\ell\cdot z +2\ell$ many, remain intact after these contractions.

We call a vertex of $H'$ \emph{heavy} if its model in $H$ is a
subset of $V(\cupall\hat{\mathcal{P}})\cap V(\cupall\hat{\mathcal{L}}).$
Notice that  the model of each heavy vertex of $H'$ contains exactly one vertex of each $Q_{i},i\in[a].$
Also, the distance in $H'$ between every two heavy vertices of $H'$ is more than $\ell$ in $H',$ since every path between heavy vertices contains at least $\ell$ vertices that are not heavy.
For every $i\in[a],$ we set $\tilde{V}_{i}$ to be the set of heavy vertices of $H'$ whose model in $H$ intersects $S_{i}$ and observe that for every $i\in[a],$ since
$|Q_{i}\cap S_{i}|\geq 2^{a-1}\cdot r^{2},$ it follows that $|\tilde{V}_{i}|\geq2^{a-1}\cdot r^{2}.$
For every $i\in[a]$, we set $V_i$ to be a set containing \textsl{exactly} $2^{a-1}\cdot r^{2}$ elements of $\tilde{V}_i$.
Let $\mathcal{V}:=\{V_{1},\ldots, V_{a}\}$ and observe that
for every $u,v\in  \cupall\mathcal{V},$ where $u\neq v,$ it holds that $\mathsf{dist}_{H'}(u,v) > \ell.$

To conclude the proof of Claim 1, observe that since $H'$ is a $q$-grid, with $q\geq (\ell+1) \cdot z+2\ell+1$ and $z=\lceil\sqrt{n}\rceil,$ it follows that $R'$ is a contraction of $H'$ (see \autoref{label_elegantemente}) and $\mathcal{V}$ is a collection of subsets of vertices of the middle horizontal path of $R'$ satisfying the claimed conditions.
Here, $n$ is asked to be at least $2^{a-1}\cdot r^2 \cdot a\cdot (\ell+1)$ in order to allow the middle horizontal path of $R'$ to host the $(2^{a-1}\cdot r^2, a,\ell)$-scattered set $\mathcal{V}$.
Claim 1 follows. \hfill$\diamond$
	\begin{figure}[H]
		\centering
		\includegraphics[width=5cm]{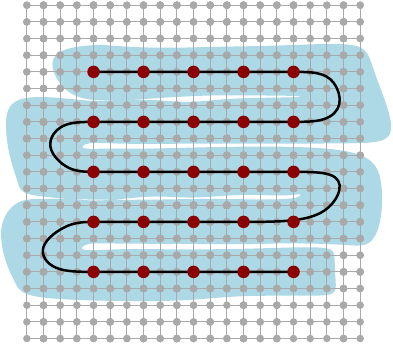}
		\caption{Illustration of the grid $H'$ and how $R'$ ``fits'' in $H'.$
			The red vertices are the heavy vertices of $H'$ and the black line represents the middle horizontal path of $R'.$}
		\label{label_elegantemente}
	\end{figure}

	Following Claim 1, let $\mathcal{V}$ be a collection of vertex sets satisfying the properties above.
	It is easy to see that the sets in $\mathcal{V}$ are not necessarily disjoint.
	For each vertex $v\in \cupall\mathcal{V},$ we define the \emph{trace} of $v$ in ${\mathcal{S}}$
	to be the set $I_{v}:=\{i\in [a]\mid \text{the model of }v\text{ in }H\text{ intersects }S_{i}\}.$
	We say that a set $U\subseteq \cupall\mathcal{V}$ is \emph{full with respect to ${\mathcal{S}}$} if $\bigcup_{v\in U}I_{v}=[a].$
	We now argue that the following claim holds.
	\medskip

	\noindent\emph{Claim 2:} There is some $y\in[a]$ and a collection ${\mathcal{C}}=\{C_{1}, \cdots, C_{y}\}$ of $y$ pairwise disjoint subsets of $\cupall\mathcal{V},$
	each of cardinality $r^{2},$ such that if we pick any vertex from every set $C_{j}$ then the resulting set is full with respect to ${\mathcal{S}}.$
	\smallskip

	\noindent\emph{Proof of Claim 2:}
	Notice that each $V_{i}$ can be partitioned into a collection $\mathcal{V}_{i}$ of $2^{a-1}$ subsets
	such that every two vertices are in the same subset if and only if they have the same trace.
	For every $i\in[a],$ we set $C_{i}:=\arg\max_{V\in\mathcal{V}_{i}}\{|V|\}$ and observe that there is a superset of $\{i\}$ that is the trace of all vertices in $C_i.$
	Since $|V_{i}|= 2^{a-1}\cdot r^{2},$ it follows that $|C_{i}|\geq r^{2}.$
	Moreover, we can assume that every $C_{i}$ contains exactly $r^{2}$ vertices (by removing extra vertices).
	Since the trace of the vertices of every $C_{i}$ contains $i,$  if we pick a vertex from every $C_{i}$ then the resulting set is full with respect to ${\mathcal{S}}.$
	Notice that, by the definition of the sets $C_i, i\in [a],$ for every $i,j\in[a],$ either $C_i \cap C_j = \emptyset$ or $C_{i}=C_{j}.$
	Therefore, we can obtain a collection ${\mathcal{C}}=\{C_i \mid i\in[a] \mbox{ and }\forall j\in [a]\setminus\{i\}, \ C_i \cap C_j = \emptyset\}$ as the desired one. Claim 2 follows. \hfill$\diamond$
	\medskip

	To conclude the proof, consider the graph $R'$ from Claim~1 and the collection ${\mathcal{C}}$ from Claim~2.
	Following Claims~1 and~2, ${\mathcal{C}}$ is a collection of subsets of vertices of the middle horizontal path of $R'$ that is also $(r^{2}, y, \ell)$-scattered in the middle horizontal path of $R'.$
	The lemma now follows by applying \autoref{label_automatisation}.
\end{proof}

Given a graph $G$ and a set $A\subseteq V(G),$ we say
that a graph $H$ is an \emph{$A$-fixed contraction} (resp. \emph{$A$-fixed minor}) of $G$ if $H$ can be obtained from $G$ (resp. a subgraph $G'$ of $G$ where $A\subseteq V(G')$) after contracting edges without endpoints  in $A.$
A graph $H$ is an \emph{$A$-apex partially triangulated $r$-grid} if  it can be obtained by an partially triangulated $r$-grid $\Gamma$ after adding a set $A$
of new vertices and some edges between the vertices of $A$ and $V(\Gamma).$
A \emph{complete $A$-apex partially triangulated $r$-grid} is a graph obtained by an $A$-apex partially triangulated $r$-grid by adding every edge between the vertices of $A$ and the vertices of the grid.

\begin{lemma}\label{label_satisfactory}
	There exist three functions $\funref{label_impercepbbly}, \funref{label_einbegreifen}: \mathbb{N}^{2}\to \mathbb{N},$ and $\funref{label_presupongamos}:\mathbb{N}\to \mathbb{N}$
	such that if $r,a\in \mathbb{N},$
	$H$ is an $A$-apex partially triangulated $h$-grid, where $A$ is a subset of $V(H)$ of size $a$ and $h\geq \funref{label_impercepbbly}(r,a)+2\cdot \funref{label_presupongamos}(r),$
	and
	each vertex $v\in A$ has at least $\funref{label_einbegreifen}(r,a)$ neighbors in the central
	$\funref{label_impercepbbly}(r,a)$-grid of $H\setminus A,$
	then $H$ contains as an $A$-fixed contraction a complete $A$-apex partially triangulated $r$-grid.
	Moreover,  $\funref{label_impercepbbly}(r,a)=\mathcal{O}(r^{4}\cdot 2^a),$
	$\funref{label_einbegreifen}(r,a)=\mathcal{O}(r^{6}\cdot 2^{a}),$ and $\funref{label_presupongamos}(r)=\mathcal{O}(r^{2}).$
\end{lemma}

Notice that by we can derive \autoref{label_satisfactory} from \autoref{label_emporteroient} by applying the latter for $H:=H\setminus A$ and $S_i, i\in[a]$ to be the set of neighbors of $v_i \in A$ in the central $\funref{label_impercepbbly}(r,a)$-grid of $H\setminus A.$

\subsection{The proof}
\label{label_reposadamente}
In this subsection we present some additional results that will allow us to prove \autoref{lemma_bidim_branch}, and we conclude with its proof.

The following easy observation intuitively states that every planar graph $H$ is a minor of a big enough grid, where the relationship between the size of the grid and $|V(H)|$ is linear (see e.g.,~\cite{RobertsonST94quic}).
\begin{proposition}\label{label_verbindungen}
	There exists a function $\newfun{label_lebeziatnikov}:\mathbb{N}\to\mathbb{N}$ such that every planar graph on $n$ vertices is a minor of the
	$\funref{label_lebeziatnikov}(n)$-grid. Moreover, $\funref{label_lebeziatnikov}(n)=\mathcal{O}(n).$
\end{proposition}

The next result intuitively states that given a graph $G$ and a set $A\subseteq V(G),$ a ``big enough'' (in terms of $a_{\mathcal{F}},s_{\mathcal{F}},$ and $k$) complete $A$-apex partially triangulated grid of $G$ is a structure that ``forces'' every set $S\subseteq V(G)$ of size at most $k$ such that $G\setminus S \in \mathbf{excl}(\mathcal{F})$ to intersect $A.$

\begin{lemma}\label{label_beschaffenheit}
	There exists a function $\newfun{label_understanding}: \mathbb{N}^3 \to \mathbb{N}$ such that if
	$\mathcal{F}$ is a finite family of graphs, $k\in\mathbb{N},$ and $G$ is a graph that  contains a complete $A$-apex partially triangulated $\funref{label_understanding}(a_{\mathcal{F}},s_{\mathcal{F}},k)$-grid $H$
as an $A$-fixed minor for some $A\subseteq V(G)$ with $|A|= a_{\mathcal{F}},$
then
for every set $S\subseteq V(G)$ that intersects the models of at most $k$ vertices of $H$ and such that $G\setminus S \in \mathbf{excl}(\mathcal{F}),$
it holds that  $S\cap A\neq\emptyset.$
Moreover $\funref{label_understanding}(a_{\mathcal{F}},s_{\mathcal{F}},k)=\mathcal{O}(\sqrt{(k+{a_{\mathcal{F}}}^2+1)\cdot {s_{\mathcal{F}}}}).$
\end{lemma}

\begin{proof}
For simplicity, we set $s=s_{\mathcal{F}}$ and $a=a_{\mathcal{F}}.$
Let $G$ be a graph, $m=\funref{label_lebeziatnikov}(s-a),$ where $\funref{label_lebeziatnikov}$ is the function of
\autoref{label_verbindungen}, and $r=\big\lceil \sqrt{(k+a^{2}+1)\cdot m}\big\rceil.$
We set $\funref{label_understanding}(a,s,k)=r$
and we notice that since $m=\funref{label_lebeziatnikov}(s-a) =\mathcal{O}(s),$ it holds that $\funref{label_understanding}(a,s,k) =\mathcal{O}(\sqrt{(k+a^2+1)\cdot s}).$

Observe that since $r=\big\lceil \sqrt{(k+a^{2}+1)\cdot m}\big\rceil,$ $V(H\setminus A)$ can be partitioned
into $(k+a^{2} + 1)$ vertex sets $V_1, \ldots, V_{k+a^{2} + 1}$ such that, for every $i\in[k+a^{2} + 1],$ the graph $H[V_i]$ is a partially triangulated
$m$-grid.
Let $\mathcal{H}=\big\{H[V_{i}\cup A]\mid i\in [k+a^{2}+1]\big\}$ and notice that every $R\in\mathcal{H}$ is a complete $A$-apex partially triangulated $m$-grid.
Our aim is to prove that  if $S$ is a subset of $V(G)$ that intersects the models of at most $k$ vertices of $H$ and such that $G\setminus S \in \mathbf{excl}(\mathcal{F}),$
then $S\cap A\neq \emptyset.$
Suppose towards a contradiction that  $S\cap A=\emptyset.$
Since $S$ intersects the models of at most $k$ vertices of $H$ and $|\mathcal{H}|=k+a^{2}+1,$ there is a collection $\mathcal{H}'\subseteq\mathcal{H}$
of size $a^{2}+1$ such that for every $R\in\mathcal{H}'$ and every $v\in V(R),$
$S$ does not intersect the model of $v$ in $G.$
This implies that $\cupall\mathcal{H}'\prem G\setminus S.$
Let $L$ be a graph in $\mathcal{F}$ whose apex number is $a.$
We arrive to a contradiction by proving that $L\prem\cupall\mathcal{H}'.$
To see why $L\prem\cupall\mathcal{H}',$ fix a graph $H'\in\mathcal{H}'$ and observe that, since $m=\funref{label_lebeziatnikov}(s-a),$ \autoref{label_verbindungen} implies that  every planar graph on $s-a$ vertices is a minor of $H'\setminus A$ and
every graph on $a$ vertices is a minor of $\cupall (\mathcal{H}'\setminus \{H'\}).$
The latter is a consequence of the fact that $|\mathcal{H}'\setminus \{H'\}|=a^{2}$
and every $R\in\mathcal{H}'\setminus \{H'\}$ is a complete $A$-apex $m$-grid; thus, for each pair of vertices in $A,$ we can find a path connecting them through some $H''.$
\end{proof}

We are now ready to prove \autoref{lemma_bidim_branch}.

\begin{proof}[Proof of \autoref{lemma_bidim_branch}]
	For simplicity, we set $a=a_{\mathcal{F}}$ and $s=s_{\mathcal{F}}.$
	We set
	\begin{align*}
		r: = & \ \funref{label_understanding}(a,s,k) =\mathcal{O}((a+\sqrt{k})\cdot \sqrt{s}),
		&	\funref{@prepossession}(a, s, k) := & \ \funref{label_impercepbbly}(r,a)+2\cdot \funref{label_presupongamos}(r) +2,\\
		\funref{@proclamation}(a,s,k) := & \ \funref{label_einbegreifen}(r,a), \mbox{ and } &
		\funref{@engrandecerla}(a,s,k):= & \ \funref{label_presupongamos}(r).
	\end{align*}
	Notice that $\funref{@prepossession}(a,s,k)=\mathcal{O}(2^a \cdot  s^{5/2} \cdot k^{5/2}),$
	$\funref{@proclamation}(a,s,k)=\mathcal{O}(2^a \cdot s^3 \cdot k^3),$ and $\funref{@engrandecerla}(a,s,k)=\mathcal{O}((a^2 +k)\cdot s).$
	Let $(W,\mathfrak{R})$ be a flatness pair of $G\setminus A$ of height $h,$ where $h$ is an odd integer with $h\geq \funref{@prepossession}(a, s, k),$ and let $\tilde{\mathcal{Q}}$ be a $(W,\mathfrak{R})$-canonical partition of $G\setminus A,$ such that if $A'$ is the set of vertices of $A$ that are adjacent in $G$ to at least $\funref{@proclamation}(a,s,k)$ $\funref{@engrandecerla}(a,s,k)$-internal bags of $\tilde{\mathcal{Q}},$ then $|A'| \geq a.$

	We contract every bag in $\tilde{\mathcal{Q}}$ to a vertex.
	Since $(W,\mathfrak{R})$ is a flatness pair, this results into a planar graph that is a partially triangulated $(h-2)$-grid $\bar{\Gamma}$ (whose vertices correspond to the internal bags of $\tilde{\mathcal{Q}}$) together with an extra vertex $u_\mathsf{ext}$ (which corresponds to the external bag of $\tilde{\mathcal{Q}}$) that is adjacent to all the vertices in the perimeter of $\bar{\Gamma}.$
	We contract an edge between $u_\mathsf{ext}$ and a vertex in the perimeter of $\bar{\Gamma}$
	and we denote by $\Gamma$ the obtained partially triangulated $(h-2)$-grid.
	Notice that $\Gamma$ is an  $A$-apex partially triangulated $(h-2)$-grid that is an $A$-fixed contraction of $G.$
	Moreover, observe that if a vertex $v\in A$ is adjacent, in $G,$ to an $\funref{@engrandecerla}(a,s,k)$-internal bag of $\tilde{\mathcal{Q}},$
	then, since $\funref{@engrandecerla}(a,s,k)=\funref{label_presupongamos}(r)$ and $h-2\geq  \funref{label_impercepbbly}(r,a)+2\cdot \funref{label_presupongamos}(r),$ it is also adjacent to a vertex in the central $\funref{label_impercepbbly}(r,a)$-grid of $\Gamma\setminus A.$
	Thus, each vertex in $A'$ has at least $\funref{@proclamation}(a,s,k)$ neighbors  in the central  $\funref{label_impercepbbly}(r,a)$-grid of $\Gamma\setminus A.$
	We remove extra vertices from $A'$ until $|A'|=a,$ and we set $\Gamma' = \Gamma\setminus (A\setminus A').$
	By \autoref{label_satisfactory} applied to $\Gamma'$ and $A',$ $\Gamma'$ contains  a complete $A'$-apex partially triangulated $r$-grid $R$ as an $A'$-fixed contraction.

	Observe that $R$ is also an $A'$-fixed minor of $G.$
	Since $r=\funref{label_understanding}(a,s,k),$ \autoref{label_beschaffenheit} implies that
for every set $S\subseteq V(G)$ that intersects at most $k$ internal bags of every $(W,\mathfrak{R})$-canonical partition of $G\setminus A,$
it holds that  $S\cap A'\neq\emptyset.$
\end{proof}

\section{Concluding remarks}\label{label_travaillerez}

In this paper we prove that, for every minor-closed graph class $\mathcal{G},$
there is a  4-fold exponential, in $k,$
bound on the order of the minor obstructions for the set of all graphs that are $k$-apices of $\mathcal{G},$ namely $\mathbf{obs}(\mathcal{A}_k (\mathcal{G})).$
Improving this bound is an open challenge. We believe that any such attempt  should radically overcome the current ``treewidth-based'' state of the art on bounding obstructions, dating back to the classic ideas of~\cite{AbrahamsonF93,AdlerGK08comp,CattellDDFL00onco,CourcelleDF97,DowneyF95survey,FellowsJ13fpti,FellowsL88nonc,FellowsL89anan,FellowsL94onse,Lagergren98,Lagergren91anup,LagergrenA91mini,RobertsonS86GMV,RobertsonS95XIII,Thomas90amen}.
Note that the results of Dinneen \cite{Dinneen97} give an exponential lower bound on the size of the set $\mathbf{obs}(\mathcal{A}_k (\mathcal{G})).$
This lower bound on the size of the obstruction set readily gives a polynomial lower bound on the order of the graphs in $\mathbf{obs}(\mathcal{A}_k (\mathcal{G})).$
\medskip

In the next paper of this series we deal with the algorithmic complexity of
recognizing $k$-apices of an arbitrary minor-closed graph class $\mathcal{G}.$
Namely, in~\cite{SauST21kapiII} (whose conference version is \cite{SauST20anfp})
we construct a  $2^{\mathsf{poly}(k)} \cdot n^3$-time algorithm for this problem,
which can be improved to one running in $2^{\mathsf{poly}(k)} \cdot n^2$-time
when $\mathcal{G}$ excludes some apex graph as a minor,
where $\mathsf{poly}(k)$ is a polynomial whose degree depends on the order of the obstructions for $\mathcal{G}.$
These algorithms are strongly based on the  combinatorial results of \autoref{label_villebrequin}, \autoref{label_reconquistasen}, and \autoref{label_nominalistic}.

\paragraph{Acknowledgements.}
We are all thankful to Archontia C. Giannopoulou for long discussions on topics of this paper. The third author is especially thankful to Prof. Michael Fellows for all the inspiration and knowledge that he generously offered.
We would also like to thank the anonymous reviewers for their valuable comments that improved the presentation of the paper.

\addcontentsline{toc}{section}{References}
\bibstyle{plainurl}

\end{document}